\documentclass[a4paper,11pt]{article}
\usepackage{subfig}
\usepackage{amsmath,indentfirst,amsfonts,dsfont,amsthm,enumerate,mathrsfs,color}
\usepackage{bm,multirow,booktabs}
\usepackage{amsopn,amssymb}
\usepackage{algorithm}               
\usepackage{algcompatible}             
\usepackage{rotating}
\usepackage[dvips]{epsfig,graphicx}
\usepackage{stmaryrd}
\usepackage{listings,url,verbatim}
\usepackage{lscape}
\usepackage{rotating,bm}

\newtheorem{definition}{Definition}[section]
\newtheorem{theorem}{Theorem}[section]
\newtheorem{lemma}{Lemma}[section]

\newtheorem{problem}{Problem}[section]
\newtheorem{remark}{Remark}[section]
\newtheorem{example}{Example}[section]

\usepackage{xcolor}
\usepackage{todonotes}

\title{\bf Randomized algorithms for computing the tensor train approximation and their applications}
\author{
Maolin Che\thanks{E-mail: mlche@gzu.edu.cn and chncml@outlook.com. School of Mathematics and Statistics, Guizhou University, Guiyang, 550025, Guizhou, P. R. of China and Center for Intelligent Multidimensional Data Analysis, Hong Science Park, Shatin, Hong Kong.}
\and
Yimin Wei\thanks{E-mail: ymwei@fudan.edu.cn and yimin.wei@gmail.com. School of Mathematical Sciences and
        Key Laboratory of Mathematics for Nonlinear   Sciences,  Fudan University, Shanghai, 200433, P. R. of China.}
\and
Hong Yan\thanks{Department of Electrical Engineering and Center for Intelligent Multidimensional Data Analysis, City University of Hong Kong, 83 Tat Chee Avenue, Kowloon, Hong Kong.}}
\date{\today}

\begin{document}
\date{}
\maketitle
\begin{abstract}
In this paper, we focus on the fixed TT-rank and precision problems of finding an approximation of the tensor train (TT) decomposition of a tensor. Note that the TT-SVD and TT-cross are two well-known algorithms for these two problems. Firstly, by combining the random projection technique with the power scheme, we obtain two types of randomized algorithms for the fixed TT-rank problem. Secondly, by using the non-asymptotic theory of sub-random Gaussian matrices, we derive the upper bounds of the proposed randomized algorithms. Thirdly, we deduce a new deterministic strategy to estimate the desired TT-rank with a given tolerance and another adaptive randomized algorithm that finds a low TT-rank representation satisfying a given tolerance, and is beneficial when the target TT-rank is not known in advance. We finally illustrate the accuracy of the proposed algorithms via some test tensors from synthetic and real databases. In particular, for the fixed TT-rank problem, the proposed algorithms can be several times faster than the TT-SVD, and the accuracy of the proposed algorithms and the TT-SVD are comparable for several test tensors.
  \bigskip

  \bigskip

\end{abstract}
\newpage
\section{Introduction}

\label{sec:introduction}
Tensor decomposition is a useful tool for analyzing big data in machine learning, signal processing, neuroscience, uncertainty quantification, and so on. There exist several types of tensor decomposition, such as CANDECOMP/PARAFAC (CP) decomposition \cite{carroll1970analysis}, Tucker decomposition \cite{tucker1966some}, Hierarchical Tucker (HT) decomposition \cite{grasedyck2010hierarchical} and tensor train (TT) decomposition \cite{oseledets2011tensor}. We now consider the computation of an approximation of the TT decomposition, which can be formulated as the fixed TT-rank problem and the fixed precision problem.

An approximation of the TT decomposition provides a useful model reduction and its storage cost increases linearly with the order of tensors. The approximation of the TT decomposition can be applied to some practical problems, such as the system of linear equations (cf. \cite{oseledets2012solution,kazeer2014low}), symmetric eigenvalue and singular value problems (cf. \cite{dolgov2014computation,kressner2014low-rank,lee2015estimating}), tensor completion (cf. \cite{bengua2017efficient,kressner2014low}), supervised tensor learning (cf. \cite{chen2019support,chen2022kernelized,kour2023efficient}), and some other fields (cf. \cite{dian2019learning,wang2019distributed,wang2020adtt,yin2021tt}).

We focus on randomized algorithms for the fixed-TT-rank problem of the approximation of TT decomposition. There is well-developed error analysis applicable to several classes of random matrices in randomized algorithms for the computation of low-rank matrix approximations (cf. \cite{drineas2016randnla,halko2011finding,mahoney2011randomized,woodruff2014sketching}). Recently, randomized algorithms were designed to compute the Tucker decomposition (cf. \cite{ahmadiasl2021randomized,che2019randomized,che2020the,che2021randomized,che2023efficient,kressner2017recompression,minster2020randomized,sun2020low-rank,zhou2014decomposition}) and the CP decomposition (cf. \cite{Battaglino2017a,biagioni2015randomized,erichson2017randomized,malik2019fast,vervliet2016a}).

For the fixed TT-rank and/or precision problem, by a successively applying projections, when using the singular value decomposition (SVD) of the matrices from the corresponding unfolding, Oseledets \cite{oseledets2010tt} obtained a quasi-best approximation of the TT decomposition and the related TT-rank. This algorithm is called TT-SVD and is summarized in Algorithm \ref{RTT:alg1}. Note that TT-cross \cite{savostyanov2011fast} is another well-known algorithm for the fixed-precision problem. By replacing the standard Gaussian vectors as the Kronecker product of the standard Gaussian vectors in the adaptive range finder \cite{halko2011finding}, Che and Wei \cite{che2019randomized} proposed an adaptive randomized algorithm for the fixed-precision problem. Alger {\it et al.} \cite{alger2020tensor} presented a randomized algorithm based on actions of the tensor as a vector-valued multilinear function. Hence, for a given $0<\epsilon<1$, {\bf the first objective} is to design several strategies to estimate the $\epsilon$-TT-rank and the corresponding approximation of the TT decomposition.

For a given TT-rank, Holtz {\it et al.} \cite{holtz2012the} considered two versions of a generalised alternating least square scheme for computing the TT decomposition: one is the alternating least squares (ALS) algorithm, and another is a modified approach of ALS, denoted by MALS. Recently, many researchers focus on the fixed TT-rank problem. For example, by generalizing the algorithm in \cite{halko2011finding} to the case of tensors, Huber {\it et al.} \cite{huber2017a} derived a randomized algorithm for the approximation of the TT decomposition. Li {\it et al.} \cite{li2022faster} proposed a new quasi-optimal fast TT decomposition algorithm for large-scale sparse tensors with proven correctness and the upper bound of computational complexity derived. Shi {\it et al.} \cite{shi2023parallel} derived the parallel TT decomposition, which starts with the analogy between TT-SVD and HOSVD. In this paper, {\bf the second objective} is to design efficient randomized algorithms for the fixed TT-rank problem of an approximation of the TT decomposition. These algorithms are based on the random projection and power scheme strategies.

The main contributions are listed as follows:
\begin{enumerate}
\item For a given $0<\epsilon<1$, we first derive a new strategy to obtain a desired $\epsilon$-TT-rank of a tensor. Numerical examples illustrate that for some test tensors, the $\epsilon$-TT-rank obtained from this strategy is between that obtained from TT-SVD and TT-cross.
\item We then propose an efficient adaptive randomized algorithm for the fixed precision problem, which is summarized in Algorithm \ref{RTT:algadapt}. From this adaptive randomized algorithm, for a given $0<\epsilon<1$, we can obtain an approximation of the TT decomposition and the related $\epsilon$-TT-rank.
\item Similar to the work in \cite{che2020the}, we obtain a three-stage algorithm for the fixed TT-rank problem, as shown in Algorithm \ref{RTT:alg3}. The random matrix $\mathbf{G}$ used in Algorithm \ref{RTT:alg3} can be chosen as a standard Gaussian matrix or the Khatri-Rao product of the standard Gaussian matrices. In Section \ref{randomizedTT:sec6}, we also consider other choices of the random matrix $\mathbf{G}$, such as the SpEmb matrix (sparse subspace embedding) \cite{aizenbud2016randomized,clarkson2013low}, the SDCT matrix (subsampled randomized discrete cosine transform) \cite{avron2010blendenpik} and the Kronecker product of standard Gaussian matrices. Note that the SRFT matrix (subsampled randomized discrete Fourier transform) \cite{ailon2009the,woolfe2008a} and the SRHT matrix (subsampled randomized Hadamard transform) \cite{boutsidis2013improved} are similar to the SDCT matrix.
\item Finally, to reduce the number of matrix-matrix multiplications and the cost to generate the random matrix $\mathbf{G}$ used in Algorithm \ref{RTT:alg3}, another efficient algorithm for the fixed-TT-rank problem is proposed in Algorithm \ref{RTT:alg2}. Note that Algorithm \ref{RTT:alg2} can be viewed as the generalization of the work in \cite{bjarkason2019pass} to the high-order tensor case.
\end{enumerate}
\subsection{Notations and organizations}
We now introduce some notations used in this paper. We use lower case letters (e.g. $x,u,v$) for scalars, lower case bold letters (e.g. $\mathbf{x},\mathbf{u},
\mathbf{v}$) for vectors, bold capital letters (e.g. $\mathbf{A},\mathbf{B},\mathbf{C}$) for matrices, and calligraphic letters $\mathcal{A},\mathcal{B},\mathcal{C},\dots$ for tensors. This notation is consistently used for the lower-order parts of a given structure. For example, the entry with the row index $i$ and the column index $j$ in a matrix ${\bf A}$, that is, $({\bf A})_{ij}$, is represented as $a_{ij}$ (also $(\mathbf{x})_i=x_i$ and $(\mathcal{A})_{i_1i_2\dots i_N}=a_{i_1i_2\dots i_N}$). We use $\mathbf{A}^\top$ to denote the transpose of $\mathbf{A}\in\mathbb{R}^{I\times J}$ and $\mathbf{I}_I$ to denote the identity matrix in $\mathbb{R}^{I\times I}$. We use ${\rm Tr}(\mathbf{A})$ to denote the trace of $\mathbf{A}\in\mathbb{R}^{I\times I}$, that is, ${\rm Tr}(\mathbf{A})=\mathbf{A}(1,1)+\mathbf{A}(2,2)+\dots+\mathbf{A}(I,I)$.

The generalization of vectors and matrices is named as a tensor, denoted by $\mathcal{A}\in\mathbb{R}^{I_1\times I_2\times \dots\times I_N}$. All entries of $\mathcal{A}$ are given by $\mathcal{A}(i_1,i_2,\dots, i_N)\in\mathbb{R}
$ with $i_n=1,2,\dots,I_n$ and $n=1,2,\dots,N.$ For given two matrices $\mathbf{A}\in\mathbb{R}^{I_1\times J_1}$ and $\mathbf{B}\in \mathbb{R}^{I_2\times J_2}$, its Kronecker product $\mathbf{A}\otimes \mathbf{B}\in\mathbb{R}^{I_1I_2\times J_1J_2}$ is given by
\begin{equation*}
    \mathbf{A}\otimes \mathbf{B}=
    \begin{pmatrix}
    \mathbf{A}(1,1)\mathbf{B}& \mathbf{A}(1,2)\mathbf{B}& \dots & \mathbf{A}(1,J_1)\mathbf{B}\\
    \mathbf{A}(2,1)\mathbf{B}& \mathbf{A}(2,2)\mathbf{B}& \dots & \mathbf{A}(2,J_1)\mathbf{B}\\
    \vdots & \vdots & \ddots & \vdots\\
    \mathbf{A}(I_1,1)\mathbf{B}& \mathbf{A}(I_1,2)\mathbf{B}& \dots & \mathbf{A}(I_1,J_1)\mathbf{B}\\
    \end{pmatrix}.
\end{equation*}
The Khatri-Rao product of $\mathbf{A}\in\mathbb{R}^{I_1\times J}$ and $\mathbf{B}\in \mathbb{R}^{I_2\times J}$ is denoted by $\mathbf{A}\odot\mathbf{B}\in\mathbb{R}^{I_1I_2\times J}$. In detail, the $j$-th column of $\mathbf{A}\odot\mathbf{B}$ is $\mathbf{A}(:,j)\otimes\mathbf{B}(:,j)$ with $j=1,2,\dots,J$. A matrix $\mathbf{Q}\in\mathbb{R}^{I\times K}$ with $I>K$ is orthonormal if $\mathbf{Q}^\top\mathbf{Q}=\mathbf{I}_K$.

The rest of this paper is organized as follows. Some basic definitions and results are introduced in Section \ref{randomizedTT:sec2}. In Section \ref{randomizedTT:sec3}, based on random projection, the power scheme and the SVD, we propose two randomized algorithms for the fixed TT-rank problem of an approximation of the TT decomposition. In Section \ref{randomizedTT:sec4}, based on the bounds for the singular values of the standard Gaussian matrix or the sub-Gaussian matrix with independent columns, we derive the upper bound of $\|\mathcal{A}-\mathcal{Q}_1\times_{2}^1\mathcal{Q}_2\times_{3}^1\dots\times_{3}^1\mathcal{Q}_N\|_F$, where all the $\mathbf{Q}_n$ are obtained from the proposed algorithms. In Section \ref{randomizedTT:sec5}, we deduce an adaptive randomized algorithm for the fixed precision problem of an approximation of the TT decomposition and analyze the accuracy of $\|\mathcal{A}-\mathcal{Q}_1\times_{2}^1\mathcal{Q}_2\times_{3}^1\dots\times_{3}^1\mathcal{Q}_N\|_F^2$ in floating point arithmetic. In Section \ref{randomizedTT:sec6}, some examples are used to illustrate the accuracy and efficiency of our algorithms. This paper is concluded in Section \ref{randomizedTT:sec7}.

\section{Preliminaries}
\label{randomizedTT:sec2}

For a given $n=1,2,\dots,N$, the tensor $\mathcal{C}=\mathcal{A}\times_{n}{\bf B}$ is the mode-$n$ product \cite{cichocki2009nonnegative,kolda2009tensor} of $\mathcal{A}\in \mathbb{R}^{I_1\times I_2\times \dots \times I_N}$ by ${\bf B}\in \mathbb{R}^{J_n\times I_n}$ and its entries are given by
\begin{equation*}
\begin{split}
&\mathcal{C}(i_{1},\dots, i_{n-1},j,i_{n+1},\dots, i_{N})\\
&=\sum_{i_{n}=1}^{I_n}\mathcal{A}(i_{1},\dots, i_{n-1},i_{n},i_{n+1},\dots, i_{N})\mathbf{B}(j,i_{n}).
\end{split}
\end{equation*}

The mode-$(n,m)$ product \cite{grasedyck2013a,oseledets2010tt}, denoted by $\mathcal{C}=\mathcal{A}\times_{n}^m\mathcal{B}$, of $\mathcal{A}\in \mathbb{R}^{I_1\times I_2\times \dots \times I_N}$ and $\mathcal{B}\in \mathbb{R}^{J_1\times J_2\times \dots\times J_M}$ with $I_n=J_m$ is defined by
\begin{equation*}
\begin{split}
&\mathcal{C}(i_1,\dots, i_{n-1},i_{n+1},\dots, i_N,j_1,\dots, j_{m-1},i_{m+1},\dots, j_N)\\
&=\sum_{i_n=1}^{I_n}\mathcal{A}(i_1,\dots, i_{n-1},i_n,i_{n+1},\dots, i_N)\cdot\mathcal{B}(j_1,\dots, j_{m-1},i_n,j_{m+1},\dots, j_M).
\end{split}
\end{equation*}

The inner product of $\mathcal{A}\in \mathbb{R}^{I_1\times I_2\times \dots \times I_N}$ and $\mathcal{B}\in \mathbb{R}^{I_1\times I_2\times \dots \times I_N}$ is defined as \cite{kolda2009tensor}
$$\langle\mathcal{A},\mathcal{B}\rangle=\sum_{i_{1}=1}^{I_1} \sum_{i_{2}=1}^{I_2} \dots \sum_{i_{N}=1}^{I_N}\mathcal{A}(i_{1},i_{2},\dots, i_{N})\mathcal{B}(i_{1},i_{2},\dots, i_{N}),$$
and the {\it Frobenius norm} of a tensor $\mathcal{A}$ is given by  $\|\mathcal{A}\|_{F}=\sqrt{\langle\mathcal{A},\mathcal{A}\rangle}$.

As shown in \cite{kolda2009tensor}, ${\bf A}_{(n)}$ is the mode-$n$ unfolding matrix of $\mathcal{A}\in \mathbb{R}^{I_1\times I_2\times \dots \times I_N}$ and is obtained by arranging the mode-$n$ fibers into columns of a matrix. The $(i_n,j)$ element of ${\bf A}_{(n)}$ is equivalent to $\mathcal{A}(i_{1},i_{2},\dots, i_{N})$, where
\begin{equation*}
\begin{split}
j=&i_1+(i_{2}-1)I_1+\dots+(i_{n-1}-1)I_1\dots I_{n-2}\\
&+(i_{n+1}-1)I_1\dots I_{n-1}\\
&+\dots+(i_N-1)I_1\dots I_{n-1}I_{n+1}\dots I_{N-1}.
\end{split}
\end{equation*}

Another form for unfolding a tensor $\mathcal{A}\in \mathbb{R}^{I_1\times I_2\times \dots \times I_N}$ into a matrix is denoted by ${\bf A}_{([n])}$ \cite{oseledets2011tensor}, whose $(i,j)$-element is equivalent to $\mathcal{A}(i_{1},i_{2},\dots, i_{N})$, with
\begin{equation*}
\begin{split}
i&=i_1+(i_2-1)I_1+\dots +(i_n-1)I_1\dots I_{n-1},\\
j&=i_{n+1}+(i_{n+2}-1)I_{n+1}+\dots +(i_N-1)I_{n+1}\dots I_{N-1}.
\end{split}
\end{equation*}

We use the MATLAB command {\it reshape} to reorganize elements of a tensor. In detail, for $m=1,2,\dots,M$ and $n=1,2,\dots,N$, let $M$ positive integers $J_m$ and $N$ positive integers $I_n$ satisfy $J_1J_2\dots J_M=I_1I_2\dots I_N$. If $\mathcal{A}\in\mathbb{R}^{I_1\times I_2\times \dots\times I_N}$, then ${\it reshape}(\mathcal{A},[J_1,J_2,\dots,J_M])$ returns a tensor in $\mathbb{R}^{J_1\times J_2\times \dots\times J_M}$ by stacking entries according to their multi-index.

\subsection{TT decomposition}
\label{randomizedTT:sec2:1}
The TT decomposition (cf. \cite{grasedyck2013a,oseledets2011tensor}) of $\mathcal{A}\in \mathbb{R}^{I_1\times I_2\times \dots \times I_N}$ is given by
\begin{equation*}
\mathcal{A}(i_{1},\dots, i_{N})\approx\sum_{j_1=1}^{\mu_1}\dots\sum_{j_{N-1}=1}^{\mu_{N-1}}
\mathcal{Q}_1(i_1,j_1)\dots \mathcal{Q}_N(j_{N-1},i_N),
\end{equation*}
where $\mathcal{Q}_1(i_1)=\mathcal{Q}_1(i_1,:)$, $\mathcal{Q}_n(i_n)=\mathcal{Q}_n(:,i_n,:)\ (n=2,3,\dots,N-1)$ and $\mathcal{Q}_N(i_N)=\mathcal{Q}_n(:,i_N)\in\mathbb{R}$. The TT decomposition of $\mathcal{A}$ can be also represented as
\begin{equation*}
    \begin{split}
  \mathcal{A}&\approx(((\mathcal{Q}_1\times_{2}^1\mathcal{Q}_2)\times_{3}^1)
\dots\times_{N-1}^1\mathcal{Q}_{N-1})\times_{N}^1\mathcal{Q}_N\\
&=\mathcal{Q}_1\times_{2}^1(\mathcal{Q}_2\times_{3}^1
\dots\times_{3}^1(\mathcal{Q}_{N-1}\times_{3}^1\mathcal{Q}_N)).
    \end{split}
\end{equation*}
For clarity, in the rest, the term $\mathcal{Q}_1\times_{2}^1(\mathcal{Q}_2\times_{3}^1
\dots\times_{3}^1(\mathcal{Q}_{N-1}\times_{3}^1\mathcal{Q}_N))$ is simplified as $\mathcal{Q}_1\times_{2}^1\mathcal{Q}_2\times_{3}^1
\dots\times_{3}^1\mathcal{Q}_{N-1}\times_{3}^1\mathcal{Q}_N$.

The TT-rank of $\mathcal{A}$ is defined as $$\{R_1,R_2,\dots,R_{N-1}\}=\{{\rm rank}(\mathbf{A}_{([1])}),{\rm rank}(\mathbf{A}_{([2])}),\dots,{\rm rank}(\mathbf{A}_{([N-1])})\}.$$

Based on the TT decomposition, several other decomposition have been suggested, such as the QTT decomposition {\rm \cite{khoromskij2011o(dlogN)}}, the QTT-Tucker decomposition and {\rm \cite{dolgov2013two-level}} the Hierarchical Tucker decomposition {\rm \cite{ballani2013black,grasedyck2010hierarchical,hackbusch2009a}}.

For a given tolerance $0<\epsilon<1$, an approximation of the TT decomposition of $\mathcal{A}\in \mathbb{R}^{I_1\times I_2\times \dots \times I_N}$ and its corresponding TT-rank are obtained from TT-SVD (see \cite[Algorithm 1]{oseledets2011tensor}). We call the corresponding TT-rank as {\it an $\epsilon$-TT-rank} of $\mathcal{A}$. Note that this definition is similar to the numerical rank of a matrix.

\begin{algorithm}
\caption{TT-SVD \cite[Algorithm 1]{oseledets2011tensor}}
\label{RTT:alg1}
\begin{algorithmic}[1]
\STATEx {\bf Input:} $\mathcal{A}\in \mathbb{R}^{I_1\times I_2\times \dots \times I_N}$, and $0<\epsilon<1$.
\STATEx {\bf Output:} An approximation of the TT decomposition of $\mathcal{A}$: $\mathcal{B}=\mathcal{Q}_1\times_2^1\mathcal{Q}_2\times_3^1\dots\times_3^1\mathcal{Q}_N$ and $(N-1)$ positive integers $\{\mu_1,\mu_2,\dots,\mu_{N-1}\}$.
\STATE Let $\mu_0=1$ and compute $\delta=\frac{\epsilon}{\sqrt{N-1}}\|\mathcal{A}\|_F$.
\STATE Set $\mathbf{A}_0=\mathbf{A}_{(1)}$.
\FOR {$n=1,2,\ldots,N-1$}
\STATE Reshape $\mathbf{A}_{n-1}$ as a matrix $\mathbf{A}_n\in\mathbb{R}^{\mu_{n-1}I_n\times \prod_{m=n+1}^NI_m}$.
\STATE Compute a low rank approximation ${\bf U}{\bf \Sigma} {\bf V}^\top$ of ${\bf A}_n$ such that $\|{\bf A}_n-{\bf U}{\bf \Sigma} {\bf V}^\top\|_F\leq \delta$ and $\mu_n={\rm rank}({\bf U})$.
\STATE Reshape $\mathbf{U}$ as a tensor  $\mathcal{Q}_n\in\mathbb{R}^{\mu_{n-1}\times I_n\times \mu_n}$ and update ${\bf A}_{n}={\bf \Sigma}{\bf V}^\top$.
\ENDFOR
\STATE Reshape $\mathbf{A}_{N-1}$ as a tensor  $\mathcal{Q}_N\in\mathbb{R}^{\mu_{N-1}\times I_N}$.
\STATE Form $\mathcal{B}=\mathcal{Q}_1\times_2^1\mathcal{Q}_2\times_3^1\dots\times_3^1\mathcal{Q}_N$.
\end{algorithmic}
\end{algorithm}

Since each step of Algorithm {\rm \ref{RTT:alg1}} yields a TT-core, Algorithm {\rm \ref{RTT:alg1}} is a sequential algorithm. Meanwhile, Daas {\it et al.} \cite{daas2022parallel} presented efficient parallel algorithms for performing mathematical operations for low-rank tensors represented in the TT format, such as addition, elementwise multiplication, computing norms and inner products, orthonormalization, and rounding (rank truncation). Shi {\it et al.} \cite{shi2023parallel} proposed parallelizable algorithms for computing an approximation of the TT decomposition and provided theoretical guarantees of accuracy, parallelization methods and scaling analysis. R\"{o}hrig-Z\"{o}llner {\it et al.} \cite{rohrig2022performance} focused on the computation of the TT decomposition on current multicore CPUs.

We now review the computational complexity of Algorithm \ref{RTT:alg1}. When we set $I_n=I$ with $n=1,2,\dots,N$ and $\mu_n=\mu$ with $n=1,2,\dots,N-1$, TT-SVD requires
\begin{equation*}
    O(I^N)+O(I^N\mu)+(N-2)O(I\mu^2)+\sum_{n=1}^{N-1}I^n\mu+\sum_{n=2}^{N-1}(O(I^n\mu)+O(I^n\mu^2))
\end{equation*}
operations for obtaining an approximation of the TT decomposition of $\mathcal{A}$.

\section{Randomized algorithms for the fixed-TT-rank problem}
\label{randomizedTT:sec3}
From the work in \cite{che2019randomized}, when the desired TT-rank of $\mathcal{A}\in\mathbb{R}^{I_1\times I_2\times \dots\times I_N}$ is given by $\{\mu_1,\mu_2,\dots,\mu_{N-1}\}$, the fixed-TT-rank problem for the TT approximation is summarized as follows.
\begin{problem}
\label{RTT:prob2}
    Suppose that $\mu_n<\min\{\prod_{k=1}^nI_k,\prod_{k=n+1}^NI_{k}\}$ with $n=1,2,\dots,N-1$. For a  tensor $\mathcal{A}\in \mathbb{R}^{I_1\times I_2\times \dots \times I_N}$, we want to find $N$ tensors $\mathcal{Q}_n\in \mathbb{R}^{\mu_{n-1}\times I_n\times \mu_{n}}$ with $\mu_0=\mu_N=1$ to solve the following optimization problem
    \begin{equation*}
        \min_{\mathcal{G}_1,\dots,\mathcal{G}_N}
        \|\mathcal{A}-\mathcal{G}_1\times_{3}^1\mathcal{G}_2\times_{3}^1\dots\times_{3}^1\mathcal{G}_N\|_F,
    \end{equation*}
    where for $n=1,2,\dots,N-1$, the core tensors $\mathcal{G}_n$ satisfy
    \begin{equation*}
        \mathbf{G}_n^\top\mathbf{G}_n=\mathbf{I}_{\mu_n},\quad \mathbf{G}_n={\rm reshape}(\mathcal{G}_n,[\mu_{n-1}I_n,\mu_n]).
    \end{equation*}
\end{problem}
Suppose that $\{\mathcal{Q}_1,\mathcal{Q}_2,\dots,\mathcal{Q}_N\}$ is a solution of Problem \ref{RTT:prob2}. Let $\mathbf{Q}_1=\mathcal{Q}_1$ and $\mathbf{Q}_n={\rm reshape}(\mathcal{Q}_n,[\mu_{n-1}I_n,\mu_n])$ with $n=2,3,\dots,N-1$. Let $\mathcal{A}_0=\mathcal{A}$ and define
\begin{equation*}
    \begin{cases}
        \mathcal{A}_{1}=\mathcal{Q}_{1}\times_{1}^{1}\mathcal{A}_{0}
        \in\mathbb{R}^{\mu_{1}\times I_{2}\times\dots\times I_N},\\
        \mathcal{A}_{n}=\mathcal{Q}_{n}\times_{1,2}^{1,2}\mathcal{A}_{n-1}\in\mathbb{R}^{\mu_{n}\times I_{n+1}\times\dots\times I_N},
    \end{cases}
\end{equation*}
with $n=2,3,\dots,N-1$.

Let $\mathbf{A}_n={\rm reshape}(\mathcal{A}_{n-1},[\mu_{n-1}I_{n},I_{n+1}\dots I_N])$ with $n=1,2,\dots,N-1$. Then we have
\begin{equation*}
    \begin{split}
        &\mathcal{A}-\mathcal{Q}_1\times_{2}^1\mathcal{Q}_2\times_{3}^1\dots\times_{3}^1\mathcal{Q}_N\\
        &=\mathcal{A}-\mathcal{Q}_1\times_{2}^1\mathcal{A}_1+\mathcal{Q}_1\times_{2}^1\mathcal{A}_1-\mathcal{Q}_1\times_{2}^1\mathcal{Q}_2\times_3^1\mathcal{A}_2\\
        &+\dots+\mathcal{Q}_1\times_2^1\mathcal{Q}_2\times_3^1\dots\times_3^1
        \mathcal{Q}_{N-2}\times_3^1\mathcal{A}_{N-2}\\
        &-\mathcal{Q}_1\times_2^1\mathcal{Q}_2\times_3^1\dots\times_3^1\mathcal{Q}_{N-1}\times_3^1\mathcal{A}_{N-1},
    \end{split}
\end{equation*}
with $\mathcal{Q}_N=\mathcal{A}_{N-1}$. Now we have
\begin{equation*}
    \begin{split}
        &\|\mathcal{A}-\mathcal{Q}_1\times_{2}^1\mathcal{Q}_2\times_{3}^1\dots\times_{3}^1\mathcal{Q}_N\|_F\leq\|\mathcal{A}-\mathcal{Q}_1\times_{2}^1\mathcal{A}_1\|_F\\
        &+\|\mathcal{Q}_1\times_{2}^1\mathcal{A}_1-\mathcal{Q}_1\times_{2}^1\mathcal{Q}_2\times_3^1\mathcal{A}_2\|_F\\
        &+\dots+\|\mathcal{Q}_1\times_2^1\mathcal{Q}_2\times_3^1\dots
        \times_3^1\mathcal{Q}_{N-2}\times_3^1\mathcal{A}_{N-2}\\
        &-\mathcal{Q}_1\times_2^1\mathcal{Q}_2\times_3^1\dots\times_3^1\mathcal{Q}_{N-2}\times_3^1\mathcal{Q}_{N-1}\times_3^1\mathcal{A}_{N-1}\|_F.
    \end{split}
\end{equation*}
Note that we have
\begin{equation*}
    \begin{split}
        \|\mathcal{A}-\mathcal{Q}_1\times_{2}^1\mathcal{A}_1\|_F&=\|\mathcal{A}-\mathcal{Q}_1\times_{2}^1(\mathcal{Q}_1\times_{1}^1\mathcal{A})\|_F\\
        &=\|\mathcal{A}-(\mathcal{Q}_1\times_{2}^2\mathcal{Q}_1)\times_{1}^1\mathcal{A}\|_F\\
        &=\|\mathcal{A}\times_1(\mathbf{I}_{I_1}-\mathbf{Q}_1\mathbf{Q}_1^\top)\|_F\\
        &=\|(\mathbf{I}_{I_1}-\mathbf{Q}_1\mathbf{Q}_1^\top)\mathbf{A}_{1}\|_F,
    \end{split}
\end{equation*}
and
\begin{equation*}
    \begin{split}
        \|\mathcal{Q}_1\times_{2}^1\mathcal{A}_1-\mathcal{Q}_1\times_{2}^1\mathcal{Q}_2\times_3^1\mathcal{A}_2\|_F&=\|\mathcal{Q}_1\times_{2}^1\mathcal{A}_1-\mathcal{Q}_1\times_{2}^1\mathcal{Q}_2\times_3^1(\mathcal{Q}_2\times_{1,2}^{1,2}\mathcal{A}_1)\|_F\\
        &=\|\mathcal{Q}_1\times_{2}^1\mathcal{A}_1-\mathcal{Q}_1\times_{2}^1((\mathcal{Q}_2\times_3^3\mathcal{Q}_2)\times_{1,2}^{1,2}\mathcal{A}_1)\|_F\\
        &\leq \|\mathcal{A}_1-(\mathcal{Q}_2\times_3^3\mathcal{Q}_2)\times_{1,2}^{1,2}\mathcal{A}_1\|_F\\
        &= \|(\mathbf{I}_{\mu_1I_2}-\mathbf{Q}_2\mathbf{Q}_2^\top)\mathbf{A}_2\|_F.
    \end{split}
\end{equation*}
Let $\mathcal{B}_{n-1}=\mathcal{A}_{n-1}-\mathcal{Q}_{n}\times_3^1\mathcal{A}_{n}$, $\mathcal{B}_{k}=\mathcal{Q}_{k}\times_{3}^1\mathcal{B}_{k+1}$ with $k=n-2,\dots,3,2$, and $\mathcal{B}_1=\mathcal{Q}_1\times_2^1\mathcal{B}_2$. It is clear to see that $\mathcal{B}_1=\mathcal{Q}_1\times_2^1\mathcal{Q}_2\times_{3}^1\dots\times_3^1\mathcal{Q}_{n-1}\times_3^1(\mathcal{A}_{n-1}-\mathcal{Q}_{n}\times_3^1\mathcal{A}_{n})$. Then
\begin{equation*}
    \begin{split}
        \|\mathcal{B}_1\|_F&=\|\mathcal{Q}_1\times_2^1\mathcal{B}_2\|_F=\|\mathbf{Q}_1{\rm reshape}(\mathcal{B}_2,[\mu_1,I_2\dots I_N])\|_F\\
        &\leq \|{\rm reshape}(\mathcal{B}_2,[\mu_1,I_2\dots I_N])\|_F=\|\mathcal{B}_2\|_F\\
        &\leq\dots\leq\|\mathcal{B}_{n-1}\|_F= \|\mathcal{A}_{n-1}-\mathcal{Q}_{n}\times_3^1\mathcal{A}_{n}\|_F\\
        &=\|\mathcal{A}_{n-1}-(\mathcal{Q}_{n}\times_{3}^3\mathcal{Q}_n)\times_{1,2}^{1,2}\mathcal{A}_{n-1}\|_F\\
        &=\|(\mathbf{I}_{\mu_{n-1}I_n}-\mathbf{Q}_n\mathbf{Q}_n^\top)\mathbf{A}_{n}\|_F.
    \end{split}
\end{equation*}
with $n=3,5,\dots,N-1$. Overall, then we have
\begin{equation}
\label{RTT:eqnadd1}
    \begin{split}
        &\|\mathcal{A}-\mathcal{Q}_1\times_{2}^1\mathcal{Q}_2\times_{3}^1\dots\times_{3}^1\mathcal{Q}_N\|_F\leq
        \sum_{n=1}^{N-1}\|\mathbf{A}_{n}-\mathbf{Q}_{n}\mathbf{Q}_{n}^\top\mathbf{A}_{n}\|_F.
    \end{split}
\end{equation}

For the case of $N=4$, the overall process is illustrated in Figure \ref{TT-SVD:figadd1}.

\begin{figure}[htb]
	\centering
	\includegraphics[scale=0.15]{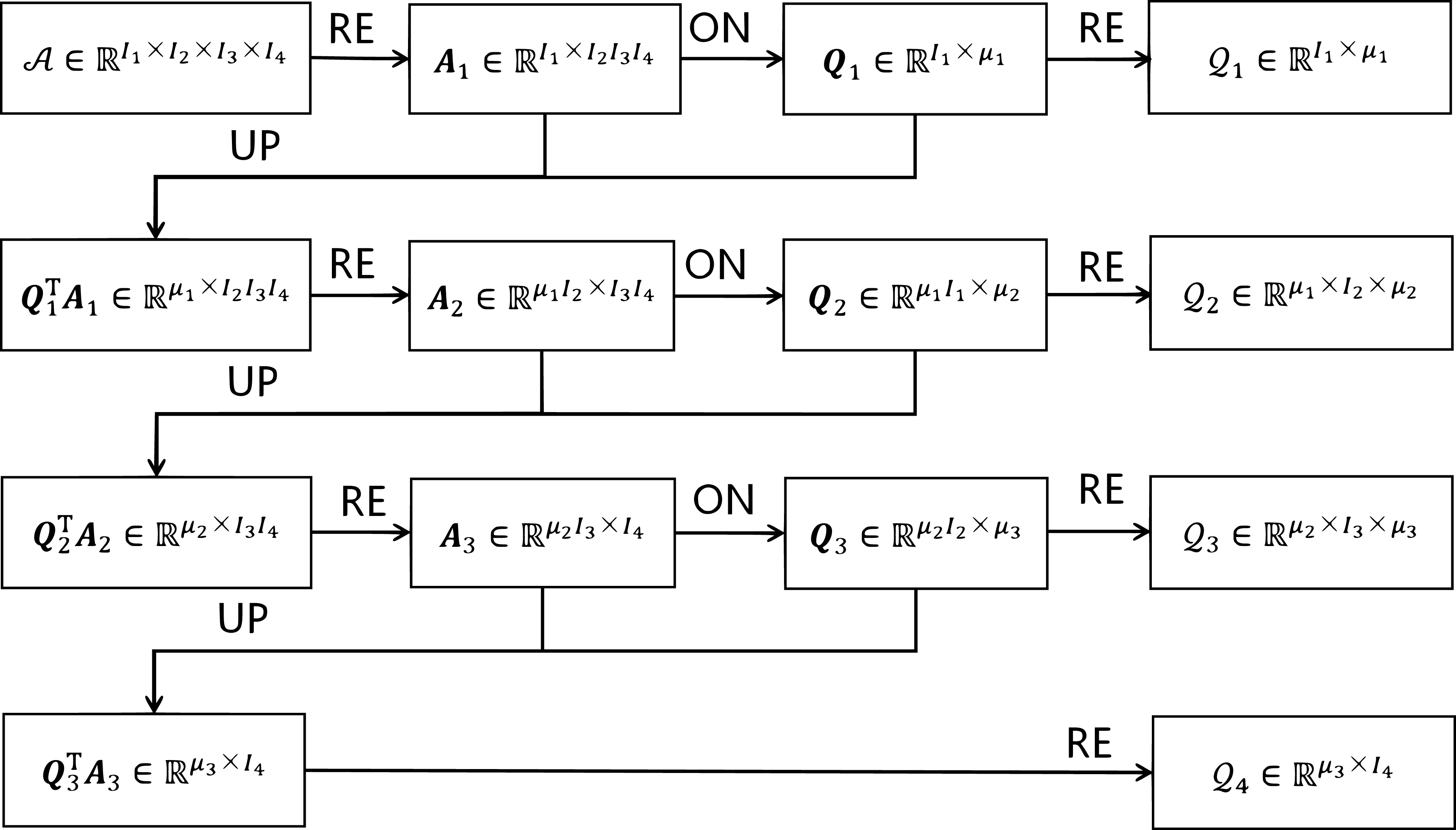}\\
	\caption{Illustration for the way to find an approximate solution for Problem \ref{RTT:prob2} with $N=4$. In the diagram, RE, ON and UP means reshaping, an orthonormal basis, and updating, respectively.}\label{TT-SVD:figadd1}
\end{figure}

In Figure \ref{TT-SVD:figadd1}, ``RE'' represents the reshaping operation between tensors and matrices, ``ON'' represents an orthonormal basis to approximate the column space of any matrix with given rank parameter and ``UP'' means to update the temporary matrix. It is clear to see that
\begin{equation*}
    \begin{split}
        \mathbf{Q}_n^\top\mathbf{A}_n&=\mathbf{Q}_n^\top{\rm reshape}(\mathcal{A}_{n-1},[\mu_{n-1}I_n,I_{n+1}\dots I_N])\\
        &={\rm reshape}(\mathcal{A}_{n},[\mu_n,I_{n+1}\dots I_N]).
    \end{split}
\end{equation*}

An approximate solution for Problem \ref{RTT:prob2} can be obtained by solving the following problem.
\begin{problem}
\label{RTT:prob3}
    Suppose that $\mathcal{A}\in \mathbb{R}^{I_1\times I_2\times \dots \times I_N}$. We solve the following $(N-1)$ optimization problems.
    \begin{enumerate}
        \item[{\rm (a)}] For the case of $n=1$, when $\mu_1\leq\min\{I_1,I_2\dots I_N\}$, the objective is to find an orthonormal matrix ${\bf Q}_1\in \mathbb{R}^{I_1\times \mu_1}$ such that
        \begin{equation*}
            {\bf Q}_1={\rm argmin}_{\mathbf{U}_1}\|\mathbf{A}_1-\mathbf{U}_1\mathbf{U}_1^{\top}\mathbf{A}_1\|_F
        \end{equation*}
        where $\mathbf{U}_1\in\mathbb{R}^{I_1\times \mu_1}$ is orthonormal.
        \item[{\rm (b)}] For the case of $n=2,3,\dots,N-1$, when $\mu_n< \min\{\mu_{n-1}I_n,I_{n+1}\dots I_N\}$, the objective is to find an orthonormal matrix ${\bf Q}_n\in \mathbb{R}^{\mu_{n-1}I_n\times \mu_n}$ such that
        \begin{equation*}
            \begin{split}
                & \mathbf{Q}_n={\rm argmin}_{\mathbf{U}_n}\|\mathbf{A}_n-\mathbf{U}_n\mathbf{U}_n^{\top}\mathbf{A}_n\|_F,
            \end{split}
        \end{equation*}
        where $\mathbf{U}_n\in\mathbb{R}^{\mu_{n-1}I_n\times \mu_n}$ is orthonormal.
	\end{enumerate}
\end{problem}
By setting $\mathcal{Q}_1=\mathbf{Q}_1$, $\mathcal{Q}_n={\rm reshape}(\mathbf{Q}_n,[\mu_{n-1},I_n,\mu_n])$ with $n=2,3,\dots,N-1$ and $\mathcal{Q}_N=\mathcal{A}_{N-1}$, then the $N$-tuple $\{\mathcal{Q}_1,\mathcal{Q}_2,\dots,\mathcal{Q}_N\}$ is an approximate solution for Problem \ref{RTT:prob2}. According to the above descriptions, a solution for Problem \ref{RTT:prob2} can be obtained by solving Problem \ref{RTT:prob3}, with $n$ from $1$ to $N-1$ inductively.

\subsection{Framework for the first algorithm}

It is worthy noting that each subproblem in Problem \ref{RTT:prob3} is the low rank matrix approximation. In spirit of randomized NLA (numerical linear algebra), the framework of the first randomized algorithm is given in Algorithm \ref{RTT:alg3}.

\begin{algorithm}[htb]
\caption{The first randomized algorithm for an approximation of the TT decomposition with a given TT-rank}
\label{RTT:alg3}
\begin{algorithmic}[1]
\STATEx {\bf Input:} A tensor $\mathcal{A}\in \mathbb{R}^{I_1\times I_2\times \dots \times I_N}$, the desired TT-rank $\{\mu_1,\mu_2,\dots,\mu_{N-1}\}$, an oversampling parameter $R$ and the number of subspace iterations $q\geq 0$.
\STATEx {\bf output:} An approximation of the TT decomposition of $\mathcal{A}$: $\mathcal{B}=\mathcal{Q}_1\times_2^1\mathcal{Q}_2\times_3^1\dots\times_3^1\mathcal{Q}_N$.
\STATE Set $\mu_0=1$, $\mathbf{Q}_0\in\mathbb{R}^{I_1\times I_1}$ being the identity matrix and $\mathbf{A}_0=\mathbf{A}_{(1)}$.
\FOR {$n=1,2,\dots,N-1$}
\STATE Reshape $\mathbf{Q}_{n-1}^\top\mathbf{A}_{n-1}$ as $\mathbf{A}_n\in\mathbb{R}^{\mu_{n-1}I_n\times \prod_{m=n+1}^NI_m}$.
\STATE Generate a matrix ${\bf \Omega}_{n}\in\mathbb{R}^{\prod_{m=n+1}^NI_m\times (\mu_n+R)}$.
\STATE Compute ${\bf Z}_n=(\mathbf{A}_n\mathbf{A}_n^\top)^q\mathbf{A}_n{\bf \Omega}_n$.
\STATE Find an orthonormal matrix $\mathbf{Q}_n\in\mathbb{R}^{\mu_{n-1}I_n\times \mu_n}$ such that there exists $\mathbf{S}_n\in\mathbb{R}^{\mu_n\times (\mu_n+K)}$ for which
\begin{equation*}
    \begin{split}
        \|\mathbf{Q}_n\mathbf{S}_n-\mathbf{Z}_n\|_F=\Delta_{\mu_n+1}(\mathbf{Z}_n)=\left(\sum_{k=\mu_n+1}^{\min\{\mu_{n-1}I_n,\mu_n+R\}}\sigma_k(\mathbf{Z}_n)^2\right)^{1/2},
    \end{split}
\end{equation*}
where $\sigma_k(\mathbf{Z}_n)$ is the $k$-th singular value of $\mathbf{Z}_n$.
\STATE Reshape ${\bf Q}_n$ as a tensor $\mathcal{Q}_n\in\mathbb{R}^{\mu_{n-1}\times I_n\times \mu_n}$.
\ENDFOR
\STATE Reshape $\mathbf{Q}_{N-1}^\top\mathbf{A}_{N-1}$ as a tensor $\mathcal{Q}_N\in\mathbb{R}^{\mu_{N-1}\times I_N}$.
\STATE Form $\mathcal{B}=\mathcal{Q}_1\times_2^1\mathcal{Q}_2\times_3^1\dots\times_3^1\mathcal{Q}_N$.
\end{algorithmic}
\end{algorithm}

From the work in \cite{woolfe2008a}, for each $n$ in Algorithm \ref{RTT:alg3}, the orthonormal matrix $\mathbf{Q}_n$ is formed by extracting the first $\mu_n$ left singular vectors of $\mathbf{Z}_n$.

\begin{remark}
    For the case of $q=0$, when the matrix $\mathbf{\Omega}_n$ is a standard Gaussian matrix and the matrix $\mathbf{Q}_n$ is obtained by the thin QR decomposition of $\mathbf{Z}_n$, then Algorithm {\rm \ref{RTT:alg3}} is the same as Algorithm 5.1 in {\rm\cite{che2019randomized}}, and slightly different from the randomized TT-SVD in {\rm \cite{huber2017a}}, in which for each $n$, the size of all the core tensor $\mathcal{Q}_n$ is $\mu_{n-1}\times I_n\times \mu_n$.
\end{remark}

\begin{remark}
Note that Algorithm {\rm\ref{RTT:alg3}} works well for the case that all the matrices $\mathbf{Q}_{n-1}^\top\mathbf{A}_{n-1}$ exhibit some decay. When the singular values of some of these matrices decay slowly, the power iteration scheme has been proposed to relieve this weakness {\rm(}cf. {\rm \cite{martinsson2016randomized,halko2011finding,rokhlin2010randomized})}.
\end{remark}

For each $n$, we consider two ways to generate the matrix $\mathbf{\Omega}_n$ in Algorithm \ref{RTT:alg3}: the first one is to set the matrix $\mathbf{\Omega}_n\in\mathbb{R}^{I_{n+1}\dots I_N\times (\mu_n+R)}$ as a standard Gaussian matrix, and for the second one, we set $\mathbf{\Omega}_n=\mathbf{\Omega}_{n,n+1}\odot\dots\odot\mathbf{\Omega}_{n,N}$, where all the matrices $\mathbf{\Omega}_{n,m}\in\mathbb{R}^{I_m\times (\mu_n+R)}$ with $m=n+1,n+2,\dots,N$ are standard Gaussian matrices. As shown in \cite{sun2020low-rank}, the second type of the matrix $\mathbf{\Omega}_n$ is call the tensor random projection. Consequently, we denote Algorithm \ref{RTT:alg3} with these two ways by Algorithm \ref{RTT:alg3} with Gaussian and Algorithm \ref{RTT:alg3} with KR-Gaussian, respectively. We first analyze the computational complexity of Algorithm \ref{RTT:alg3} with Gaussian. In detail, we have
\begin{enumerate}
\item[(a)] For the case of $n=1$, forming the matrix $\mathbf{A}_1$ requires $O(I_1I_2\dots I_N)$ operations; generating $\mathbf{\Omega}_{1}$ requires $I_2\dots I_N(\mu_1+R)$ operations; computing the matrix $\mathbf{Z}_1$ requires $(4q+2)I_1I_2\dots I_N(\mu_1+R)$ operations; and obtaining $\mathbf{Q}_1$ requires $O(I_1(\mu_1+R)\mu_1)$ operations.
\item[(b)] For $n=2,\dots,N-1$, obtaining the matrix $\mathbf{A}_n$ of size $\mu_{n-1}I_n\times I_{n+1}\dots I_N$ requires $2I_nI_{n+1}\dots I_N\mu_{n-1}\mu_n+O(I_{n+1}\dots I_N\mu_n)$ operations; generating $\mathbf{\Omega}_{n}$ requires $I_{n+1}\dots I_N(\mu_n+R)$ operations; computing the matrix $\mathbf{Z}_n$ requires $(4q+2)I_n\dots I_N\mu_{n-1}(\mu_n+R)$ operations; obtaining $\mathbf{Q}_n$ requires $O(I_n\mu_{n-1}\mu_n(\mu_n+R))$ operations, and reshaping  $\mathbf{Q}_n$ to the core tensor $\mathcal{Q}_n$ requires $O(I_n\mu_{n-1}\mu_n)$ operations.
\item[(c)] Computing the $N$-th core $\mathcal{Q}_N$ as $\mathbf{Q}_{N-1}^\top\mathbf{A}_{N-1}$ requires $2I_{N-1}I_N\mu_{N-2}\mu_{N-1}$ operations.
\end{enumerate}

Hence, Algorithm \ref{RTT:alg3} with Gaussian requires
\begin{equation*}
\begin{split}
&\sum_{n=1}^{N-1}O(I_n\dots I_N\mu_{n-1})+2\cdot\sum_{n=2}^{N}I_{n-1}I_n\dots I_N\mu_{n-2}\mu_{n-1}\\
&+\sum_{n=1}^{N-1}I_{n+1}\dots I_N(\mu_n+R)+\sum_{n=1}^{N-1}O(I_n\mu_{n-1}\mu_n(\mu_n+R))\\
&+\sum_{n=2}^{N-1}O(I_n\mu_{n-1}\mu_n)+(4q+2)\sum_{n=1}^{N-1}I_n\dots I_{N}\mu_{n-1}(\mu_n+R)
\end{split}
\end{equation*}
operations for finding an approximation of the TT decomposition of $\mathcal{A}$ with given TT-rank $\{\mu_1,\mu_2,\dots,\mu_{N-1}\}$.
\begin{remark}
For each $n$ in Algorithm {\rm \ref{RTT:alg3}}, there exists another way to compute the matrix $\mathbf{Z}_n$ as follows: we first obtain the matrix $\mathbf{B}_n=\mathbf{A}_n\mathbf{A}_n^\top$, we then compute $\mathbf{Z}_n=\mathbf{A}_n\mathbf{\Omega}_n$, and we finally update $\mathbf{Z}_n=\mathbf{B}_n\mathbf{Z}_n$ $q$ times. This process needs $(q+2)$ matrix-matrix multiplications. However, for example, forming $\mathbf{A}_1\mathbf{A}_1^\top$ needs $2I_1^2I_2\dots I_N$ operations. Hence, in this paper, we avoid this way to obtain the matrix $\mathbf{Z}_n$.
\end{remark}
\begin{remark}
The difference among the algorithm in {\rm \cite{huber2017a}}, Algorithm \ref{RTT:alg3} with Gaussian, and Algorithm \ref{RTT:alg3} with KR-Gaussian is to generate the matrix $\mathbf{\Omega}_n$ and to obtain the orthonormal matrix $\mathbf{Q}_n$ from $\mathbf{Z}_n$. As shown in {\rm \cite{huber2017a}}, the size of $\mathbf{\Omega}_n$ is $I_{n+1}\dots I_N\times s_n$ with $s_n=\mu_n+R$ and the matrix $\mathbf{Q}_n$ is computed by the thin QR decomposition of $\mathbf{Z}_n$. For the case of $q=0$, the comparison of the computational complexities of the algorithm in {\rm \cite{huber2017a}}, Algorithm \ref{RTT:alg3} with Gaussian, and Algorithm \ref{RTT:alg3} with KR-Gaussian is listed in Table {\rm\ref{RTT:tab2}}.
\end{remark}
\begin{table*}[htb]
	\scriptsize
	\centering
	\begin{tabular}{c|c|c|c}
		 & The algorithm in {\rm \cite{huber2017a}} & Algorithm \ref{RTT:alg3} with Gaussian & Algorithm \ref{RTT:alg3} with KR-Gaussian \\
		 \hline
		 Obtain $\mathbf{A}_1$ & $O(I_1I_2\dots I_N)$ & $O(I_1I_2\dots I_N)$ & $O(I_1I_2\dots I_N)$ \\
		 \hline
		 Generate $\mathbf{\Omega}_1$ & $I_{2}I_{3}\dots I_Ns_1$ & $I_{2}I_{3}\dots I_N(\mu_{1}+R)$ & $(I_{2}+\dots+I_N)(\mu_1+R)$ \\
		 \hline
		 Compute $\mathbf{Z}_1$ & $2I_1I_{2}\dots I_Ns_1$ & $2I_1I_{2}\dots I_N(\mu_1+R)$ & $(\mu_1+R)$ \textsc{TenVecMult}s  \\
		 \hline
		 Obtain $\mathbf{Q}_1$ & $O(I_1s_1^2)$ & $O(I_1\mu_1(\mu_1+R))$ & $O(I_1\mu_1(\mu_1+R)I_1)$ \\
		 \hline
		 Compute $\mathbf{Q}_{n-1}^\top\mathbf{A}_{n-1}$ & $2I_n\dots I_N\mu_{n-1}s_n$ & $2I_n\dots I_N\mu_{n-1}\mu_n$ & $2I_n\dots I_N\mu_{n-1}\mu_n$ \\
		 \hline
		 Obtain $\mathbf{A}_n$ & $O(I_nI_{n+1}\dots I_Ns_n)$ & $O(I_nI_{n+1}\dots I_N\mu_n)$ & $O(I_nI_{n+1}\dots I_N\mu_n)$ \\
		 \hline
		 Generate $\mathbf{\Omega}_n$ & $I_{n}I_{n+1}\dots I_Ns_n$ & $I_{n}I_{n+1}\dots I_N(\mu_{n}+R)$ & $(I_{n}+\dots+I_N)(\mu_n+R)$ \\
		 \hline
		 Compute $\mathbf{Z}_n$ & $2I_nI_{n+1}\dots I_Ns_{n-1}s_n$ & $2I_nI_{n+1}\dots I_N\mu_{n-1}(\mu_n+R)$ & $(\mu_n+R)$ \textsc{TenVecMult}s \\
		 \hline
		 Obtain $\mathbf{Q}_n$ & $O(I_ns_{n-1}s_n^2)$ & $O(I_n\mu_{n-1}\mu_n(\mu_n+R))$ & $O(I_n\mu_{n-1}\mu_n(\mu_n+R))$ \\
          \hline
          Form $\mathcal{Q}_n$ & $O(I_ns_{n-1}s_n)$ & $O(I_n\mu_{n-1}\mu_n)$ & $O(I_n\mu_{n-1}\mu_n)$ \\
		 \hline
		 Obtain $\mathcal{Q}_N$ & $O(I_{N-1}I_Ns_{N-2}s_{N-1})$ & $O(I_{N-1}I_N\mu_{N-2}\mu_{N-1})$ & $O(I_{N-1}I_N\mu_{N-2}\mu_{N-1})$
	\end{tabular}
	\caption{For the case of $q=0$, the computational complexities of the algorithm in {\rm \cite{huber2017a}}, Algorithm \ref{RTT:alg3} with Gaussian and Algorithm \ref{RTT:alg3} with KR-Gaussian, where $n=2,3,\dots,N-1$.}
	\label{RTT:tab1}
\end{table*}

\begin{remark}
In Algorithm \ref{RTT:alg3} with KR-Gaussian, for each $n$, to generate $(N-n)$ standard Gaussian matrices $\mathbf{\Omega}_{n,m}$ requires $(I_{n+1}+\dots+I_N)(\mu_n+R)$ operations. Moreover, for $q=0$, the columns of $\mathbf{Z}_n$ can be implemented by the function ``ttv'' in MATLAB tensor toolbox {\rm \cite{bader2017matlab}}.

In Table {\rm \ref{RTT:tab1}}, for a tensor $\mathcal{A}\in\mathbb{R}^{I_1\times I_2\times \dots\times I_{N}}$ and $(N-1)$ vectors $\mathbf{x}_n\in\mathbb{R}^{I_n}$ with $n=2,3,\dots,N$, the symbol $\mathcal{A}\times_2\mathbf{x}_2^\top\dots\times_N\mathbf{x}_N^\top$ is denoted by a \textsc{TenVecMult} and implemented by ${\rm ttv}(\mathcal{A},\{\mathbf{x}_2,\dots,\mathbf{x}_N\},[2,\dots,N])$.
\end{remark}

\subsection{Framework of the second algorithm}

Note that in Algorithm \ref{RTT:alg3}, generating $\mathbf{\Omega}_n$ needs $\sum_{n=1}^{N-1}I_{n+1}\dots I_N(\mu_n+R)$ operations for Algorithm \ref{RTT:alg3} with Gaussian and $\sum_{n=1}^{N-1}(I_{n+1}+\dots +I_N)(\mu_n+R)$ operations for Algorithm \ref{RTT:alg3} with KR-Gaussian, and for the case of $q\geq 1$, to compute the matrix $\mathbf{Z}_n$ needs $2q+1$ matrix-matrix multiplications.

In order to reduce the computational complexity of generating the matrix $\mathbf{\Omega}_n$ and obtaining the matrix $\mathbf{Z}_n$, we deduce another randomized algorithm for the approximation of TT-decomposition of $\mathcal{A}\in\mathbb{R}^{I_1\times I_2\times \dots\times I_N}$ with a given TT-rank $\{\mu_1,\mu_2,\dots,\mu_{N-1}\}$, which is summarized in Algorithm \ref{RTT:alg2}.
\begin{algorithm}[htb]
\caption{The second randomized algorithm for an approximation of the TT decomposition with a given TT-rank}
\label{RTT:alg2}
\begin{algorithmic}[1]
\STATEx {\bf Input:} A tensor $\mathcal{A}\in \mathbb{R}^{I_1\times I_2\times \dots \times I_N}$, the desired TT-rank $\{\mu_1,\mu_2,\dots,\mu_{N-1}\}$, an oversampling parameter $R$ and the number of subspace iterations $q\geq 1$.
\STATEx {\bf Output:} An approximation of the TT decomposition of $\mathcal{A}$: $\mathcal{B}=\mathcal{Q}_1\times_2^1\mathcal{Q}_2\times_3^1\dots\times_3^1\mathcal{Q}_N$.
\STATE Set $\mu_0=1$, $\mathbf{Q}_0\in\mathbb{R}^{I_1\times I_1}$ being the identity matrix and $\mathbf{A}_0=\mathbf{A}_{(1)}$.
\FOR {$n=1,2,\dots,N-1$}
\STATE Reshape $\mathbf{Q}_{n-1}^\top\mathbf{A}_{n-1}$ as $\mathbf{A}_n\in\mathbb{R}^{\mu_{n-1}I_n\times \prod_{m=n+1}^NI_m}$.
\STATE Generate a standard Gaussian matrix ${\bf \Omega}_{n}\in\mathbb{R}^{\mu_{n-`}I_n\times (\mu_n+R)}$.
\STATE Compute ${\bf Z}_n=(\mathbf{A}_n\mathbf{A}_n^\top)^q{\bf \Omega}_n$.
\STATE Find an orthonormal matrix $\mathbf{Q}_n\in\mathbb{R}^{\mu_{n-1}I_n\times \mu_n}$ such that there exists $\mathbf{S}_n\in\mathbb{R}^{\mu_n\times (\mu_n+K)}$ for which
\begin{equation*}
\begin{split}
\|\mathbf{Q}_n\mathbf{S}_n-\mathbf{Z}_n\|_F=\Delta_{\mu_n+1}(\mathbf{Z}_n).
\end{split}
\end{equation*}
\STATE Reshape ${\bf Q}_n$ as a tensor $\mathcal{Q}_n\in\mathbb{R}^{\mu_{n-1}\times I_n\times \mu_n}$.
\ENDFOR
\STATE Reshape $\mathbf{Q}_{N-1}^\top\mathbf{A}_{N-1}$ as a tensor $\mathcal{Q}_N\in\mathbb{R}^{\mu_{N-1}\times I_N}$.
\STATE Form $\mathcal{B}=\mathcal{Q}_1\times_2^1\mathcal{Q}_2\times_3^1\dots\times_3^1\mathcal{Q}_N$.
\end{algorithmic}
\end{algorithm}

Similar to Algorithm \ref{RTT:alg3} with Gaussian, Algorithm \ref{RTT:alg2} needs
\begin{equation*}
\begin{split}
&\sum_{n=1}^{N-1}O(I_n\dots I_N\mu_{n-1})+2\cdot\sum_{n=2}^{N}I_{n-1}I_n\dots I_N\mu_{n-2}\mu_{n-1}\\
&+\sum_{n=1}^{N-1}I_{n}(\mu_n+R)+\sum_{n=1}^{N-1}O(I_n\mu_{n-1}\mu_n(\mu_n+R))\\
&+\sum_{n=2}^{N-1}O(I_n\mu_{n-1}\mu_n)+4q\sum_{n=1}^{N-1}I_n\dots I_{N}\mu_{n-1}(\mu_n+R)
\end{split}
\end{equation*}
operations for finding an approximation of the TT decomposition of $\mathcal{A}$ with given TT-rank $\{\mu_1,\mu_2,\dots,\mu_{N-1}\}$.
\section{Theoretical analysis}
\label{randomizedTT:sec4}
According to (\ref{RTT:eqnadd1}), if we estimate the upper bound of $\|\mathbf{A}_{n}-\mathbf{Q}_{n}\mathbf{Q}_{n}^\top\mathbf{A}_{n}\|_F$ with $n=1,2,\dots,N-1$, then it is easy to derive the upper bound of $\|\mathcal{A}-\mathcal{Q}_1\times_{2}^1\mathcal{Q}_2\times_{3}^1\dots\times_{3}^1\mathcal{Q}_N\|_F$, where $\mathbf{Q}_n$ is obtained from the proposed algorithms.
\subsection{The case for Algorithm \ref{RTT:alg3} with standard Gaussian matrices}
We introduce some necessary results for sub-random Gaussian matrices \cite{litvak2005smallest,rudelson2009smallest}. We first review the definition of the sub-Gaussian random variable.
\begin{definition}{\bf(see \cite[Definition 3.2]{shabat2016randomized})}
	\label{RTT:def4}
	A real valued random variable $X$ is called a sub-Gaussian random variable if there exists $b>0$ such that for all $t>0$ we have $\mathbf{E}(\exp(tX))\leq \exp(b^2t^2/2)$. A random variable $X$ is centered if $\mathbf{E}(X)=0$.
\end{definition}
The following theorem states the upper bound of the largest singular value of a random sub-Gaussian matrix.
\begin{theorem}{\bf(see \cite{litvak2005smallest})}
	\label{RTT:thm4}
	Suppose that $\mathbf{\Omega}\in \mathbb{R}^{I\times J}$ is random sub-Gaussian with $I\leq J$, $\tau\geq1$ and $a'>0$. Then
	$
	\mathbf{P}(\|\mathbf{\Omega}\|_2>a\sqrt{J})\leq \exp(-a'J)
	$,
	where $a=6\tau\sqrt{a'+4}$.
\end{theorem}
We now introduce the lower bound of the smallest singular value of a random sub-Gaussian matrix.

\begin{theorem}{{\bf (see \cite{litvak2005smallest})}}
	\label{RTT:thm5}
	Let $\tau\geq1$, $a>0$ and $a'>0$. Let $\mathbf{\Omega}\in\mathbb{R}^{I\times J}$ with $J>(1+1/\ln(I))I$. Suppose that the entries of $\mathbf{A}$ are centered independent random variables such that {\rm (a)} moments: $\mathbf{E}(|a_{ij}|^3)\leq \tau^3$; {\rm (b)} norm: $\mathbf{P}(\|\mathbf{\Omega}\|_2>a\sqrt{J})\leq \exp(-a'J)$; {\rm (c)} variance: $\mathbf{E}(|a_{ij}|^2)\leq 1$. Then, there exist positive constants $b$ and $b'$ such that
	\begin{equation*}
	\mathbf{P}(\sigma_I(\mathbf{\Omega})\leq b\sqrt{J})\leq \exp(-b'J).
	\end{equation*}
\end{theorem}

Following the work in \cite{che2020the,rokhlin2010randomized,shabat2016randomized}, we have the following lemma to illustrate that the product $\mathbf{Q}_n\mathbf{Q}_n^\top\mathbf{A}_{n}$
is a good approximation to $\mathbf{A}_n\in\mathbb{R}^{\mu_{n-1}I_n\times I_{n+1}\dots I_N}$, provided that there exist matrices $\mathbf{\Omega}_n\in\mathbb{R}^{I_{n+1}\dots I_N\times (\mu_n+R)}$ and $\mathbf{R}_n\in\mathbb{R}^{\mu_n\times (\mu_n+R)}$ such that: (a) $\mathbf{Q}_n\in\mathbb{R}^{\mu_{n-1}I_n\times \mu_n}$ is orthonormal;
(b) $\mathbf{Q}_n\mathbf{R}_n$ is a good approximation to $(\mathbf{A}_n\mathbf{A}_n^\top)^q\mathbf{A}_n\mathbf{\Omega}_n$, and
(c) there exists a matrix $\mathbf{F}_n\in\mathbb{R}^{(\mu_n+R)\times I_{n+1}\dots I_N}$ such that $\|\mathbf{F}\|_2$ is medium, and $(\mathbf{A}_n\mathbf{A}_n^\top)^q\mathbf{A}_n\mathbf{\Omega}_n\mathbf{F}_n$ is a good approximation to $\mathbf{A}_n$.

\begin{lemma}
\label{RTT-lem1}
For a given $n$, suppose that $\mu_{n-1}$, $\mu_n$, $R$, and $I_m \ (m=n,n+1,\dots, N)$ are positive integers with $\mu_n< (\mu_n+R)< \min\{\mu_{n-1}I_n,I_{n+1}\dots I_N\}$. Suppose further that $\mathbf{A}_n$ is a real $\mu_{n-1}I_n\times I_{n+1}\dots I_N$ matrix, $\mathbf{Q}_n$ is a real $\mu_{n-1}I_n\times \mu_n$ matrix whose columns are orthonormal, $\mathbf{R}_n$ is a real $\mu_n\times (\mu_n+R)$ upper triangular matrix, $\mathbf{F}_n$ is a real $(\mu_n+R)\times I_{n+1}\dots I_N$ matrix and $\mathbf{\Omega}_n$ is a real $I_{n+1}\dots I_N\times (\mu_n+R)$ matrix. Then we have
\begin{equation*}
\begin{split}
\|\mathbf{A}_n-\mathbf{Q}_n\mathbf{Q}_n^\top\mathbf{A}_n\|_F^2&=2\|\mathbf{B}_n\mathbf{\Omega}_n\mathbf{F}_n-\mathbf{A}_n\|_F^2+2\|\mathbf{F}_n\|_2^2\|\mathbf{Q}_n\mathbf{R}_n-\mathbf{B}_n\mathbf{\Omega}_n\|_F^2,
\end{split}
\end{equation*}
with $\mathbf{B}_n=(\mathbf{A}_n\mathbf{A}_n^\top)^q\mathbf{A}_n$, where $q\geq 0$ is a given positive integer.
\end{lemma}
\begin{proof}
Note that the proof is similar to that for Lemma 4.5 in \cite{che2020the}, Lemma 3.1 in \cite{rokhlin2010randomized} and Lemma 4.4 in \cite{shabat2016randomized}. Hence, we omit the detail process.
\end{proof}

The upper bound for $\|\mathbf{Q}_n\mathbf{R}_n-\mathbf{B}_n\mathbf{\Omega}_n\|_F^2$ is given in following lemma, which is from Theorem 5.2 in \cite{che2021randomized} or Theorem 4.10 in \cite{che2020the}.
\begin{lemma}
\label{RTT-lem2}
    Let the assumptions be the same as in Lemma \ref{RTT-lem1}. Suppose that $\mathbf{\Omega}_n\in\mathbb{R}^{I_{n+1}\dots I_N\times (\mu_n+R)}$ is a standard Gaussian matrix. We define $a_n$ and $a_n'$ as in Theorem \ref{RTT:thm4}. Then, the inequality
    \begin{equation*}
    \|\mathbf{Q}_n\mathbf{R}_n-\mathbf{B}_n\mathbf{\Omega}_n\|_F\leq a_n\sqrt{I_{n+1}\dots I_N}\Delta_{\mu_n+1}(\mathbf{A}_n)^{2q+1}
    \end{equation*}
    holds with probability at least $1-\exp(-a_n'I_{n+1}\dots I_N)$, where
    $a_n=6\tau\sqrt{a_n'+1}$, $\tau\geq1$ and
    \begin{equation*}
    \Delta_{\mu_n+1}(\mathbf{A}_n)=\left(\sum_{k=\mu_n+1}^{\min\{\mu_{n-1}I_n,I_{n+1}\dots I_N\}}\sigma_k(\mathbf{A}_n)^2\right)^{1/2}
    \end{equation*}
\end{lemma}
\begin{proof}
Following Lemma 4.4 in \cite{che2020the}, we have
\begin{equation*}
\|\mathbf{Q}_n\mathbf{R}_n-\mathbf{B}_n\mathbf{\Omega}_n\|_F^2\leq \sum_{i=\mu_n+1}^{\min\{\mu_{n-1}I_n,\mu_n+R\}}\sigma_i(\mathbf{B}_n\mathbf{\Omega}_n)^2.
\end{equation*}
By combining the above inequality with Lemma 4.3 in \cite{che2020the}, we have
\begin{equation*}
\|\mathbf{Q}_n\mathbf{R}_n-\mathbf{B}_n\mathbf{\Omega}_n\|_F^2\leq \|\mathbf{\Omega}_n\|_2^2\sum_{i=\mu_n+1}^{\min\{\mu_{n-1}I_n,I_{n+1}\dots I_N\}}\sigma_i(\mathbf{B}_n)^2,
\end{equation*}
which implies that
\begin{equation*}
\|\mathbf{Q}_n\mathbf{R}_n-\mathbf{B}_n\mathbf{\Omega}_n\|_F\leq a_n\sqrt{I_{n+1}\dots I_N}\Delta_{\mu_n+1}(\mathbf{B}_n)
\end{equation*}
holds with probability at least $1-\exp(-a_n'\mu_{n-1}I_n)$. Note that $\Delta_{\mu_n+1}(\mathbf{B}_n)\leq \Delta_{\mu_n+1}(\mathbf{A}_n)^{2q+1}$. Then the proof is completed.
\end{proof}
For the case of $\mathbf{\Omega}_n$ being any standard Gaussian matrix, we now give the upper bounds for $\|\mathbf{B}_n\mathbf{\Omega}_n\mathbf{F}_n-\mathbf{A}_n\|_F$ and $\|\mathbf{F}_n\|_2$, which is provided by Theorem 4.8 in \cite{che2020the}.
\begin{theorem}
\label{RTT-thm1}
Let the assumptions be the same as in Lemma \ref{RTT-lem1}. Suppose that $\mathbf{\Omega}_n\in\mathbb{R}^{I_{n+1}\dots I_N\times (\mu_n+R)}$ is a standard Gaussian matrix with $\mu_n+R>(1+1/\ln(\mu_n))\mu_n$ and $\mu_n+R\leq \min\{\mu_{n-1}I_n,I_{n+1}\dots I_N\}$. We define $a_n$, $a_n'$, $b_n$ and $b_n'$ as in Theorems \ref{RTT:thm4} and \ref{RTT:thm5}. Then, there exists a matrix $\mathbf{F}_n\in\mathbb{R}^{(\mu_n+R)\times I_{n+1}\dots I_N}$ such that
\begin{equation*}
\begin{split}
    &\|\mathbf{B}_n\mathbf{\Omega}_n\mathbf{F}_n-\mathbf{A}_n\|_F\leq \Delta_{\mu_n+1}(\mathbf{A}_n)+\frac{a_n\sqrt{I_{n+1}\dots I_N}}{b_n\sigma_{\mu_n}(\mathbf{A}_n)^{2q}\sqrt{\mu_n+R}}\Delta_{\mu_n+1}(\mathbf{A}_n)^{2q+1}
\end{split}
\end{equation*}
and
\begin{equation*}
\|\mathbf{F}_n\|_2\leq \frac{1}{b_n\sigma_{\mu_n}(\mathbf{A}_n)^{2q}\sqrt{\mu_n+R}}
\end{equation*}
with probability at least $$1-\exp(-a_n'I_{n+1}\dots I_n)-\exp(-b_n'(\mu_n+R)).$$
\end{theorem}
\begin{remark}

As shown in \cite{hochstenbach2010subspace}, we have
\begin{equation*}
\begin{split}
&\sum_{i=\mu_n+1}^{\min(\mu_{n-1} I_n,I_{n+1}\dots I_N)}\sigma_i(\mathbf{A}_{n})^2\leq \sum_{i=\mu_n+1}^{\min(I_1\dots I_n,I_{n+1}\dots I_N)}\sigma_i(\mathbf{A}_{([n])})^2.
\end{split}
\end{equation*}
However, the lower bound for $\sigma_{\mu_n}(\mathbf{A}_n)$ is unclear for any target tensor $\mathcal{A}\in\mathbb{R}^{I_1\times I_2\times \dots\times I_N}$. Hence, in the following, we only consider the theoretical results for Algorithm \ref{RTT:alg3} under the case of $q=0$.
\end{remark}
Hence, for the case of $q=0$, when $\{\mathbf{Q}_1,\mathbf{Q}_2,\dots,\mathbf{Q}_N\}$ is obtained from Algorithm {\rm \ref{RTT:alg3}} with standard Gaussian matrices, the upper bound for $\|\mathcal{A}-\mathcal{Q}_1\times_{2}^1\mathcal{Q}_2\times_{3}^1\dots\times_{3}^1\mathcal{Q}_N\|_F$ is summarized in the following theorem.
\begin{theorem}
\label{RTT:thm13}
For each $n$, let the assumptions be the same as in Lemma \ref{RTT-lem1}. Suppose that $\mathbf{\Omega}_n\in\mathbb{R}^{I_{n+1}\dots I_N\times (\mu_n+R)}$ is a standard Gaussian matrix with $\mu_n+R>(1+1/\ln(\mu_n))\mu_n$ and $\mu_n+R\leq \min\{\mu_{n-1}I_n,I_{n+1}\dots I_N\}$. We define $a_n$, $a_n'$, $b_n$ and $b_n'$ as in Theorems \ref{RTT:thm4} and \ref{RTT:thm5}. For $\mathcal{A}\in\mathbb{R}^{I_1\times I_2\times\dots\times I_N}$, if all the matrices $\mathbf{Q}_n$ are derived from Algorithm {\rm \ref{RTT:alg3}} with standard Gaussian matrices and $q=0$, then the inequality
\begin{equation*}
\|\mathcal{A}-\mathcal{Q}_1\times_{2}^1\mathcal{Q}_2\times_{3}^1\dots\times_{3}^1\mathcal{Q}_N\|_F\leq
2 \sum\limits_{n=1}^{N-1}\alpha_n
\Delta_{\mu_n+1}(\mathbf{A}_{([n])})
\end{equation*}
holds with the probability at least $$1-\sum_{n=1}^{N-1}(\exp(-b_{n}'(\mu_n+R))+\exp(-a_{n}'I_{n+1}\dots I_N)),$$ where the expression of $\alpha_n$ is given by
\begin{equation*}
\begin{split}
\alpha_n=&1+\sqrt{\frac{a_{n}^2I_{n+1}\dots I_N}{b_{n}^2(\mu_n+R)}}+\sqrt{\frac{a_{n}^2I_{n+1}\dots I_N}{b_{n}^2(\mu_n+R)}}.
\end{split}
\end{equation*}
Here $\Delta_{\mu_n+1}(\mathbf{A}_{([n])})$ is defined by
\begin{equation*}
\Delta_{\mu_n+1}(\mathbf{A}_{([n])})=\sqrt{\sum_{k=\mu_n+1}^{\min(I_1\dots I_n,I_{n+1}\dots I_N)}\sigma_i(\mathbf{A}_{([n])})^2},
\end{equation*}
where $\sigma_i(\mathbf{A}_{([n])})$ is the $i$-th singular value of $\mathbf{A}_{([n])}$.
\end{theorem}
\begin{proof}
By combining (\ref{RTT:eqnadd1}), Lemma \ref{RTT-lem1} with $q=0$ and the fact that $\sqrt{a+b}\leq \sqrt{a}+\sqrt{b}$ with $a,b\geq 0$, we have
\begin{equation}\label{RTT-eqn2}
    \begin{split}
        \|\mathcal{A}-\mathcal{Q}_1\times_{2}^1\mathcal{Q}_2\times_{3}^1\dots\times_{3}^1\mathcal{Q}_N\|_F
        &\leq2\sum_{n=1}^{N-1}\|\mathbf{A}_n\mathbf{\Omega}_n\mathbf{F}_n-\mathbf{A}_n\|_F\\
        &\quad+2\sum_{n=1}^{N-1}\|\mathbf{F}_n\|_2\|\mathbf{Q}_n\mathbf{R}_n-\mathbf{A}_n\mathbf{\Omega}_n\|_F.
    \end{split}
\end{equation}
Hence, the proof is completed by using Lemma \ref{RTT-lem2} and Theorem \ref{RTT-thm1} with $q=0$ to bound each term in the right-hand side of (\ref{RTT-eqn2}).
\end{proof}
\subsection{The case for Algorithm \ref{RTT:alg3} with the Khatri-Rao product of standard Gaussian matrices}

Similar to the case of Algorithm \ref{RTT:alg3} with standard Gaussian matrices, in this section, we consider the upper bound for $\|\mathcal{A}-\mathcal{Q}_1\times_{2}^1\mathcal{Q}_2\times_{3}^1\dots\times_{3}^1\mathcal{Q}_N\|_F$, where $\{\mathcal{Q}_1,\mathcal{Q}_2,\dots,\mathcal{Q}_N\}$ is obtained from Algorithm \ref{RTT:alg3} with the Khatri-Rao product of standard Gaussian matrices and $q=0$. The following theorem from \cite{vershynin2012introduction} provides bounds on the condition numbers of
matrices whose columns are independent sub-Gaussian isotropic random variables.
\begin{theorem}{{\bf (see \cite[Theorem 5.58]{vershynin2012introduction})}}
\label{RTT-thm2}
Let $\mathbf{A}$ be an $I\times J\ (I\geq J)$ matrix whose columns $\mathbf{\Omega}(:,j)$ are independent sub-Gaussian isotropic random vectors in $\mathbb{R}^I$ with $\|\mathbf{\Omega}(:,j)\|_2=\sqrt{I}$ almost surely (a.s.). Then for every $t\geq0$, with probability at least $1-2\exp(-c t^2)$, one has
\begin{equation*}
\sqrt{I}-C\sqrt{J}-t\leq\sigma_{\min}(\mathbf{\Omega})\leq
\sigma_{\max}(\mathbf{\Omega})\leq\sqrt{I}+C\sqrt{J}+t.
\end{equation*}
Here $C=C_K$ and $c=c_K\geq0$ depend only on the sub-Gaussian norm $K=\max_j\|\mathbf{\Omega}(:,j)\|_{\psi_2}$.
\end{theorem}

When the matrix $\mathbf{\Omega}_n$ in Algorithm \ref{RTT:alg3} is given by $\mathbf{\Omega}_n=\mathbf{\Omega}_{n,n+1}\odot\dots\odot\mathbf{\Omega}_{n,N}$, where all the matrices $\mathbf{\Omega}_{n,m}\in\mathbb{R}^{I_m\times (\mu_n+R)}$ with $m=n+1,n+2,\dots,N$ are standard Gaussian matrices, then $\mathbf{\Omega}_n$ is a random matrix whose columns are independent from one another but whose rows are not (cf. \cite{che2019randomized,reynolds2016randomized}).
\begin{lemma}{{\bf (see \cite[Lemma 4.1]{che2019randomized})}}
\label{RTT-lem4}
For a given $n$, suppose that $\mu_n$, $R$, and $I_m \ (m=n+1,\dots, N)$ are positive integers with $\mu_n< (\mu_n+R)< \min\{\mu_{n-1}I_n,I_{n+1}\dots I_N\}$. Let $\mathbf{\Omega}_{n,m}\in\mathbb{R}^{I_m\times (\mu_n+R)}$ be standard Gaussian matrix with $m=n+1,n+2,\dots,N$, and $\mathbf{Q}\in\mathbb{R}^{I_{n+1}\dots I_N\times R_Q}$ be orthonormal with $R_Q<\mu_n+R<I_{n+1}\dots I_N$. Define $\mathbf{\Omega}_n=\mathbf{\Omega}_{n,n+1}\odot\dots\odot\mathbf{\Omega}_{n,N}$. Then $\mathbf{Q}^\top\mathbf{\Omega}_n$ is a random matrix with isotropic columns.
\end{lemma}

Note that Lemma \ref{RTT-lem1} is also suitable when the matrix $\mathbf{Q}_n$ is obtained by Algorithm \ref{RTT:alg3} with the Khatri-Rao product of standard Gaussian matrices and $q=0$. We now consider the upper bound for $\|\mathbf{Q}_n\mathbf{R}_n-\mathbf{A}_n\mathbf{\Omega}_n\|_F^2$.
\begin{lemma}
\label{RTT-lem3}
    For a given $n$, suppose that $\mu_n$, $R$, and $I_m \ (m=n+1,\dots, N)$ are positive integers with $\mu_n< (\mu_n+R)< I_{n+1}\dots I_N$. Let $\mathbf{\Omega}_{n,m}\in\mathbb{R}^{I_m\times (\mu_n+R)}$ be standard Gaussian matrix with $m=n+1,n+2,\dots,N$. Define $\mathbf{\Omega}_n=\mathbf{\Omega}_{n,n+1}\odot\dots\odot\mathbf{\Omega}_{n,N}$. Then, for any $t\geq 0$, the inequality
    \begin{equation*}
    \begin{split}
    &\|\mathbf{Q}_n\mathbf{R}_n-\mathbf{A}_n\mathbf{\Omega}_n\|_F\leq (\sqrt{I_{n+1}\dots I_N}+C_{n}\sqrt{\mu_n+R}+t)\Delta_{\mu_n+1}(\mathbf{A}_n)
    \end{split}
    \end{equation*}
    holds with probability at least $1-2 \exp(-c_{n}t^2)$, where $C_{n}=C_{K_{n}}$ and $c_{n}=c_{K_{n}}\geq0$ depend only on the sub-Gaussian norm $K_{n}=\max_j\|\mathbf{\Omega}_{n}(:,j)\|_{\psi_2}$.
\end{lemma}
\begin{proof}
By Lemma 4.4 in \cite{che2020the} or Lemma 5.3 in \cite{che2021randomized}, we have
\begin{equation*}
\|\mathbf{Q}_n\mathbf{R}_n-\mathbf{A}_n\mathbf{\Omega}_n\|_F\leq \Delta_{\mu_{n}+1}(\mathbf{A}_n\mathbf{\Omega}_n).
\end{equation*}
Note that $\sigma_i(\mathbf{A}_n\mathbf{\Omega}_n)\leq \|\mathbf{\Omega}_n\|_2\sigma_i(\mathbf{A}_n)$. Then, we have
\begin{equation*}
\|\mathbf{Q}_n\mathbf{R}_n-\mathbf{A}_n\mathbf{\Omega}_n\|_F\leq \|\mathbf{\Omega}_n\|_2\Delta_{\mu_{n}+1}(\mathbf{A}_n).
\end{equation*}
Hence, the proof is completed by combining the above inequality and Theorem \ref{RTT-thm2}.
\end{proof}
We denote the SVD of $\mathbf{A}_{(n)}$ as $\mathbf{A}_{(n)}=\mathbf{U}_n\mathbf{\Sigma}_n\mathbf{V}_n^{\top}$, where $\mathbf{U}_n\in\mathbb{R}^{I_1\dots I_{n}\times I_1\dots I_n}$ and $\mathbf{V}\in\mathbb{R}^{I_{n+1}\dots I_N\times I_{n+1}\dots I_N}$ are orthogonal, and $\mathbf{\Sigma}\in\mathbb{R}^{I_1\dots I_{n}\times I_{n+1}\dots I_N}$ is diagonal with nonnegative entries. Define $\mathbf{\Omega}_n'=\mathbf{V}_n^\top\mathbf{\Omega}_n(:,)$. Let $\mathbf{\Omega}_{n,1}=\mathbf{\Omega}_n'(1:\mu_n,:)$ and $\mathbf{\Omega}_{n,1}=\mathbf{\Omega}_n'(\mu_n+1:I_{n+1}\dots I_N,:)$. It follows from Lemma \ref{RTT-lem4} that $\mathbf{V}_n^\top\mathbf{\Omega}_n$ is a random matrix with isotropic columns. Therefore, the matrices $\mathbf{\Omega}_{n,1}$ and $\mathbf{R}_{n,2}$ are also random matrices with isotropic columns.
\begin{theorem}
\label{RTT-thm3}
For a given $n$, suppose that $\mu_n$, $R$, and $I_m \ (m=n+1,\dots, N)$ are positive integers with $\mu_n< (\mu_n+R)< \min\{\mu_{n-1}I_n,I_{n+1}\dots I_N\}$. Let $\mathbf{\Omega}_{n,m}\in\mathbb{R}^{I_m\times (\mu_n+R)}$ be standard Gaussian matrix with $m=n+1,n+2,\dots,N$. Define $\mathbf{\Omega}_n=\mathbf{\Omega}_{n,n+1}\odot\dots\odot\mathbf{\Omega}_{n,N}$. Then, there exists a matrix $\mathbf{F}_n\in\mathbb{R}^{(\mu_n+R)\times I_{n+1}\dots I_N}$ such that
\begin{equation*}
    \|\mathbf{A}_n\mathbf{\Omega}_n\mathbf{F}_n-\mathbf{A}_n\|_F\leq \Delta_{\mu_n+1}(\mathbf{A}_n)+\frac{\sqrt{I_{n+1}\dots I_N-\mu_n}+C_n''\sqrt{\mu_n}+t}{\sqrt{\mu_n+R}-C_n'\sqrt{\mu_n}-t}\Delta_{\mu_n+1}(\mathbf{A}_n)
\end{equation*}
and
\begin{equation*}
\|\mathbf{F}_n\|_2\leq \frac{1}{\sqrt{\mu_n+R}-C_n'\sqrt{\mu_n}-t}
\end{equation*}
with probability at least $1-\exp(-c_n't^2)-\exp(-c_n''t^2)$, where $C_{n}'=C_{K_{n}'}$ and $c_{n}'=c_{K_{n}'}\geq0$ depend only on the sub-Gaussian norm $K_{n}'=\max_j\|\mathbf{\Omega}_{n,1}(j,:)\|_{\psi_2}$, and $C_{n}''=C_{K_{n}''}$ and $c_{n}'=c_{K_{n}''}\geq0$ depend only on the sub-Gaussian norm $K_{n}''=\max_j\|\mathbf{\Omega}_{n,2}(:,j)\|_{\psi_2}$.
\end{theorem}
\begin{proof}
Firstly, we consider the case of $\mu_{n-1}I_n\geq I_{n+1}\dots I_N$. Consider the SVD of $\mathbf{A}_n$ as  $\mathbf{A}_n=\mathbf{U}_n'\mathbf{\Sigma}_n'\mathbf{V}_n'^{\top}$, where $\mathbf{U}_n'\in\mathbb{R}^{\mu_{n-1}I_n\times I_{n+1}\dots I_N}$ is orthonormal, $\mathbf{V}_n'\in\mathbb{R}^{I_{n+1}\dots I_N\times I_{n+1}\dots I_N}$ is orthogonal, and $\mathbf{\Sigma}_n'\in\mathbb{R}^{I_{n+1}\dots I_N\times I_{n+1}\dots I_N}$ is diagonal and its diagonal entries are the singular values of $\mathbf{A}_n$.

For the matrices $\mathbf{V}_n'$ and $\mathbf{\Omega}_n$, let
\begin{equation*}
\mathbf{V}_n'^\top\mathbf{\Omega}_n=
\begin{pmatrix}
\mathbf{H}_n\\
\mathbf{R}_n
\end{pmatrix},
\end{equation*}
with $\mathbf{H}_n\in\mathbb{R}^{\mu_n\times (\mu_n+R)}$ and $\mathbf{R}\in\mathbb{R}^{(I_{n+1}\dots I_N-\mu_n)\times (\mu_n+R)}$. It follows from Lemma \ref{RTT-lem4} that $\mathbf{V}_n'^\top\mathbf{\Omega}_n$ is a random matrix with isotropic columns. It is easy to see that $\mathbf{H}_n$ and $\mathbf{\Omega}_{n,1}$ have the same distribution. This fact also holds for $\mathbf{R}_n$ and $\mathbf{\Omega}_{n,2}$. Therefore, the matrices $\mathbf{H}_n$ and $\mathbf{R}_n$ are also random matrices with isotropic columns.

Define $\mathbf{F}_n=\mathbf{P}_n\mathbf{V}_n'^\top$, where $\mathbf{P}_n\in\mathbb{R}^{(\mu_n+R)\times I_{n+1}\dots I_N}$ satisfies
\begin{equation*}
\mathbf{P}_n=
\begin{pmatrix}
\mathbf{H}_n^\dag&\mathbf{0}_{(\mu_n+R)\times (I_{n+1}\dots I_N-\mu_n)}
\end{pmatrix}.
\end{equation*}
From Theorem \ref{RTT-thm2}, for any $t\geq 0$, we get
\begin{equation*}
    \|\mathbf{F}_n\|_2=\|\mathbf{P}_n\mathbf{V}_n'^\top\|_2
    =\|\mathbf{H}_n^\dag\|_2
    \leq\frac{\|\mathbf{H}_n^\dag\|_2}{\sigma_{\mu_n}(\mathbf{A}_n)^{2q}}=\frac{1}{(\sqrt{\mu_n+R}-C_n'\sqrt{\mu_n}-t)}
\end{equation*}
with probability not less than $1-2\exp(-c_n't^2)$, where $C_{n}'=C_{K_{n}'}$ and $c_{n}'=c_{K_{n}'}\geq0$ depend only on the sub-Gaussian norm $K_{n}'=\max_j\|\mathbf{H}_{n}(j,:)\|_{\psi_2}=\max_j\|\mathbf{\Omega}_{n,1}(j,:)\|_{\psi_2}$.

Now, we bound $\|(\mathbf{A}_n\mathbf{A}_n^{\top})^q\mathbf{A}_n\mathbf{\Omega}_n\mathbf{F}_n-\mathbf{A}_n\|_F$. We define $\mathbf{\Sigma}_{n,1}'$ and $\bm{\Sigma}_{n,2}'$ to be the $\mu_n\times \mu_n$ upper-left block and the $(I_{n+1}\dots I_N-\mu_n)\times (I_{n+1}\dots I_N-\mu_n)$ lower-right block of $\mathbf{\Sigma}_n$, respectively. Then, we get
\begin{equation}\label{RP:eqn14}
\begin{split}
\mathbf{A}_n\mathbf{\Omega}_n\mathbf{F}_n-\mathbf{A}_n
&=\mathbf{U}_n'\mathbf{\Sigma}_n'
\begin{pmatrix}
\mathbf{0}_{\mu_n\times \mu_n}&\mathbf{0}_{\mu_n\times (I_{n+1}\dots I_N-\mu_n)}\\
\mathbf{R}_n\mathbf{H}_n^\dag & -\mathbf{I}_{I_{n+1}\dots I_N-\mu_n}
\end{pmatrix}
\mathbf{V}_n'^{\top}\\
&=\mathbf{U}_n'
\begin{pmatrix}
\mathbf{0}_{\mu_n\times \mu_n}&\mathbf{0}_{\mu_n\times (I_{n+1}\dots I_N-\mu_n)}\\
\mathbf{\Sigma}_{n,2}'\mathbf{R}_n\mathbf{H}_n^\dag & -\mathbf{\Sigma}_{n,2}'
\end{pmatrix}
\mathbf{V}_n'^{\top},
\end{split}
\end{equation}
which implies that
\begin{equation*}
\begin{split}
\|\mathbf{A}_n\mathbf{\Omega}_n\mathbf{F}_n-\mathbf{A}_n\|_F
&\leq
\|\mathbf{\Sigma}_{n,2}'\|_F+\|\mathbf{R}_n\|_2\|\mathbf{H}^\dag\|_2\|\mathbf{\Sigma}_{n,2}'\|_F\\
&=(1+\|\mathbf{R}_n\|_2\|\mathbf{H}^\dag\|_2)\Delta_{\mu_n+1}(\mathbf{A}_n).
\end{split}
\end{equation*}
From Theorem \ref{RTT-thm2}, for any $t\geq 0$, we get
\begin{equation*}
\|\mathbf{R}_n\|_2\leq\sqrt{I_{n+1}\dots I_N-\mu_n}+C_n'\sqrt{\mu_n+R}+t
\end{equation*}
with probability not less than $1-2\exp(-c_n''t^2)$, where $C_{n}''=C_{K_{n}''}$ and $c_{n}''=c_{K_{n}''}\geq0$ depend only on the sub-Gaussian norm $K_{n}''=\max_j\|\mathbf{R}_{n}(:,j)\|_{\psi_2}=\max_j\|\mathbf{\Omega}_{n,2}(:,j)\|_{\psi_2}$.

The case of $\mu_{n-1}I_n< I_{n+1}\dots I_N$ is similar to the case of $\mu_{n-1}I_n\geq I_{n+1}\dots I_N$ except for the matrix $\mathbf{\Sigma}_n'$. For the last case, we have $$\mathbf{\Sigma}_n'=(\mathbf{\Sigma}_n'',\mathbf{0}_{\mu_{n-1}I_n\times (I_{n+1}\dots I_N-\mu_{n-1}I_n)}),$$ where $\mathbf{\Sigma}_n''\in\mathbb{R}^{\mu_{n-1}I_n\times \mu_{n-1}I_n}$ is diagonal and its diagonal entries are the singular values of $\mathbf{A}_n$.

Hence, the proof is completed.
\end{proof}
The following theorem is deduced by combining (\ref{RTT:eqnadd1}), Lemma \ref{RTT-lem1} with $q=0$, Lemma \ref{RTT-lem3} and Theorem \ref{RTT-thm3}.
\begin{theorem}
\label{RTT-thm7}
For each $n$, suppose that $\mu_n$, $R$, and $I_m \ (m=1,2,\dots, N)$ are positive integers with $\mu_n< (\mu_n+R)< \min\{\mu_{n-1}I_n,I_{n+1}\dots I_N\}$, let $\mathbf{\Omega}_{n,m}\in\mathbb{R}^{I_m\times (\mu_n+R)}$ be standard Gaussian matrix with $m=n+1,n+2,\dots,N$, and define $\mathbf{\Omega}_n=\mathbf{\Omega}_{n,n+1}\odot\dots\odot\mathbf{\Omega}_{n,N}$. For $\mathcal{A}\in\mathbb{R}^{I_1\times I_2\times\dots\times I_N}$, if all the matrices $\mathbf{Q}_n$ are derived from Algorithm {\rm \ref{RTT:alg3}} with the Khatri-Rao product of standard Gaussian matrices and $q=0$, then for any $t\geq 0$, the inequality
\begin{equation*}
\|\mathcal{A}-\mathcal{Q}_1\times_{2}^1\mathcal{Q}_2\times_{3}^1\dots\times_{3}^1\mathcal{Q}_N\|_F\leq
2 \sum\limits_{n=1}^{N-1}\beta_n
\Delta_{\mu_n+1}(\mathbf{A}_{([n])})
\end{equation*}
holds with the probability at least $1-\sum_{n=1}^{N-1}(\exp(-c_nt^2)+\exp(-c_n't^2)+\exp(-c_n''t^2))$, where the expression of $\beta_n$ is given by
\begin{equation*}
    \beta_n=1+\frac{\sqrt{I_{n+1}\dots I_N-\mu_n}+C_n''\sqrt{\mu_n}+t}{\sqrt{\mu_n+R}-C_n'\sqrt{\mu_n}-t}+\frac{\sqrt{I_{n+1}\dots I_N}+C_n''\sqrt{\mu_n}+t}{\sqrt{\mu_n+R}-C_n'\sqrt{\mu_n}-t}.
\end{equation*}
Here, for each $n$, $C_{n}=C_{K_{n}}$ and $c_{n}=c_{K_{n}}\geq0$ depend only on  $K_{n}=\max_j\|\mathbf{\Omega}_{n}(:,j)\|_{\psi_2}$, $C_{n}'=C_{K_{n}'}$ and $c_{n}'=c_{K_{n}'}\geq0$ depend only on $K_{n}'=\max_j\|\mathbf{\Omega}_{n,1}(j,:)\|_{\psi_2}$, and $C_{n}''=C_{K_{n}''}$ and $c_{n}'=c_{K_{n}''}\geq0$ depend only on $K_{n}''=\max_j\|\mathbf{G}_{n,2}(:,j)\|_{\psi_2}$.
\end{theorem}

\subsection{The case for Algorithm \ref{RTT:alg2}}
When $\{\mathbf{Q}_1,\mathbf{Q}_2,\dots,\mathbf{Q}_N\}$ is obtained from Algorithm {\rm \ref{RTT:alg2}} with $q=1$, we now consider the upper bound for $\|\mathcal{A}-\mathcal{Q}_1\times_{2}^1\mathcal{Q}_2\times_{3}^1\dots\times_{3}^1\mathcal{Q}_N\|_F$. First of all, following the work in \cite{che2023efficient}, we have
\begin{equation}
\label{RTT-eqn4}
    \|\mathbf{A}_n-\mathbf{Q}_n\mathbf{Q}_n^\top\mathbf{A}_n\|_F^2\leq {\rm rank}(\mathbf{A}_n)\|\mathbf{A}_n\mathbf{A}_n^\top-\mathbf{Q}_n\mathbf{Q}_n^\top\mathbf{A}_n\mathbf{A}_n^\top\|_F.
\end{equation}
For the term $\|\mathbf{A}_n\mathbf{A}_n^\top-\mathbf{Q}_n\mathbf{Q}_n^\top\mathbf{A}_n\mathbf{A}_n^\top\|_F$, we have the following lemma.
\begin{lemma}
\label{RTT-lem5}
For a given $n$, suppose that $\mu_{n-1}$, $\mu_n$, $R$, and $I_m \ (m=n,n+1,\dots, N)$ are positive integers with $\mu_n< (\mu_n+R)< \min\{\mu_{n-1}I_n,I_{n+1}\dots I_N\}$. For a given matrix $\mathbf{A}_n\in\mathbb{R}^{\mu_{n-1}I_n\times I_{n+1}\dots I_N}$, there exist an orthonormal matrix $\mathbf{Q}_n\in\mathbb{R}^{\mu_{n-1}I_n\times \mu_n}$, and two matrices $\mathbf{\Omega}_n\in\mathbb{R}^{\mu_{n-1}I_n\times (\mu_n+R)}$ and $\mathbf{F}_n\in\mathbb{R}^{(\mu_n+R)\times \mu_{n-1}I_n}$ such that
\begin{equation*}
\begin{split}
&\|\mathbf{A}_n\mathbf{A}_n^\top-\mathbf{Q}_n\mathbf{Q}_n^\top\mathbf{A}_n\mathbf{A}_n^\top\|_F=2\|\mathbf{A}_n\mathbf{A}_n^\top\mathbf{\Omega}_n\mathbf{F}_n-\mathbf{A}_n\mathbf{A}_n^\top\|_F\\
&+2\|\mathbf{F}_n\|_2\|\mathbf{Q}_n\mathbf{R}_n-\mathbf{A}_n\mathbf{A}_n^\top\mathbf{\Omega}_n\|_F,
\end{split}
\end{equation*}
\end{lemma}
The following theorem is obtained by replacing $\mathbf{B}_n$ in Lemma \ref{RTT-lem2} with $\mathbf{A}_n\mathbf{A}_n^\top$.
\begin{lemma}
\label{RTT-lem6}
For a given $n$, suppose that $\mu_{n-1}$, $\mu_n$, $R$, and $I_m \ (m=n,n+1,\dots, N)$ are positive integers with $\mu_n< (\mu_n+R)< \min\{\mu_{n-1}I_n,I_{n+1}\dots I_N\}$. Suppose that $\mathbf{\Omega}_n\in\mathbb{R}^{\mu_{n-1}I_n\times (\mu_n+R)}$ is a standard Gaussian matrix. We define $a_n$ and $a_n'$ as in Theorem \ref{RTT:thm4}. Then, the inequality
\begin{equation*}
\|\mathbf{Q}_n\mathbf{R}_n-\mathbf{A}_n\mathbf{A}_n^\top\mathbf{\Omega}_n\|_F\leq a_n\sqrt{\mu_{n-1}I_n}\Delta_{\mu_n+1}(\mathbf{A}_n\mathbf{A}_n^\top)
\end{equation*}
holds with probability at least $1-\exp(-a_n'\mu_{n-1}I_n)$, where
$a_n=6\tau\sqrt{a_n'+1}$, $\tau\geq1$ and
\begin{equation*}
\Delta_{\mu_n+1}(\mathbf{A}_n\mathbf{A}_n^\top)=\left(\sum_{k=\mu_n+1}^{\mu_{n-1}I_n}\sigma_k(\mathbf{A}_n\mathbf{A}_n^\top)^2\right)^{1/2}
\end{equation*}
\end{lemma}
We now consider the upper bounds for $\|\mathbf{A}_n\mathbf{A}_n^\top\mathbf{\Omega}_n\mathbf{F}_n-\mathbf{A}_n\mathbf{A}_n^\top\|_F$ and $\|\mathbf{F}_n\|_2$, which are summarized in the following theorem.
\begin{theorem}
\label{RTT-thm6}
For a given $n$, suppose that $\mu_{n-1}$, $\mu_n$, $R$, and $I_m \ (m=n,n+1,\dots, N)$ are positive integers with $\mu_n< (\mu_n+R)< \min\{\mu_{n-1}I_n,I_{n+1}\dots I_N\}$. Suppose that $\mathbf{\Omega}_n\in\mathbb{R}^{\mu_{n-1}I_n\times (\mu_n+R)}$ is a standard Gaussian matrix. We define $a_n$, $a_n'$, $b_n$ and $b_n'$ as in Theorems \ref{RTT:thm4} and \ref{RTT:thm5}. Then, there exists a matrix $\mathbf{F}_n\in\mathbb{R}^{(\mu_n+R)\times \mu_{n-1}I_n}$ such that
\begin{equation*}
\|\mathbf{A}_n\mathbf{A}_n^\top\mathbf{\Omega}_n\mathbf{F}_n-\mathbf{A}_n\mathbf{A}_n^\top\|_F\leq \Delta_{\mu_n+1}(\mathbf{A}_n\mathbf{A}_n^\top)+\frac{a_n\sqrt{\mu_{n-1}I_n}}{b_n\sqrt{\mu_n+R}}\Delta_{\mu_n+1}(\mathbf{A}_n\mathbf{A}_n^\top)
\end{equation*}
and
\begin{equation*}
\|\mathbf{F}_n\|_2\leq \frac{1}{b_n\sqrt{\mu_n+R}}
\end{equation*}
with probability at least $$1-\exp(-a_n'\mu_{n-1}I_n)-\exp(-b_n'(\mu_n+R)).$$
\end{theorem}
\begin{proof}
The results are easily proved by replacing $\mathbf{B}_n$ in Theorem \ref{RTT-thm1} with $\mathbf{A}_n\mathbf{A}_n^\top$ and setting $q=0$.
\end{proof}
By combining (\ref{RTT:eqnadd1}), (\ref{RTT-eqn4}), Lemmas \ref{RTT-lem5} and \ref{RTT-lem6}, and Theorem \ref{RTT-thm6}, we have the following theorem.
\begin{theorem}
\label{RTT-thm4}
For each $n$, suppose that $\mu_{n-1}$, $\mu_n$, $R$, and $I_m \ (m=n,n+1,\dots, N)$ are positive integers with $\mu_n< (\mu_n+R)< \min\{\mu_{n-1}I_n,I_{n+1}\dots I_N\}$. Suppose that $\mathbf{\Omega}_n\in\mathbb{R}^{\mu_{n-1}I_{n}\times (\mu_n+R)}$ is a standard Gaussian matrix with $\mu_n+R>(1+1/\ln(\mu_n))\mu_n$. We define $a_n$, $a_n'$, $b_n$ and $b_n'$ as in Theorems \ref{RTT:thm4} and \ref{RTT:thm5}.

For $\mathcal{A}\in\mathbb{R}^{I_1\times I_2\times\dots\times I_N}$, if all the matrices $\mathbf{Q}_n$ are derived from Algorithm {\rm \ref{RTT:alg2}} with $q=1$, then the inequality
\begin{equation*}
     \|\mathcal{A}-\mathcal{Q}_1\times_{2}^1\mathcal{Q}_2\times_{3}^1\dots\times_{3}^1\mathcal{Q}_N\|_F\leq
     2 \sum\limits_{n=1}^{N-1}\gamma_n
     {\rm rank}(\mathbf{A}_{([n])})\Delta_{\mu_n+1}(\mathbf{A}_{([n])}\mathbf{A}_{([n])}^\top)
\end{equation*}
holds with the probability at least $$1-\sum_{n=1}^{N-1}(\exp(-b_{n}'(\mu_n+R))+\exp(-a_{n}'\mu_{n-1}I_n)),$$ where the expression of $\gamma_n$ is given by
\begin{equation*}
\begin{split}
\gamma_n=&1+\sqrt{\frac{a_{n}^2\mu_{n-1}I_n}{b_{n}^2(\mu_n+R)}}+\sqrt{\frac{a_{n}^2\mu_{n-1}I_n}{b_{n}^2(\mu_n+R)}}.
\end{split}
\end{equation*}
Here $\Delta_{\mu_n+1}(\mathbf{A}_{([n])}\mathbf{A}_{([n])}^\top)$ is defined by
\begin{equation*}
\Delta_{\mu_n+1}(\mathbf{A}_{([n])}\mathbf{A}_{([n])}^\top)=\sqrt{\sum_{k=\mu_n+1}^{\min(I_1\dots I_n,I_{n+1}\dots I_N)}\sigma_i(\mathbf{A}_{([n])})^4},
\end{equation*}
where $\sigma_i(\mathbf{A}_{([n])})$ is the $i$-th singular value of $\mathbf{A}_{([n])}$.
\end{theorem}

\subsection{Comparison of the upper bounds of TT-SVD and the proposed algorithms}
\label{RTT:sec4}
For a given desired TT-rank $\{\mu_1,\mu_2,\dots,\mu_{N-1}\}$, suppose that $\mathcal{B}$ is an approximation of the TT decomposition of $\mathcal{A}\in\mathbb{R}^{I_1\times I_2\times\dots\times I_N}$, which is obtained from TT-SVD, the algorithm in \cite{huber2017a}, and Algorithms \ref{RTT:alg3} and \ref{RTT:alg2}. For TT-SVD, as shown in \cite{oseledets2010tt}, we have that
\begin{equation*}
    \|\mathcal{A}-\mathcal{B}\|_F\leq\sqrt{\sum_{n=1}^{N-1}\sum_{i_n=\mu_n+1}^{\widehat{I}_n}\sigma_{i_n}(\mathbf{A}_{([n])})^2}\leq\sum_{n=1}^{N-1}\Delta_{\mu_n+1}(\mathbf{A}_{([n])}),
\end{equation*}
where the second inequality holds with the fact that $\sqrt{a+b}\leq \sqrt{a}+\sqrt{b}$ with $a,b\geq 0$.

According to \cite[Theorem 2]{huber2017a}, we have that
\begin{equation*}
    \|\mathcal{A}-\mathcal{B}\|_F\leq \sum_{n=1}^{N-1}\eta_n\Delta_{\mu_n+1}(\mathbf{A}_{(n)})
\end{equation*}
with the probability at least $1-\sum_{n=1}^{N-1}(5t_n^{-R}+2e^{-u_n^2/2})$, where for $n=1,2,\dots,N-1$, $t_n>1$, $u_n>1$ and
\begin{equation*}
    \eta_n=1+t_n\sqrt{\frac{12\mu_n}{R}}+u_nt_n\frac{e\sqrt{\mu_n+R}}{R+1}.
\end{equation*}

On the one hand, for some suitable choices of all the parameters, the values of $\eta_n$, $\alpha_n$, $\beta_n$ and $\gamma_n$ are larger than 1. This means that the upper bound obtained by TT-SVD is the tightest. On the other hand, for Algorithms \ref{RTT:alg3} and \ref{RTT:alg2}, the values of all the parameters, satisfying the corresponding assumptions, are not unique, which implies that the upper bounds obtained by these algorithms have no obvious relationships.

\section{Randomized algorithms for the fixed precision problem}
\label{randomizedTT:sec5}
Note that Problem \ref{RTT:prob2} is the fixed TT-rank problem for the approximate TT decomposition of a tensor, in which One key assumption is that the desired TT-rank $\{\mu_1,\mu_2,\dots,\mu_{N-1}\}$ is known. However, it is more common to encounter the situation where the desired TT-rank $\{\mu_1,\mu_2,\dots,\mu_{N-1}\}$ is unknown in advance. In this section, we focus on the fixed precision problem for the approximate TT decomposition of a tensor.

To be precise, for a given tensor $\mathcal{A}\in\mathbb{R}^{I_1\times I_2\times \dots \times I_N}$ and a computational tolerance $\epsilon>0$, we seek to determine a tensor $\mathcal{B}$ with low TT-rank such that
\begin{equation*}
    \|\mathcal{A}-\mathcal{B}\|_F\leq \epsilon\|\mathcal{A}\|_F.
\end{equation*}
If there exist $N$ tensors $\mathcal{Q}_n\in\mathbb{R}^{\mu_{n-1}\times I_n\times \mu_n}$ with $\mu_0=\mu_N=1$ such that
\begin{equation}
\label{RTT-eqn5}
    \|\mathcal{A}-\mathcal{Q}_1\times_{2}^1\mathcal{Q}_2\times_{3}^1\dots\times_{3}^1\mathcal{Q}_N\|_F\leq \epsilon\|\mathcal{A}\|_F,
\end{equation}
then we call the $(N-1)$-tuple $\{\mu_1,\mu_2,\dots,\mu_{N-1}\}$ as an $\epsilon$-TT-rank of $\mathcal{A}$.

For each $n$, let $\mathbf{A}_n\in\mathbb{R}^{\mu_{n-1}I_n\times I_{n+1}\dots I_n}$ be the same as in Section \ref{randomizedTT:sec3}. From (\ref{RTT:eqnadd1}), if there exists $\mathbf{Q}_n\in\mathbb{R}^{\mu_{n-1}I_n\times \mu_n}$ such that
\begin{equation}
\label{RTT-eqn6}
    \|\mathbf{A}_n-\mathbf{Q}_n\mathbf{Q}_n^\top\mathbf{A}_n\|_F\leq \frac{\epsilon}{\sqrt{N-1}}\|\mathcal{A}\|_F,
\end{equation}
then (\ref{RTT-eqn5}) is obtained. In this case, we call $\{\mu_1,\mu_2,\dots,\mu_{N-1}\}$ as an approximate $\epsilon$-TT-rank of $\mathcal{A}$. Note that TT-SVD and TT-cross are two common algorithms to estimate the $\epsilon$-TT-rank and the associated TT approximation.
\subsection{Greedy TT-SVD}
For each $n$, let $\widehat{I}_n=\min\{\prod_{k=1}^nI_k,\prod_{k={n+1}}^NI_k\}$. We now give a new strategy to estimate an $\epsilon$-TT-rank with a given tolerance $\epsilon\in(0,1)$, which is summarized in Algorithm \ref{RTT:alggreedy-tt}.
\begin{algorithm}
\caption{Greedy TT-SVD}
\label{RTT:alggreedy-tt}
\begin{algorithmic}[1]
\STATEx {\bf Input:} $\mathcal{A}\in \mathbb{R}^{I_1\times I_2\times \dots \times I_N}$ and $0<\epsilon<1$.
\STATEx {\bf Output:} The $(N-1)$-tuple $\{\mu_1,\mu_2,\dots,\mu_{N-1}\}$.
\STATE Compute $\delta=\frac{\epsilon}{\sqrt{N-1}}\|\mathcal{A}\|_F$.
\STATE Set $\mu_n=1$ for $n=1,2,\dots,N-1$.
\WHILE{$\sum_{n=1}^{N-1}\sum_{i_n=\mu_n+1}^{\widehat{I}_n}\sigma_{i_n}(\mathbf{A}_{([n])})^2\geq \delta^2$}
\STATE Select $1\leq j_0\leq N-1$ such that
\begin{equation*}
j_0={\rm argmax}_{1\leq j\leq N-1}\sigma_{\mu_j}(\mathbf{A}_{([j])}).
\end{equation*}
\STATE Set $\mu_{j_0}=\mu_{j_0}+1$.
\ENDWHILE
\STATE Return $\{\mu_1,\mu_2,\dots,\mu_{N-1}\}$.
\end{algorithmic}
\end{algorithm}

It is obvious to see that for a given $\epsilon>0$, each entry in the $\epsilon$-TT-rank obtained by Algorithm \ref{RTT:alggreedy-tt} is larger than the corresponding one obtained by TT-SVD. Note that if there exist two distinct positive integers $j_0'$ and $j_0''$ such that
\begin{equation*}
\sigma_{\mu_{j_0'}}(\mathbf{A}_{([j_0'])})=\sigma_{\mu_{j_0''}}(\mathbf{A}_{([j_0''])})=\max_{1\leq j\leq N-1}\sigma_{\mu_j}(\mathbf{A}_{([j])}),
\end{equation*}
we set $j_0=\min\{j_0',j_0''\}$.
\begin{remark}
The framework in Algorithm {\rm \ref{RTT:alggreedy-tt}} is similar to that in Ehrlacher {\it et al.} {\rm\cite[Algorithm 3.1]{ehrlacher2021adaptive}}, which investigated an adaptive compression method for tensors based on high-order singular value decomposition.
\end{remark}

\subsection{An efficient adaptive randomized algorithm}
Following an adaptive randomized range finder in \cite{halko2011finding}, Che and Wei \cite{che2019randomized} proposed an adaptive randomized algorithm (denoted by Adap-rand-TT) for finding an approximation of the TT decomposition with a given $\epsilon>0$. Note that for each $n$, the random vector $\mathbf{g}\in\mathbb{R}^{I_{n+1}\dots I_N}$ used in Adap-rand-TT is given by $\mathbf{g}=\mathbf{g}_{n+1}\odot\dots\odot\mathbf{g}_N$, where all the vectors $\mathbf{g}_m\in\mathbb{R}^{I_m}$ are standard Gaussian vectors with $m=n+1,n+2,\dots,N$. For this case, we denote Adap-rand-TT by Adap-rand-TT with KR-Gaussian. Similarly, when the random vector $\mathbf{g}\in\mathbb{R}^{I_{n+1}\dots I_N}$ used in Adap-rand-TT is a standard Gaussian vector, this type of Adap-rand-TT is denoted by Adap-rand-TT with KR-Gaussian.
\begin{remark}
As we know, the random vector $\mathbf{g}\in\mathbb{R}^{I_{n+1}\dots I_N}$ used in the adaptive randomized range finder {\rm\cite{halko2011finding}} can be also replaced by $\mathbf{g}=\mathbf{g}_{1}\odot\dots\odot\mathbf{g}_m$, where all the vectors $\mathbf{g}_i\in\mathbb{R}^{I_i}$ are standard Gaussian vectors with $i=1,2,\dots,m$. Unfortunately, it is unclear that there exists a common way to choose all the $I_i$.
\end{remark}

Note that the error estimator for the adaptive randomized range finder \cite{halko2011finding} often overestimates the approximation error, yielding much larger output matrices than what is necessary. To overcome this issue, several more efficient adaptive randomized range finders for matrices have been proposed in \cite{martinsson2016randomized,yu2018efficient}. For a given matrix $\mathbf{A}\in\mathbb{R}^{I\times J}$ and a computational tolerance $\epsilon>0$, we use Algorithm 2 in \cite{yu2018efficient} and assume that it can be invoked as $$\mathbf{Q}={\rm AdaptRangeFinder}(\mathbf{A},\epsilon,b),$$ where $b$ is a blocking integer to determine how many columns of the random matrix $\mathbf{\Omega}\in\mathbb{R}^{J\times b}$ to draw at a time. The number of columns of $\mathbf{Q}$ is taken to the rank of $\mathbf{Q}\mathbf{Q}^\top\mathbf{A}$.
\begin{remark}
As shown in \cite{yu2018efficient}, the matrix $\mathbf{\Omega}\in\mathbb{R}^{J\times b}$ is a standard Gaussian matrix. Meanwhile, the matrix $\mathbf{\Omega}$ can be given by $\mathbf{\Omega}=\mathbf{\Omega}_1\odot\dots\odot\mathbf{\Omega}_m$, where all the matrices $\mathbf{\Omega}_k\in\mathbb{R}^{J_k\times b}$ are standard Gaussian matrices with $J_1\dots J_m=J$. A key issue for this case is that it is not easy to determine $\{J_1,J_2,\dots,J_m\}$ from the given integer $J$.
\end{remark}

For a given tolerance $\epsilon>0$, by using Algorithm 2 in \cite{yu2018efficient} to find $\mathbf{Q}_n$ such that (\ref{RTT-eqn6}) holds, we obtain a new adaptive randomized algorithm for approximating the TT decomposition of $\mathcal{A}\in\mathbb{R}^{I_1\times I_2\times \dots \times I_N}$, which is summarized in Algorithm \ref{RTT:algadapt}.

\begin{algorithm}
\caption{The adaptive randomized algorithm for TT-SVD}
\label{RTT:algadapt}
\begin{algorithmic}[1]
\STATEx {\bf Input:} $\mathcal{A}\in \mathbb{R}^{I_1\times I_2\times \dots \times I_N}$, $0<\epsilon<1$, the blocking integer $b\geq 1$ and the number of subspace iterations $q\geq 0$.
\STATEx {\bf Output:} An approximation of the TT decomposition of $\mathcal{A}$: $\mathcal{B}=\mathcal{Q}_1\times_2^1\mathcal{Q}_2\times_3^1\dots\times_3^1\mathcal{Q}_N$ and $(N-1)$ positive integers $\{\mu_1,\mu_2,\dots,\mu_{N-1}\}$.
\STATE Let $\mu_0=1$ and compute $\delta=\frac{\epsilon}{\sqrt{N-1}}\|\mathcal{A}\|_F$.
\STATE Set $\mathbf{A}_0=\mathbf{A}_{(1)}$.
\FOR {$n=1,2,\ldots,N-1$}
\STATE Reshape $\mathbf{A}_{n-1}$ as a matrix $\mathbf{A}_n\in\mathbb{R}^{\mu_{n-1}I_n\times \prod_{m=n+1}^NI_m}$.
\STATE Compute $\mathbf{Q}_n={\rm AdaptRangeFinder}(\mathbf{A}_n,\delta,b,q)$.
\STATE  Set $\mu_n$ being the number of the columns of $\mathbf{Q}_n$.
\STATE Reshape $\mathbf{Q}_n$ as a tensor  $\mathcal{Q}_n\in\mathbb{R}^{\mu_{n-1}\times I_n\times \mu_n}$ and update ${\bf A}_{n}={\bf \Sigma}{\bf V}^\top$.
\ENDFOR
\STATE Reshape $\mathbf{A}_{N-1}$ as a tensor  $\mathcal{Q}_N\in\mathbb{R}^{\mu_{N-1}\times I_N}$.
\STATE Form $\mathcal{B}=\mathcal{Q}_1\times_2^1\mathcal{Q}_2\times_3^1\dots\times_3^1\mathcal{Q}_N$.
\end{algorithmic}
\end{algorithm}

By utilizing the structure of $\mathcal{A}$, for each $n$, there exist two ways to generate the matrix $\mathbf{\Omega}_n\in\mathbb{R}^{I_{n+1}\dots I_N\times b}$, which is used in the ``AdaptRangeFinder'' function:
\begin{enumerate}
\item[(a)] the matrix $\mathbf{\Omega}_n$ is a standard Gaussian matrix; and
\item[(b)] the matrix $\mathbf{\Omega}_n$ is given by $\mathbf{\Omega}_n=\mathbf{\Omega}_{n,n+1}\odot\dots\odot\mathbf{\Omega}_{n,N}$, where all the matrices $\mathbf{\Omega}_{n,m}\in\mathbb{R}^{I_m\times b}$ are standard Gaussian matrices with $m=n+1,n+2,\dots,I_N$.
\end{enumerate}

\begin{remark}
Note that in Algorithm \ref{RTT:algadapt}, for each $n$, the blocking integer $b$ may be not the same. For this case, the matrix $\mathbf{Q}_n$ is computed by $\mathbf{Q}_n={\rm AdaptRangeFinder}(\mathbf{A}_n,\delta,b_n,q)$.
\end{remark}

Following the work in \cite{yu2018efficient}, the computational complexity of Algorithm \ref{RTT:algadapt} can be easily obtained and we omit it.

We now consider the accuracy of $\|\mathcal{A}-\mathcal{Q}_1\times_{2}^1\mathcal{Q}_2\times_{3}^1\dots\times_{3}^1\mathcal{Q}_N\|_F^2$ in floating point arithmetic. In floating point arithmetic, the machine precision $\epsilon_{{\rm mach}}$ characterizes the maximum relative error of converting a real number to its floating-point representation, that is, for any $x\in\mathbb{R}$, we have
\begin{equation*}
    |fl_r(x)|\leq \epsilon_{{\rm mach}}.
\end{equation*}

For clarity, let $E=\|\mathcal{A}-\mathcal{Q}_1\times_{2}^1\mathcal{Q}_2\times_{3}^1\dots\times_{3}^1\mathcal{Q}_N\|_F^2$. For each term in the right-hand side of (\ref{RTT:eqnadd1}), we have
\begin{equation*}
    \begin{split}
        \|\mathbf{A}_n-\mathbf{Q}_n\mathbf{Q}_n^\top\mathbf{A}_n\|_F^2
        &={\rm Tr}((\mathbf{A}_n-\mathbf{Q}_n\mathbf{Q}_n^\top\mathbf{A}_n)^\top(\mathbf{A}_n-\mathbf{Q}_n\mathbf{Q}_n^\top\mathbf{A}_n))\\
        &={\rm Tr}(\mathbf{A}_n^\top\mathbf{A}_n-2\mathbf{A}_n^\top\mathbf{Q}_n\mathbf{Q}_n^\top\mathbf{A}_n+\mathbf{A}_n^\top\mathbf{Q}_n\mathbf{Q}_n^\top\mathbf{Q}_n\mathbf{Q}_n^\top\mathbf{A}_n)\\
        &={\rm Tr}(\mathbf{A}_n^\top\mathbf{A}_n-\mathbf{A}_n^\top\mathbf{Q}_n\mathbf{Q}_n^\top\mathbf{A}_n)\\
        &=\|\mathbf{A}_n\|_F^2-\|\mathbf{Q}_n^\top\mathbf{A}_n\|_F^2,
    \end{split}
\end{equation*}
which implies that
\begin{equation*}
    E\leq \sum_{n=1}^{N-1}(\|\mathbf{A}_n\|_F^2-\|\mathbf{Q}_n^\top\mathbf{A}_n\|_F^2).
\end{equation*}
According to the definition of the Frobenius norm of matrices, the relative error of $\|\mathbf{A}_n\|_F^2$ is bounded by $2\epsilon_{{\rm mach}}$. Note that for each $n$, we have $\|\mathbf{Q}_n^\top\mathbf{A}_n\|_F\leq\|\mathbf{A}_n\|_F\leq \|\mathcal{A}\|_F$. Hence, we have
\begin{equation*}
    |fl_r(E)|\leq \sum_{n=1}^{N-1}(|fl_r(\|\mathbf{A}_n\|_F^2)|+|fl_r(\|\mathbf{Q}_n^\top\mathbf{A}_n\|_F^2)|)\leq 4(N-1)\epsilon_{{\rm mach}}\|\mathcal{A}\|_F^2.
\end{equation*}
If we want to guarantee that $|fl_r(E)|\leq \delta E$, we shall enforce $4(N-1)\epsilon_{{\rm mach}}\|\mathcal{A}\|_F^2\leq \delta E$. This means
\begin{equation*}
    \epsilon>\sqrt{E}\geq \sqrt{\frac{4(N-1)\epsilon_{{\rm mach}}}{\delta}}\|\mathcal{A}\|_F.
\end{equation*}
So, we obtain the following theorem for the accuracy of $\|\mathcal{A}-\mathcal{Q}_1\times_{2}^1\mathcal{Q}_2\times_{3}^1\dots\times_{3}^1\mathcal{Q}_N\|_F^2$ in floating point arithmetic.
\begin{theorem}
Suppose that the tensor $\mathcal{A}$ and the tolerance $\epsilon$ are the input to Algorithm \ref{RTT:algadapt}. If $\epsilon>\sqrt{\frac{4(N-1)\epsilon_{{\rm mach}}}{\delta}}\|\mathcal{A}\|_F$, the relative error of $\|\mathcal{A}-\mathcal{Q}_1\times_{2}^1\mathcal{Q}_2\times_{3}^1\dots\times_{3}^1\mathcal{Q}_N\|_F^2$ must be no greater than $\delta$.
\end{theorem}

\section{Numerical examples}
\label{randomizedTT:sec6}
We implement our algorithms on a laptop with Intel Core i5-4200M CPU (2.50GHz) and 8GB RAM. The codes are implemented by MATLAB R2016b. Some basic tensor operations are implemented in the MATLAB Tensor Toolbox \cite{bader2017matlab}. For a given TT-rank, we compare the efficiencies of TT-SVD \cite{oseledets2011tensor}, randomized TT-SVD \cite{huber2017a}, PSTT \cite{shi2023parallel}, Algorithm \ref{RTT:alg3} with Gaussian, Algorithm \ref{RTT:alg3} with KR-Gaussian and Algorithm \ref{RTT:alg2} to calculate an approximation of the TT decomposition of a tensor through some examples. For a given tolerance $\epsilon>0$, we compare the efficiency and precision of TT-SVD, TT-cross, Adap-TT \cite{che2019randomized}, Greedy TT-SVD (Algorithm \ref{RTT:alggreedy-tt}) and Algorithm \ref{RTT:algadapt} to estimate the $\epsilon$-TT-rank and its corresponding TT decomposition. Note that TT-SVD and TT-cross are implemented by {\bf tt\_tensor} and {\bf dmrg\_cross} in the MATLAB Toolbox for working with high-dimensional tensors in the TT-format \cite{oseledets2012tt}. For reproducibility, we have shared the Matlab codes of the proposed algorithms on https://github.com/chncml/randomized-TT.

For a given TT-rank $\{\mu_1,\mu_2,\dots,\mu_{N-1}\}$ and $\mu_0=\mu_N=1$, we define the relative error for  $\mathcal{B}=\mathcal{Q}_1\times_3^1\mathcal{Q}_2\times_3^1\dots
\times_3^1\mathcal{Q}_N$ of $\mathcal{A}\in\mathbb{R}^{I_1\times I_2\times\dots\times I_N}$ as
\begin{equation}\label{RTT:eqn2}
{\rm RE}=\|\mathcal{A}-\mathcal{B}\|_F/\|\mathcal{A}\|_F,
\end{equation}
where all the $\mathcal{Q}_n\in\mathbb{R}^{\mu_{n-1}\times I_n\times \mu_n}$ are derived from the proposed algorithms. The FIT value to approximate the tensor $\mathcal{A}$ is defined by ${\rm FIT}=1-{\rm RE}$.
\subsection{The test tensors}
\label{randomizedTT:sec6:1}
In this section, we introduce several test tensors from synthetic and real databases, which will be used in the following sections.
\begin{example}
\label{RTT-exm1}
Similar to {\rm\cite{sun2020low-rank}}, the tensor $\mathcal{A}\in\mathbb{R}^{50\times 50\times 50\times 50\times 50}$ is given as
\begin{equation*}
\mathcal{A}=\mathcal{P}+\frac{\gamma\|\mathcal{P}\|_F}{\sqrt{50^5}}\mathcal{N}
\end{equation*}
with $\mathcal{P}=\mathcal{G}_1\times_2^1\mathcal{G}_2\times_3^1\mathcal{G}_3\times_3^1\mathcal{G}_4\times_3^1\mathcal{G}_5$, where the entries of $\mathcal{G}_1\in\mathbb{R}^{50\times 20}$, $\mathcal{G}_2,\mathcal{G}_3,\mathcal{G}_4\in\mathbb{R}^{20\times 50\times 20}$ and $\mathcal{G}_5\in\mathbb{R}^{20\times 50}$ are Gaussian variables with mean zero and variance unit, and the entries of $\mathcal{N}\in\mathbb{R}^{50\times 50\times 50\times 50\times 50}$ are an independent family of standard normal random variables. Here, we set $\gamma=10^{-4}$.
\end{example}
\begin{example}
\label{RTT-exm2}
We define $\mathcal{B}=\mathcal{P}+\beta \mathcal{N}$, where $\mathcal{P}$ is given in Example \ref{RTT-exm1} and $\mathcal{N}\in\mathbb{R}^{50\times 50\times 50\times 50\times 50}$ is an unstructured perturbation tensor with different noise levels $\beta$. The following signal-to-noise ratio (SNR) measure will be used:
\begin{equation*}
{\rm SNR}\ [{\rm dB}]=10\log\left(\frac{\|\mathcal{B}\|_F^2}
{\|\beta\mathcal{N}\|_F^2}\right).
\end{equation*}
\end{example}
\begin{example}
\label{RTT-exm3}
The following three tensors are obtained by discretizing three smooth functions, respectively, as follows:
\begin{equation*}
\begin{split}
\mathcal{C}(i_1,i_2,i_3,i_4,i_5)&=\sin\sqrt{\sum_{k=1}^5\left(\frac{i_k-1}{I-1}\right)^2},\\
\mathcal{D}(i_1,i_2,i_3,i_4,i_5)&=\frac{I-1}{I+i_1+i_2+i_3+i_4+i_5}
\end{split}
\end{equation*}
with $i_1,i_2,i_3,i_4,i_5=1,2,\dots,40$.
The type of $\mathcal{C}$ comes from \cite{lestandi2021numerical}, and the type of $\mathcal{D}$ is chosen from \cite{caiafa2010generating}.
\end{example}
\begin{example}
    \label{RTT-exm4}
    The Extended Yale Face database B\footnote{\url{http://vision.ucsd.edu/~iskwak/ExtYaleDatabase/ExtYaleB.html}.} contains 5760 single light source images of 10 subjects each seen under 576 viewing conditions (9 poses x 64 illumination conditions). This database has $560$ images that contain the first $12$ possible illuminations of $28$ different people under the first pose, where each image has $480\times 640$ pixels in a grayscale range. This is resized into a $451\times 602\times 540$ tensor $\mathcal{A}_{\rm YaleB-3D}$ with real numbers. We also reshape $\mathcal{A}_{\rm YaleB-3D}$ as another tensor $\mathcal{A}_{\rm YaleB-5D}$ of size $41\times 42\times 43\times 44\times 45$.

    We adopt a hyperspectral image (HSIs) as the second test data. This is a subimage of the Washington DC Mall database\footnote{\url{https://engineering.purdue.edu/biehl/MultiSpec/hyperspectral.html}.} of the size $1280\times 307\times 191$. We then generate two tensors $\mathcal{A}_{{\rm DCmall-3D}}\in\mathbb{R}^{1280\times 307\times 191}$ and $\mathcal{A}_{{\rm DCmall-5D}}\in\mathbb{R}^{35\times 37\times 38\times 39\times 39}$ from this database.

    The third database is the Columbia object image library COIL-100\footnote{COIL-100 can be downloaded from \url{https://www.cs.columbia.edu/CAVE/software/softlib/coil-100.php}.} \cite{nene1996columbia}, which consists of $7200$ color images that contain 100 objects under 72 different rotations. Each image has $128\times 128$ pixels in the RGB color range. We reshape this data to a third-order tensor $\mathcal{A}_{{\rm COIL-3D}}$ of size $1024\times 1152\times 300$. We also reshape $\mathcal{A}_{{\rm COIL-3D}}$ as another tensor $\mathcal{A}_{{\rm COIL-5D}}$ of size $64\times  64\times 50\times 72\times 24$.
\end{example}
\begin{figure}[htb]
\centering
\begin{tabular}{c}
\includegraphics[width=5in, height=1.8in]{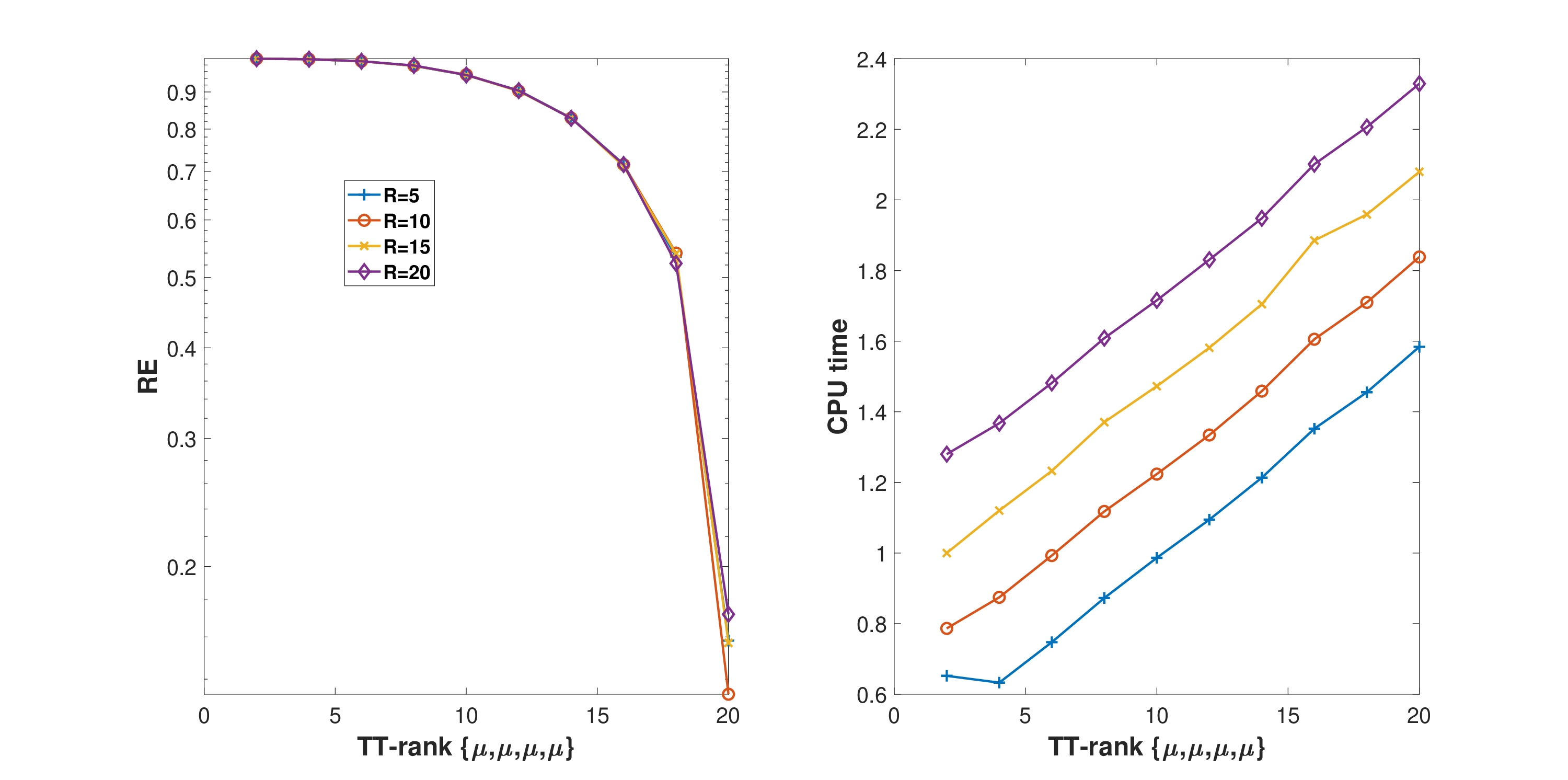}\\
\includegraphics[width=5in, height=1.8in]{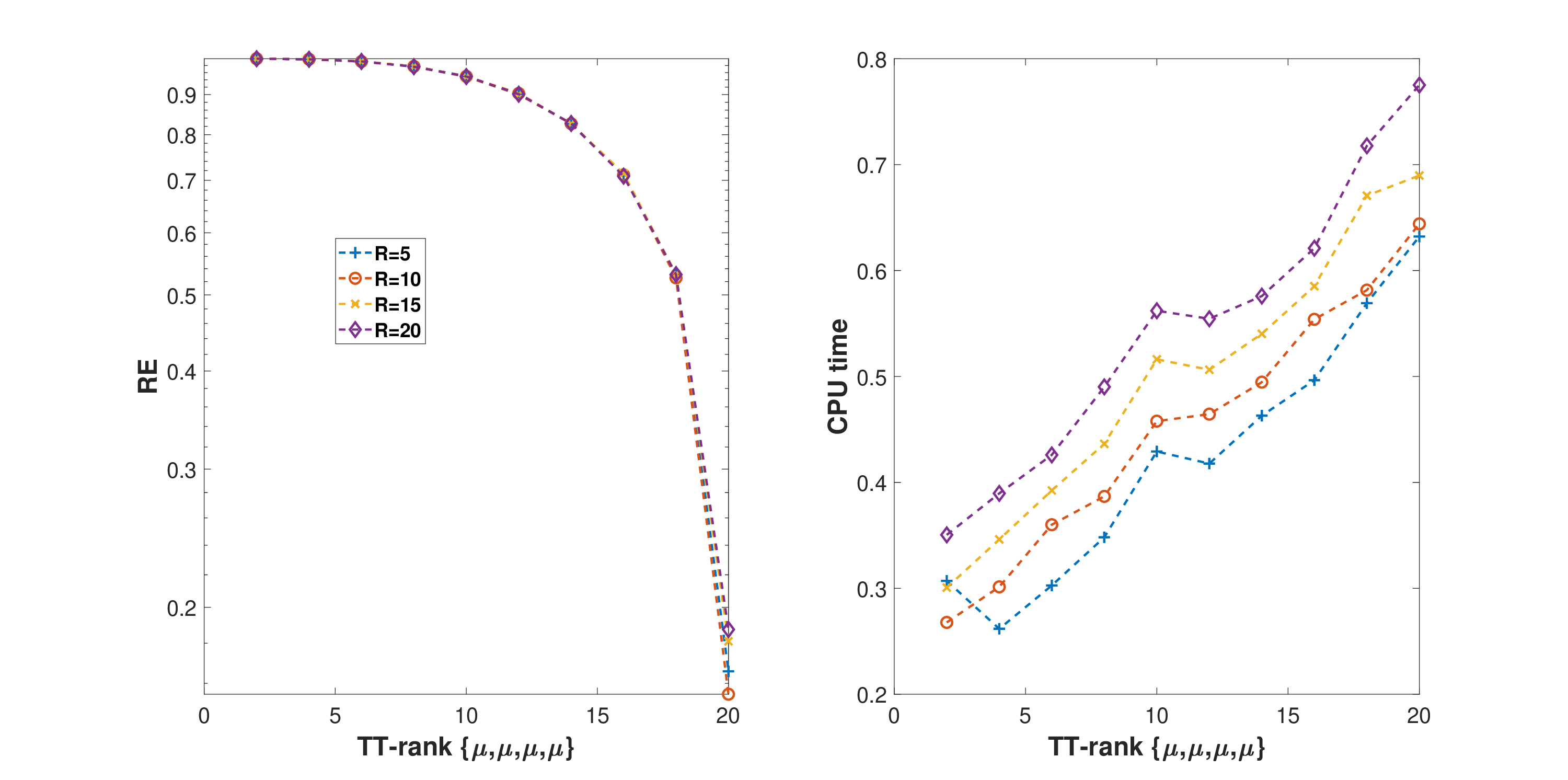}\\
\includegraphics[width=5in, height=1.8in]{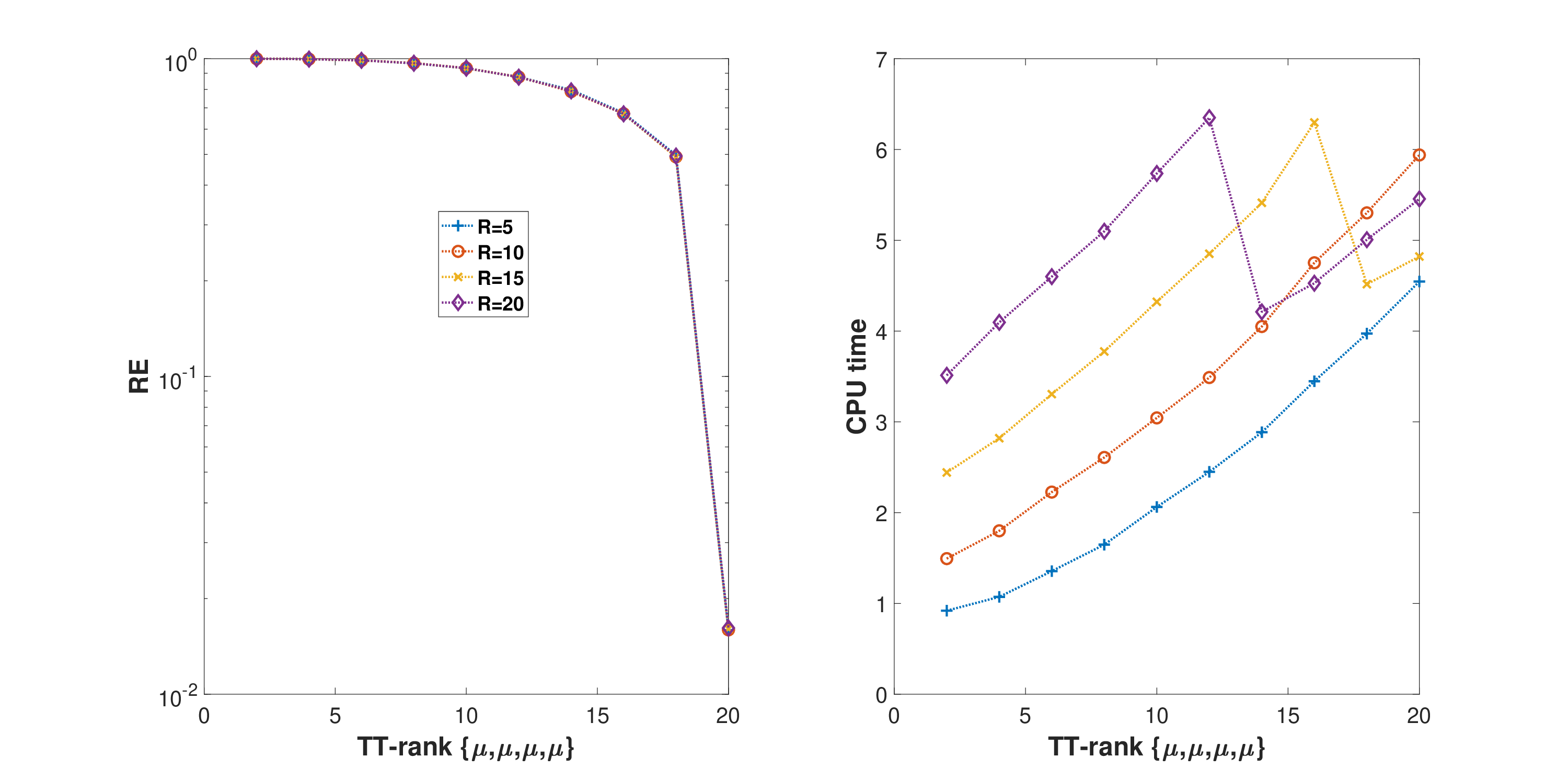}\\
\end{tabular}
\caption{For different oversampling parameters $R$, the results by applying Algorithm \ref{RTT:alg3} with Gaussian (top row), Algorithm \ref{RTT:alg3} with KR-Gaussian (middle row) and Algorithm \ref{RTT:alg2} (bottom row) to $\mathcal{A}$ with $\mu=2,4,6,\dots,20$.}\label{RTT:figadd1}
\end{figure}
\subsection{The choices of $R$ and $q$}
\label{RTT:sec6-1}
Note that there exist two parameters $R$ and $q$ in Algorithms \ref{RTT:alg3} and \ref{RTT:alg2}. From \cite{halko2011finding}, the value of the oversampling parameter $R$ is set to 10, and $q=1,2$ are enough to achieve a good accuracy. in this section, we consider the values of $R$ and $q$ via a test tensor $\mathcal{A}$, given in Example \ref{RTT-exm1}.

We set $q=0$ in Algorithm \ref{RTT:alg3} and $q=1$ in Algorithm \ref{RTT:alg2}. For $R=5,10,15,20$, when we apply Algorithms \ref{RTT:alg3} and \ref{RTT:alg2} with different TT-ranks $\{\mu,\mu,\mu,\mu\}$, the values of RE and CPU time are shown in Figure \ref{RTT:figadd1}. It is shown that for each TT-rank, Algorithms \ref{RTT:alg3} and \ref{RTT:alg2} with $R=5$ is worse than that with $R=10,15,20$, Algorithms \ref{RTT:alg3} and \ref{RTT:alg2} with $R=10,15,20$ are similar, and as $R$ increases, CPU time of Algorithms \ref{RTT:alg3} and \ref{RTT:alg2} also increases. Hence, in the rest, we set $R=10$, which is consistent with the case of randomized algorithms for low-rank approximations \cite{halko2011finding}.

In the case of $R=10$, for different $q$, using Algorithms \ref{RTT:alg3} and \ref{RTT:alg2} with different TT-ranks $\{\mu,\mu,\mu,\mu\}$, Figure \ref{RTT-fig2} illustrates the RE and CPU time values. From this figure, in terms of RE, Algorithm \ref{RTT:alg3} with $q=1,2,3$ are similar and better than that with $q=0$, and Algorithm \ref{RTT:alg2} with $q=1,2,3$ are similar. Hence, in the rest, we set $q=0,1$ for Algorithm \ref{RTT:alg3} and $q=1,2$ for Algorithm \ref{RTT:alg2}.

\begin{figure}[htb]
\centering
\begin{tabular}{c}
\includegraphics[width=5in, height=1.8in]{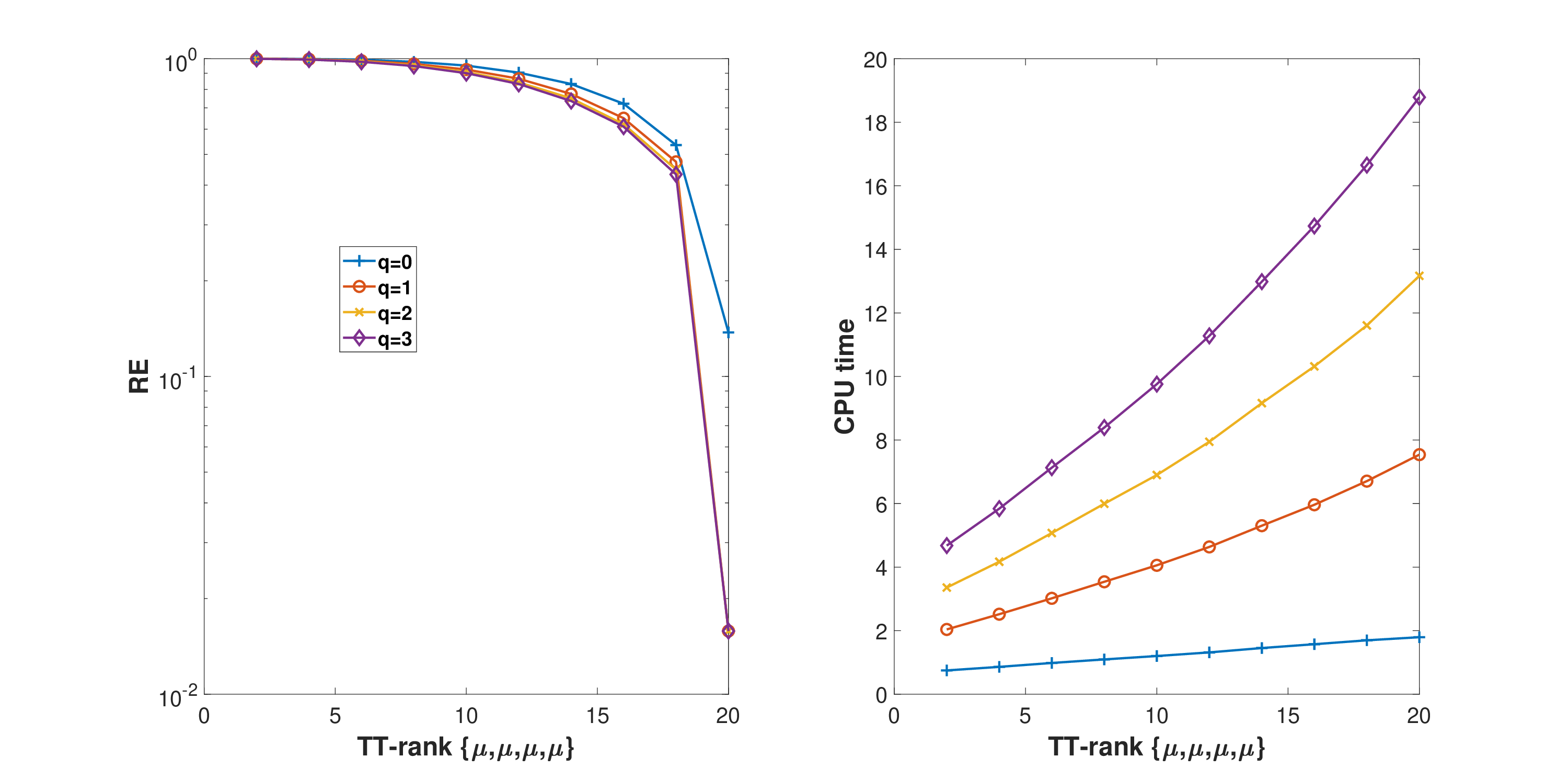}\\
\includegraphics[width=5in, height=1.8in]{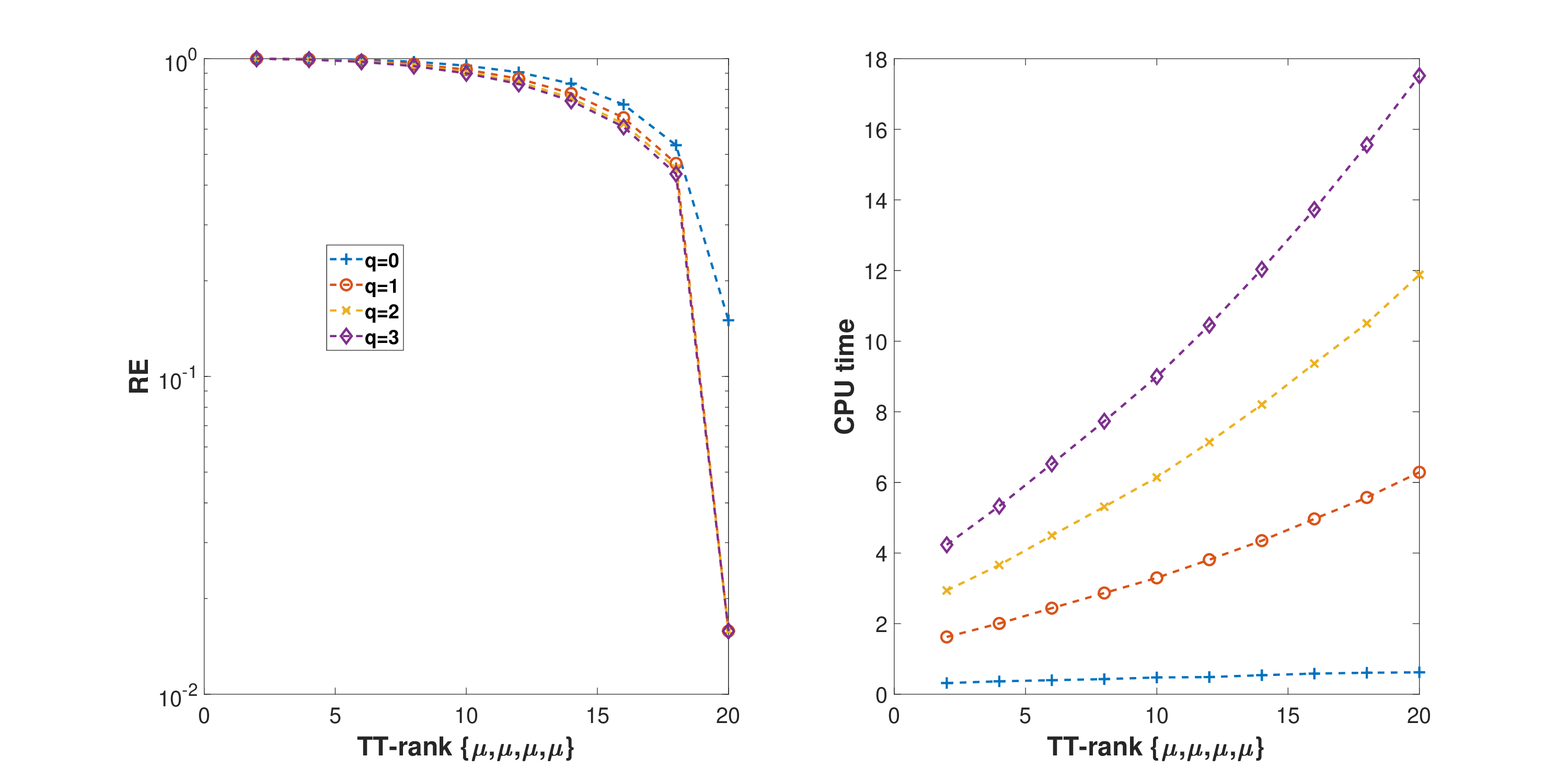}\\
\includegraphics[width=5in, height=1.8in]{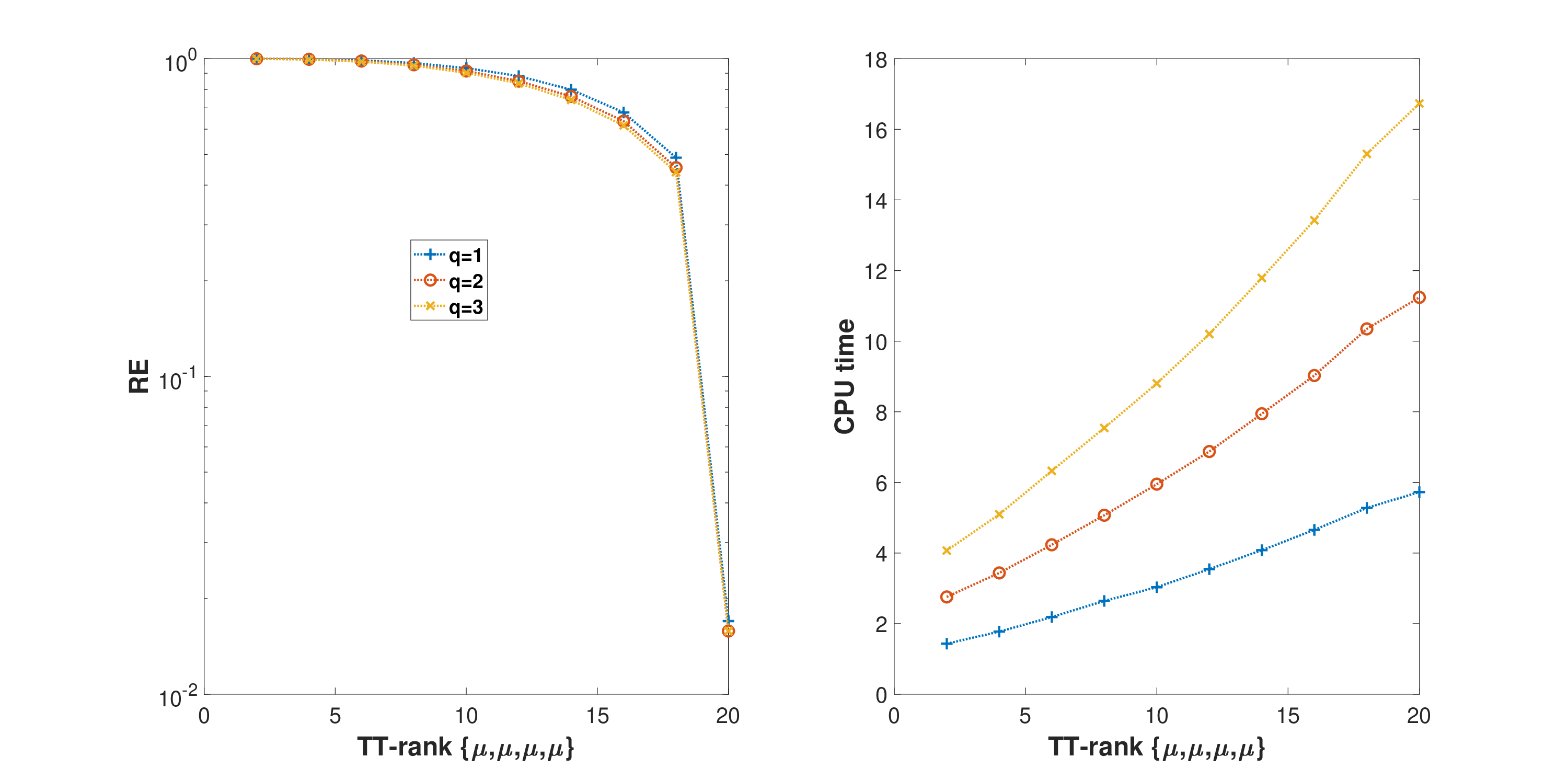}\\
\end{tabular}
\caption{For different $q$, the results by applying Algorithm \ref{RTT:alg3} with Gaussian (top row), Algorithm \ref{RTT:alg3} with KR-Gaussian (middle row) and Algorithm \ref{RTT:alg2} (bottom row) to $\mathcal{A}$ with $\mu=2,4,6,\dots,20$.}\label{RTT-fig2}
\end{figure}

\subsection{Other choices of $\mathbf{G}_n$ in Algorithm \ref{RTT:alg3}}
\label{randomizedTT:sec6:subsec3}

Similar to the work in \cite{che2020the,che2021randomized}, the matrix $\mathbf{\Omega}_n$ in Algorithm \ref{RTT:alg3} can be also generated by $\mathbf{\Omega}_n=\mathbf{\Omega}_{n,n+1}\otimes\dots\otimes\mathbf{\Omega}_{n,N}$, where all the matrices $\mathbf{\Omega}_{n,m}\in\mathbb{R}^{I_m\times L_{n,m}}$ are the standard Gaussian matrices with $L_{n,n+1}\dots L_{n,N}\geq \mu_n+R$. Under this case, Algorithm \ref{RTT:alg3} is denoted by {\it Algorithm \ref{RTT:alg3} with Kron-Gaussian}. Unfortunately, it is unclear that there exists a common way to choose these $L_{n,m}$. Here, one way to choose these $L_{n,m}$ is that: we set $L_{n,n+1}=L_{n,n+2}=\dots=L_{n,N}={\rm ceil}((\mu_n+R)^{n-N})$, where ${\rm ceil}(x)$ rounds the value of $x\in\mathbb{R}$ to the nearest integer towards positive infinity. For this case, we have $L_{n,n+1}\dots L_{n,N}\geq\mu_n+R$ and set $\mathbf{\Omega}_n=\mathbf{\Omega}_n(:,1:\mu_n+R)$.

Similarly to Algorithm \ref{RTT:alg3} with Gaussian, when we replace the standard Gaussian matrices in Algorithm \ref{RTT:alg3} by the SpEmb and SDCT matrices, the related algorithms are, respectively, indicated by {\it Algorithm \ref{RTT:alg3} with SpEmb} and {\it Algorithm \ref{RTT:alg3} with SDCT}. Since the SRFT matrices are suitable for complex-valued matrices and tensors, and for the SRHT matrices, the dimension is restricted to a power of two, we do not consider the case of using the SRFT and SRHT matrices to replace the standard Gaussian matrices in Algorithm \ref{RTT:alg3}.

\begin{figure}[htb]
\centering
\begin{tabular}{c}
\includegraphics[width=5in, height=1.8in]{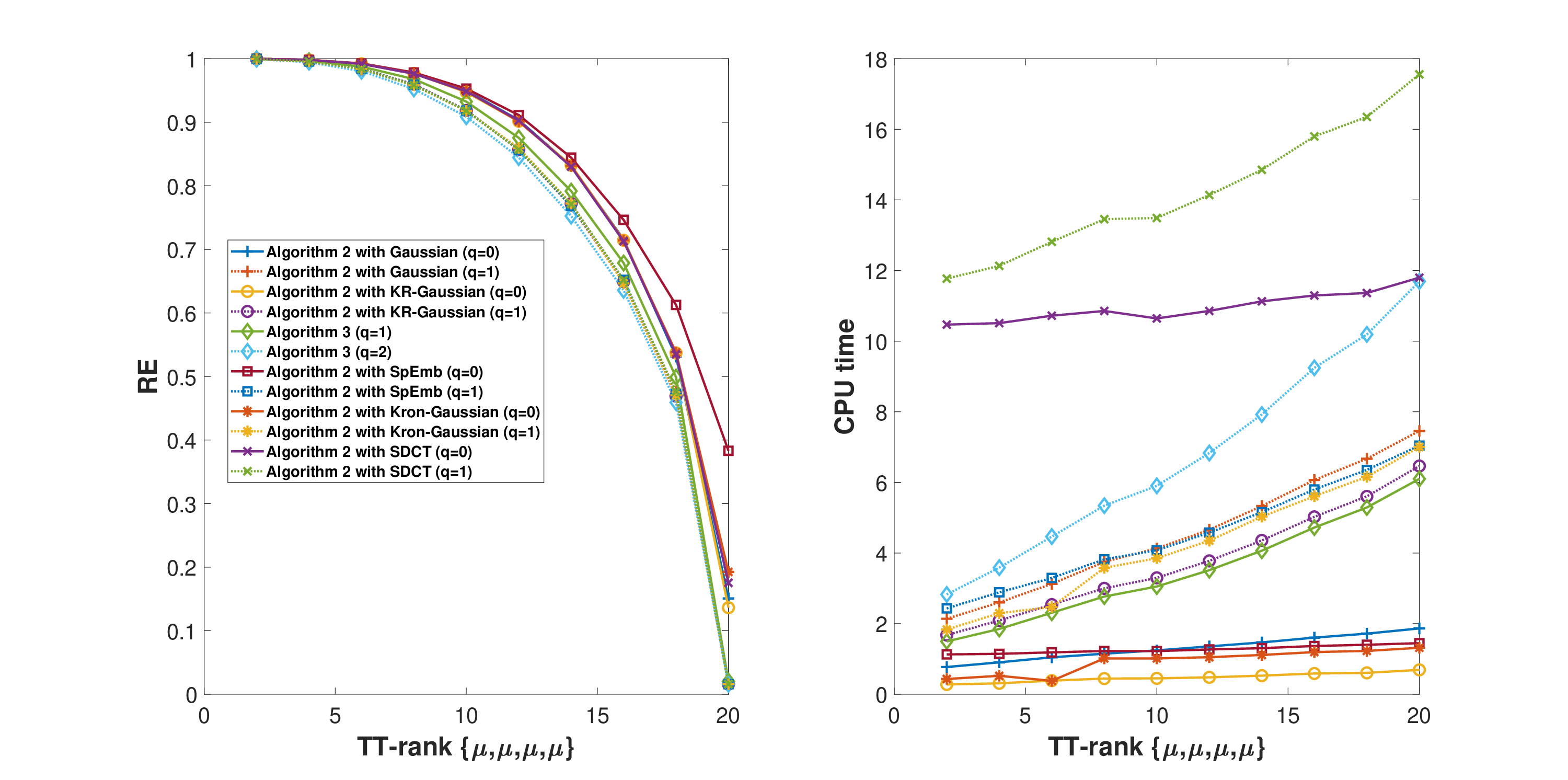}\\
\includegraphics[width=5in, height=1.8in]{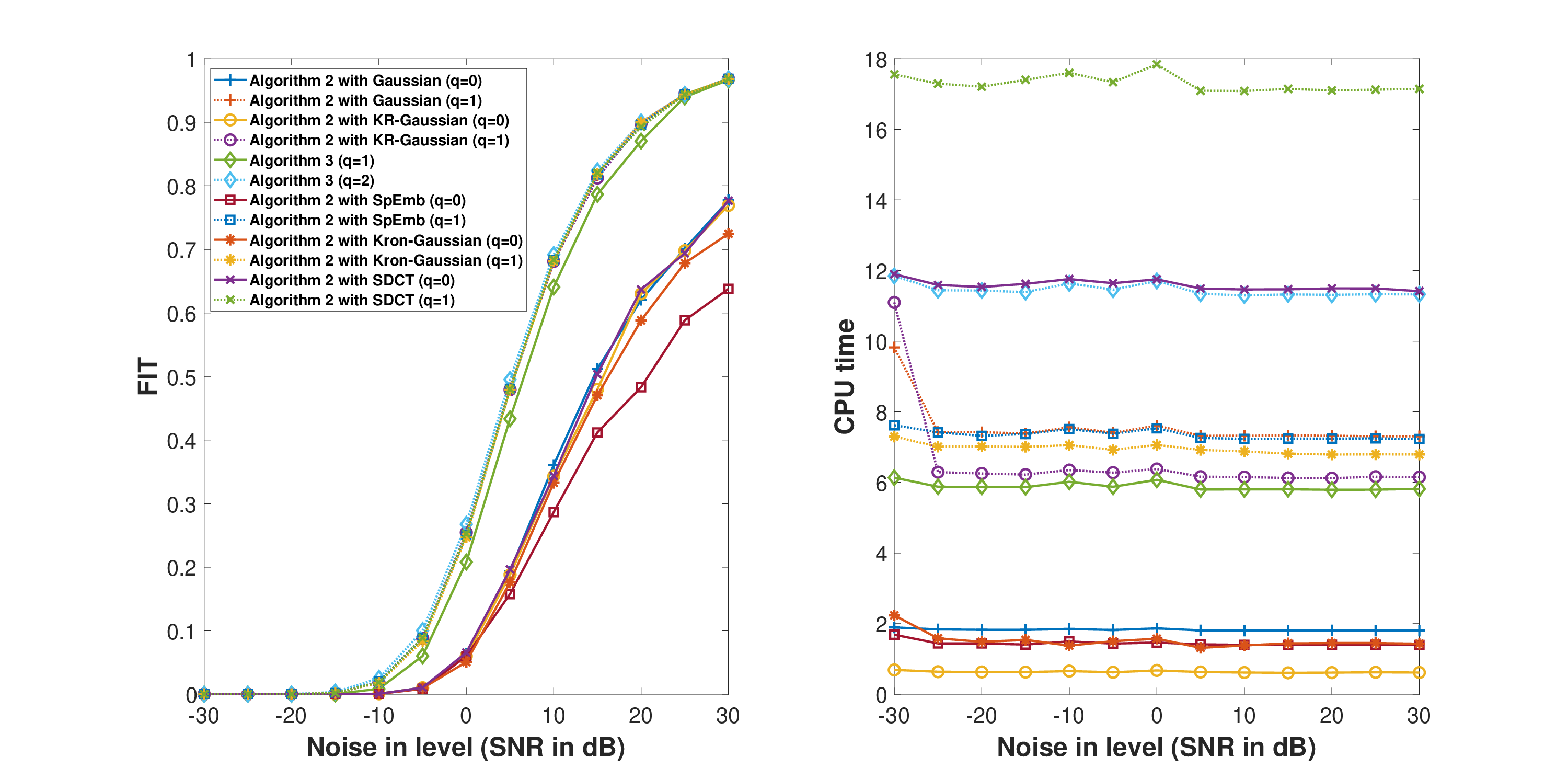}
\end{tabular}
\caption{The results by applying Algorithm \ref{RTT:alg3} with Gaussian, KR-Gaussian, Kron-Gaussian, SpEmb and SDCT, and Algorithm \ref{RTT:alg2} to $\mathcal{A}$ (top row) with $\mu=2,4,6,\dots,20$, and $\mathcal{B}$ (bottom row) with ${\rm SNR}= -30,-25,\dots,25,30$.}\label{RTT-fig3}
\end{figure}

For two test tensors $\mathcal{A}$ and $\mathcal{B}$, given in Example \ref{RTT-exm1} and \ref{RTT-exm2}, respectively, we now compare the accuracy and efficiency of Algorithm \ref{RTT:alg3} with Gaussian, KR-Gaussian, Kron-Gaussian, SpEmb and SDCT, and Algorithm \ref{RTT:alg2}. For the tensor $\mathcal{A}$, its desired TT-rank is denoted by $\{\mu,\mu,\mu,\mu\}$ and the values of $\mu$ are chosen from $\{2,4,\dots,20\}$. For the tensor $\mathcal{B}$, the desired TT-rank $\{\mu,\mu,\mu,\mu\}$ is set to $\{20,20,20,20\}$ and the values of {\rm SNR} are chosen from $\{-30,-25,\dots,25,30\}$.

By applying Algorithms \ref{RTT:alg3} and \ref{RTT:alg2} to $\mathcal{A}$ with different $\mu$ and $\mathcal{B}$ with different SNR, the related results are shown in Figure \ref{RTT-fig3}. It is shown that in terms of RE and FIT, for $q=0$, Algorithm \ref{RTT:alg3} with SpEmb is the worst one, Algorithm \ref{RTT:alg2} with $q=1$ is comparable with all the cases of Algorithm \ref{RTT:alg3} with $q=1$ and Algorithm \ref{RTT:alg2} with $q=2$, and Algorithm \ref{RTT:alg2} with $q=1$ is better than all the cases of Algorithm \ref{RTT:alg3} with $q=0$; in terms of CPU time, Algorithm \ref{RTT:alg3} with KR-Gaussian and $q=0$ is the fastest one, for Algorithm \ref{RTT:alg3}, the case of $q=0$ is faster than the case of $q=1$, and for $q=1$, Algorithm \ref{RTT:alg2} is faster than all the cases of Algorithm \ref{RTT:alg3}.

\subsection{Robustness for the proposed algorithms}
\label{randomizedTT:sec6:4}
When $\mathcal{Q}_n\in\mathbb{R}^{\mu_{n-1}\times I_n\times \mu_n}$ are obtained by applying Algorithms \ref{RTT:alg3} and \ref{RTT:alg2} to $\mathcal{A}\in\mathbb{R}^{I_1\times I_2\times\dots\times I_N}$ with a given TT-rank $\{\mu_1,\mu_2,\dots,\mu_{N-1}\}$ with $\mu_0=\mu_N=1$ and $n=1,2,\dots,N$, then $\mathcal{B}=\mathcal{Q}_1\times_3^1\mathcal{Q}_2\times_3^1\dots
\times_3^1\mathcal{Q}_N$ is a random quantity, which implies that the relative error in (\ref{RTT:eqn2}) is also a random quantity. Hence, a key question is to estimate the values of RE. We use 1000 trials to estimate values of RE. In detail, for $k=1,2,\dots,1000$, let $\mathcal{B}_k\in\mathbb{R}^{I_1\times I_2\times \dots\times I_N}$ and ${\rm RE}_k=\|\mathcal{A}-\mathcal{B}_k\|_F/\|\mathcal{A}\|_F$. For the tensor $\mathcal{A}$ in Example \ref{RTT-exm1}, the mean and standard deviation of $\{{\rm RE}_1,{\rm RE}_2,\dots,{\rm RE}_{1000}\}$ is shown in Figure \ref{RTT-figadd}. It follows from this figure that the proposed algorithms are robust.

\begin{remark}
    In this section, we illustrate the robustness of Algorithms \ref{RTT:alg3} and \ref{RTT:alg2} via the tensor $\mathcal{A}$ in Example \ref{RTT-exm1}. In the future, to ensure the existing and the proposed randomized algorithms for TT decomposition safely, we will propose an unbiased estimator for RE and a resampling method to estimate the variance of the output of Algorithms \ref{RTT:alg3} and \ref{RTT:alg2}. For the matrix case, Lopes {\it et al.} \cite{lopes2018error} proposed a bootstrap method to compute a posteriori error estimates for randomized least-squares algorithms. Lopes {\it et al.} \cite{lopes2019bootstrap} developed a bootstrap method for directly estimating the error of an approximate matrix product. Lopes {\it et al.} \cite{lopes2020error} develops a fully data-driven bootstrap method that numerically estimates the actual error of sketched singular vectors/values. Epperly and Tropp \cite{epperly2022efficient} proposed a leave-one-out error estimator for randomized low-rank approximations and a jacknife resampling method to estimate the variance of the output of the randomized SVD.
\end{remark}

\begin{figure}[htb]
\centering
\includegraphics[width=5in, height=2in]{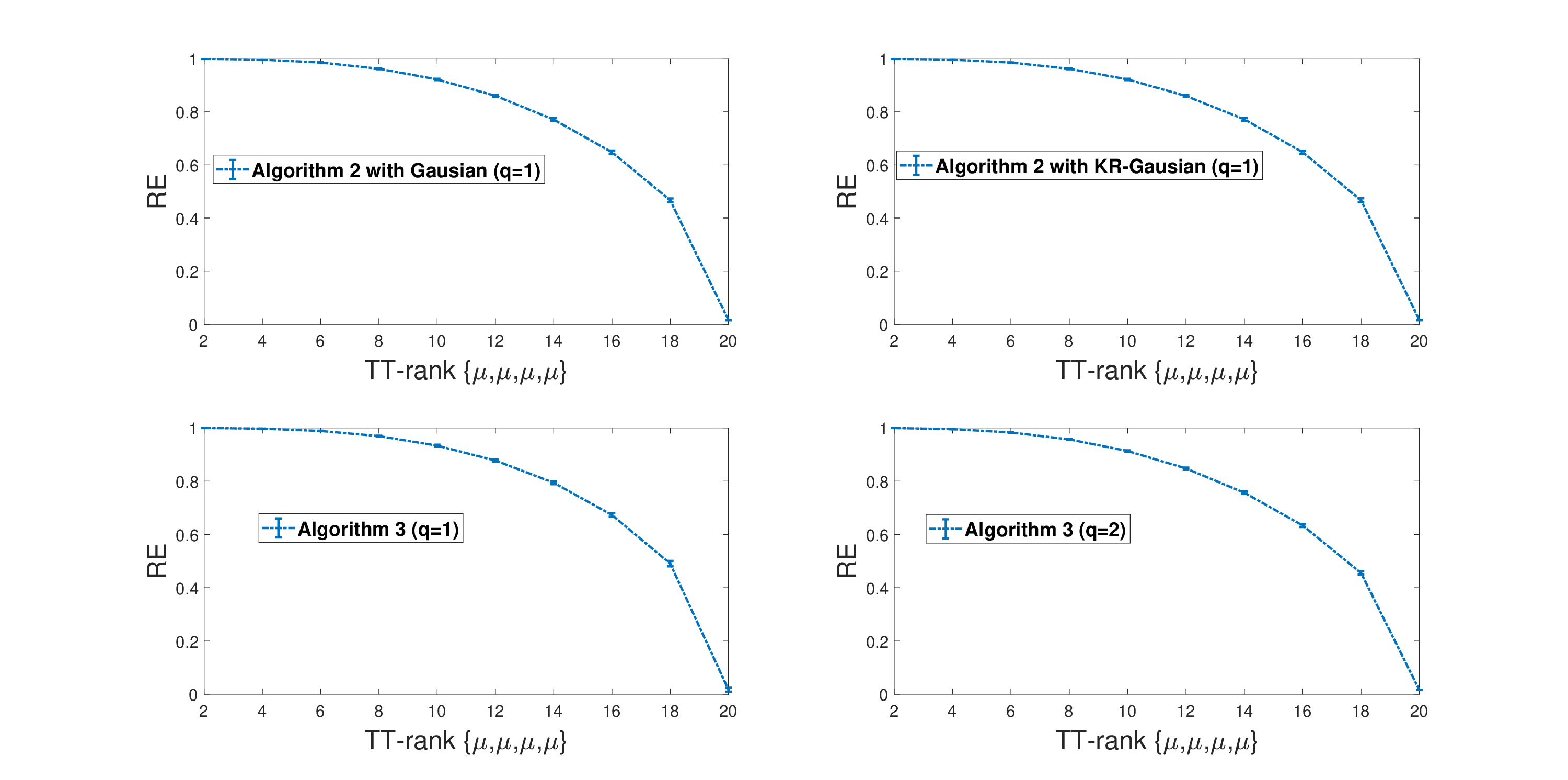}
\caption{The results by applying Algorithm \ref{RTT:alg3} with Gaussian, KR-Gaussian, and Algorithm \ref{RTT:alg2} to $\mathcal{A}$ with $\mu=2,4,6,\dots,20$. Error bars show one standard deviation.}\label{RTT-figadd}
\end{figure}

\subsection{Comparison with TT-SVD and PSTT}
\label{randomizedTT:sec6:5}
In this section, we first compare the accuracy of TT-SVD, PSTT, Algorithm \ref{RTT:alg3} with Gaussian and KR-Gaussian, and Algorithm \ref{RTT:alg2} for compressing the tensors from real-world databases (see Example \ref{RTT-exm4}). Note that in Algorithms \ref{RTT:alg3} and \ref{RTT:alg2}, the desired TT-rank is known in advance. Hence, for each tolerance $\epsilon>0$, we use TT-SVD to obtain the desired TT-rank. For this TT-rank, we then use Algorithm \ref{RTT:alg3} with Gaussian, KR-Gaussian and Kron-Gaussian, and Algorithm \ref{RTT:alg2} to find an approximation of the TT decomposition of the tensors $\mathcal{A}_{\rm YaleB-5D}$, $\mathcal{A}_{{\rm DCmall-5D}}$, and $\mathcal{A}_{{\rm COIL-5D}}$. For different $\epsilon>0$, the corresponding TT-rank of these three tensors are given in Table \ref{RTT:tab5}.

\begin{table}[htb]
\scriptsize
\centering
\begin{tabular}{|c|c|c|c|}
\hline
$\epsilon$ & $\mathcal{A}_{\rm YaleB-5D}$ & $\mathcal{A}_{{\rm DCmall-5D}}$ & $\mathcal{A}_{{\rm COIL-5D}}$\\
\hline
0.1 & (14,226,566,44) & (26,732,744,22) & (27,349,903,21) \\
\hline
0.05 & (34,989,1259,45) & (33,1102,884,25) & (56,1755,1600,24) \\
\hline
0.01 & (41,1623,1769,45) & (35,1288,1211,33) & (63,3303,1712,24) \\
\hline
0.005 & (41,1711,1951,45) & (35,1293,1331,36) & (64,3914,1726,24) \\
\hline
0.001 & (41,1721,1977,45) & (35,1295,1477,39) & (64,4060,1728,24) \\
\hline
0.0005 & (41,1722,1980,45) & (35,1295,1497,39) & (64,4091,1728,24) \\
\hline
0.0001 & (41,1722,1980,45) & (35,1295,1516,39) & (64,4096,1728,24) \\
\hline
\end{tabular}
\caption{For different $\epsilon>0$, the corresponding TT-rank is obtained by applying TT-SVD to the tensors $\mathcal{A}_{\rm YaleB-5D}$, $\mathcal{A}_{{\rm DCmall-5D}}$, and $\mathcal{A}_{{\rm COIL-5D}}$.}	
 \label{RTT:tab5}
\end{table}

\begin{figure}[htb]
\centering
\begin{tabular}{c}
\includegraphics[width=5in, height=1.8in]{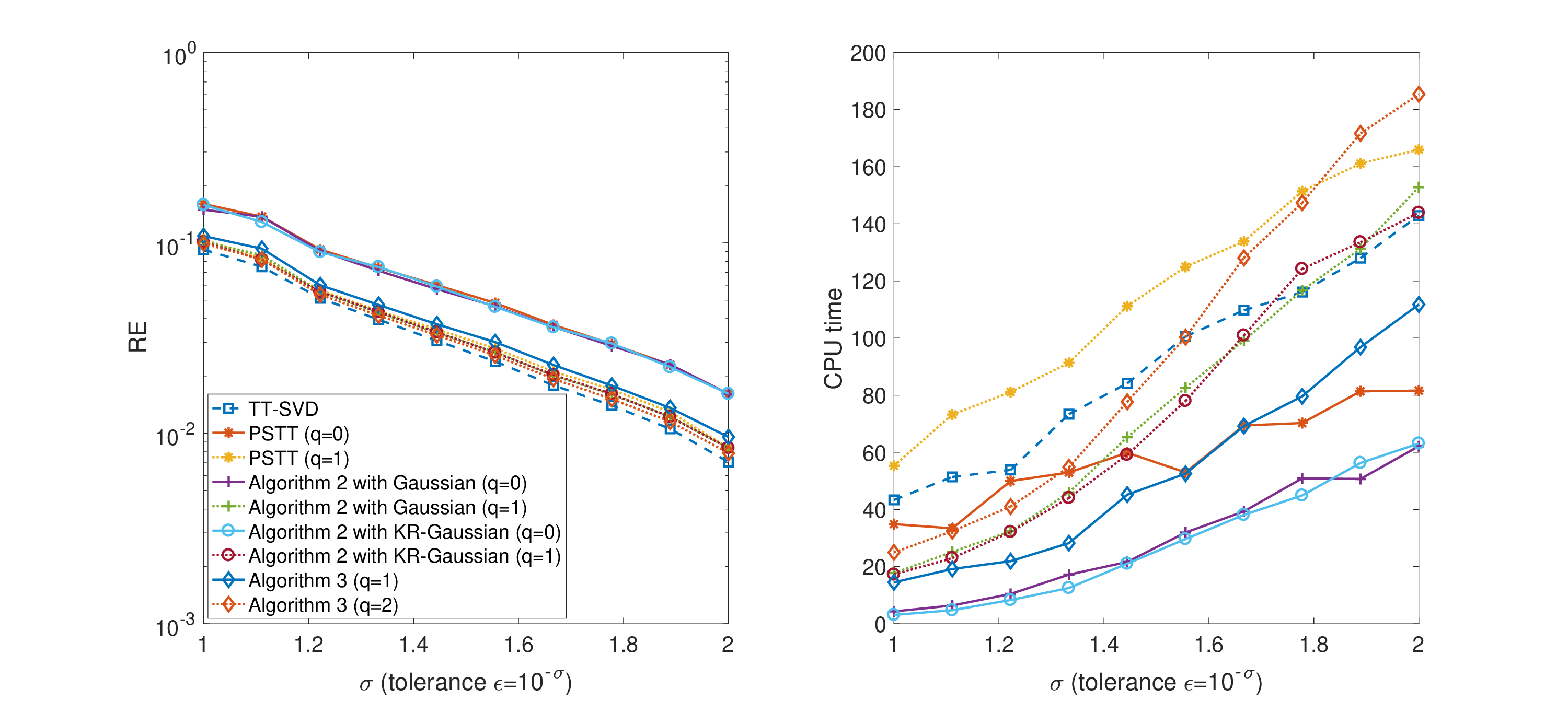}\\
\includegraphics[width=5in, height=1.8in]{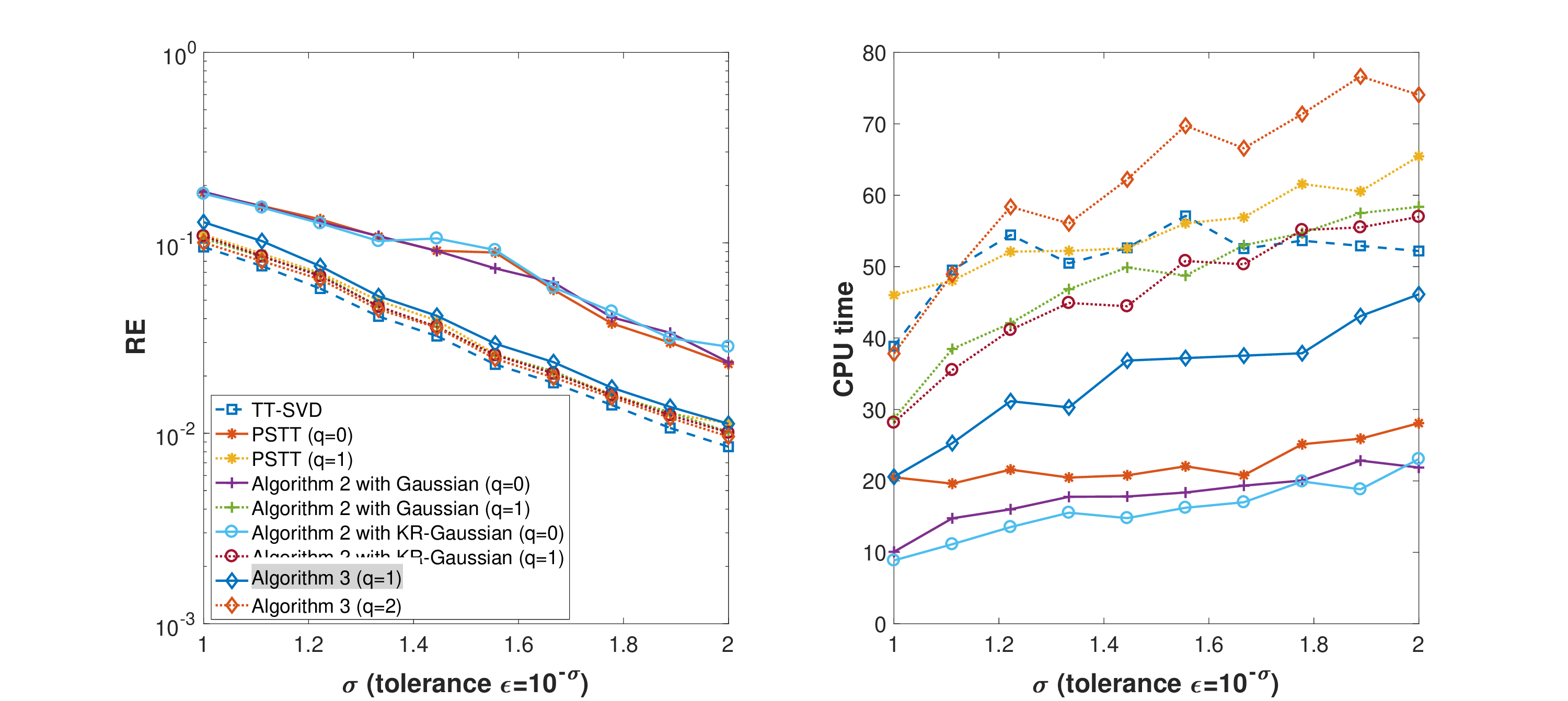}\\
\includegraphics[width=5in, height=1.8in]{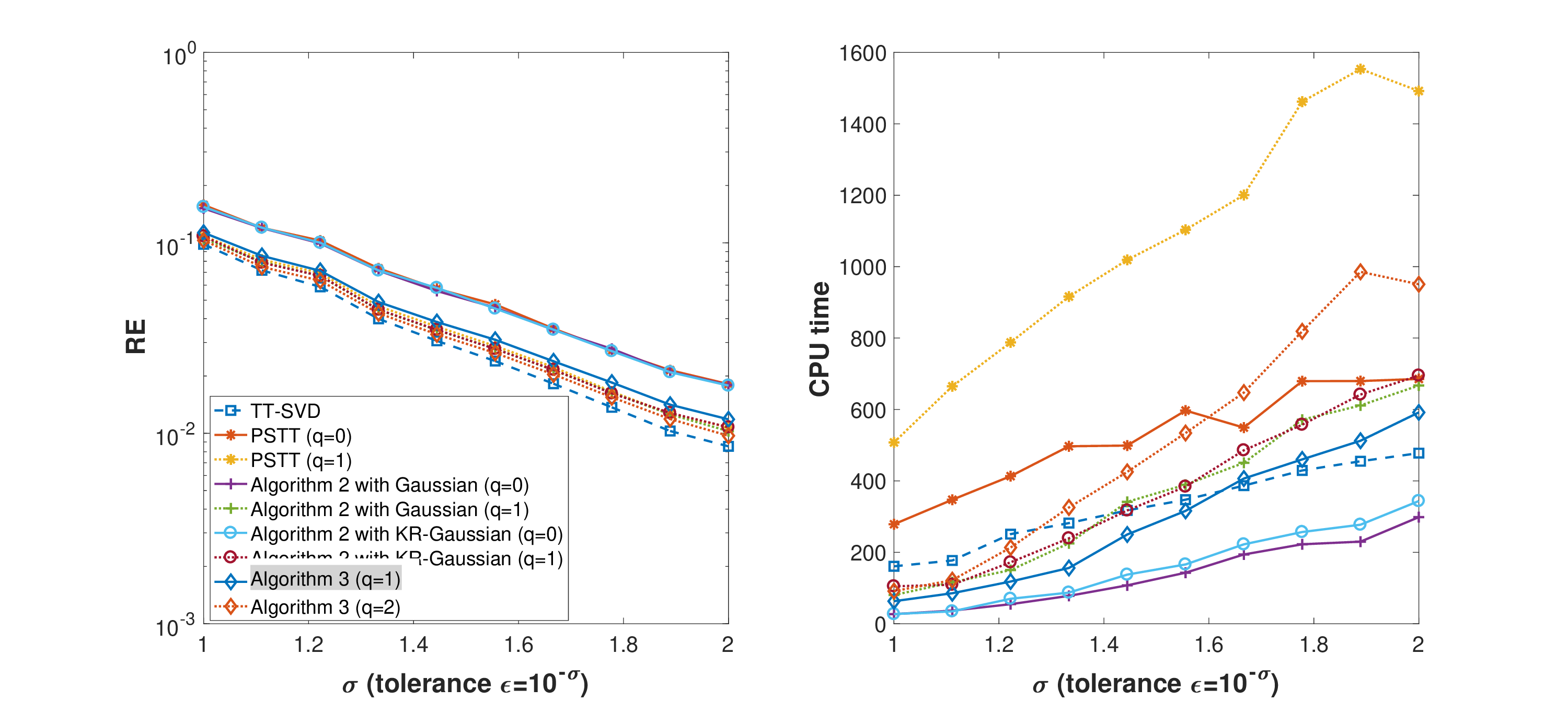}\\
\end{tabular}
\caption{For different $\epsilon>0$, the results by applying TT-SVD, PSTT, Algorithm \ref{RTT:alg3} with Gaussian and KR-Gaussian, and Algorithm \ref{RTT:alg2} to $\mathcal{A}_{\rm YaleB-5D}$ (top row), $\mathcal{A}_{{\rm DCmall-5D}}$ (middle row), and $\mathcal{A}_{{\rm COIL-5D}}$ (bottom row).}\label{RTT-fig4}
\end{figure}

From Table \ref{RTT:tab5}, we see that as $\epsilon$ decreases, each term of the corresponding TT-rank gradually increases, but the growth rate gradually decreases. In particular, as $\epsilon$ decreases, each term of the TT-rank is getting closer to the dimensions of the corresponding unfolding matrix. Hence, we set $\epsilon=10^{-\sigma}$, where $\sigma={\rm linspace}(1,2,10)$. The related results are shown in Figure \ref{RTT-fig4}, which implies that for the Gaussian, KR-Gaussian and Kron-Gaussian cases, in terms of CPU time, Algorithm \ref{RTT:alg3} ($q=0$) is faster than Algorithm \ref{RTT:alg3} ($q=1$), Algorithm \ref{RTT:alg2} ($q=1,2$), TT-SVD and PSTT ($q=1,2$), and Algorithm \ref{RTT:alg2} ($q=1$) is faster than Algorithm \ref{RTT:alg3} ($q=1$), Algorithm \ref{RTT:alg2} ($q=2$), TT-SVD and PSTT ($q=1,2$), and in terms of RE, Algorithms \ref{RTT:alg3} and \ref{RTT:alg2} ($q=1$) are comparable with and slightly worse than TT-SVD, and better than Algorithm \ref{RTT:alg3} ($q=0$).

As shown in Table \ref{RTT:tab5} and Figure \ref{RTT-fig4}, as $\epsilon$ decreases, the difference in CPU time required among TT-SVD, and Algorithms \ref{RTT:alg3} and \ref{RTT:alg2} ($q=1$) is becoming smaller and smaller.

We also consider the accuracy of TT-SVD, PSTT, Algorithm \ref{RTT:alg3} with Gaussian and KR-Gaussian, and Algorithm \ref{RTT:alg2} for the case of smaller TT-ranks. The test tensors $\mathcal{A}$ and $\mathcal{B}$ are given in Example \ref{RTT-exm1} and \ref{RTT-exm2}, respectively. For the tensor $\mathcal{A}$, its desired TT-rank is denoted by $\{\mu,\mu,\mu,\mu\}$ and the values of $\mu$ are chosen from $\{2,4,\dots,20\}$. For the tensor $\mathcal{B}$, the desired TT-rank $\{\mu,\mu,\mu,\mu\}$ is set to $\{20,20,20,20\}$ and the values of {\rm SNR} are chosen from $\{-30,-25,\dots,25,30\}$. The related results are shown in Figure \ref{RTT-fig4add1}, which illustrates that in terms of RE and FIT, Algorithm \ref{RTT:alg3} with $q=1$ and Algorithm \ref{RTT:alg2} with $q=1,2$ are comparable to TT-SVD and Algorithm \ref{RTT:alg3} with $q=0$ is worse than TT-SVD, and in terms of CPU time, Algorithm \ref{RTT:alg3} with $q=0,1$ and Algorithm \ref{RTT:alg2} with $q=1,2$ are faster than TT-SVD.

\begin{figure}[htb]
\centering
\begin{tabular}{c}
\includegraphics[width=5in, height=1.8in]{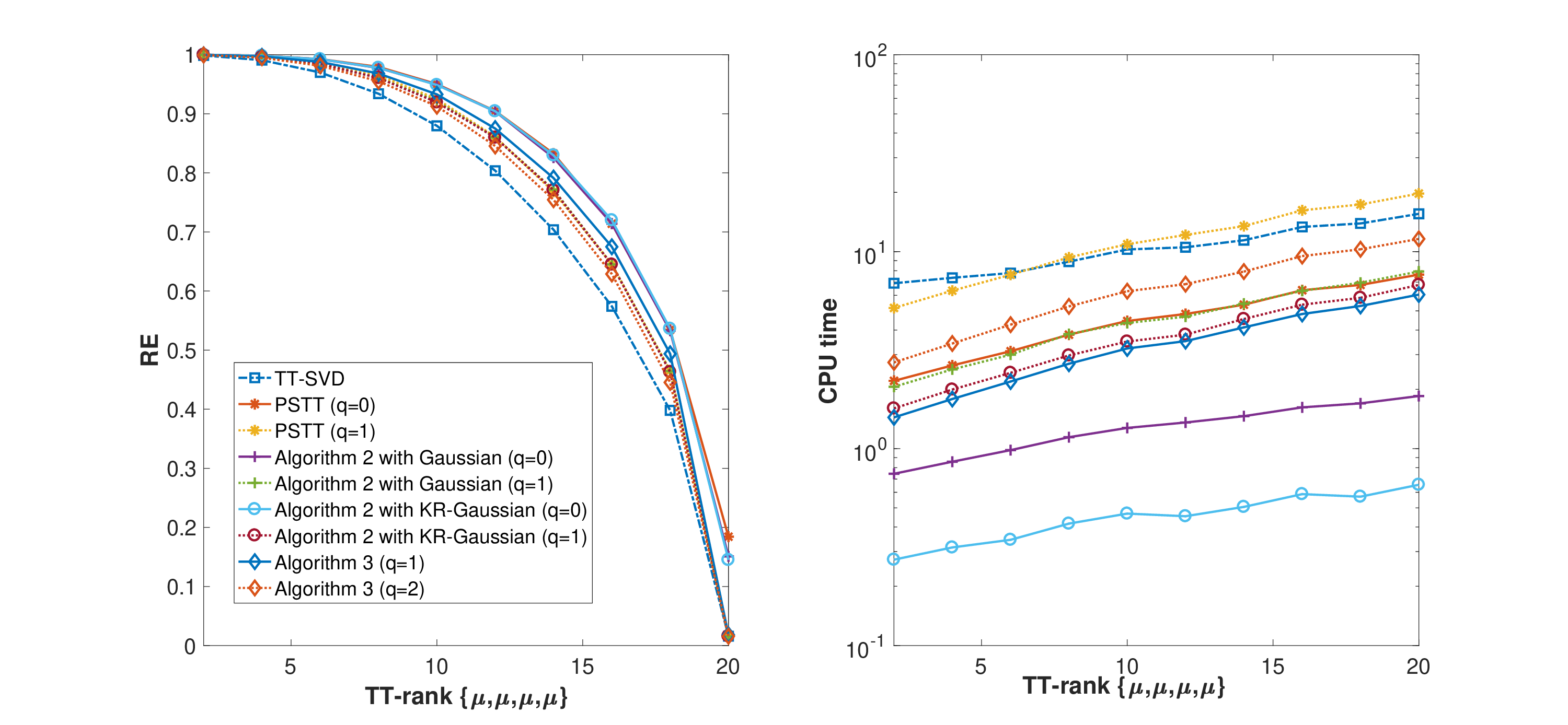}\\
\includegraphics[width=5in, height=1.8in]{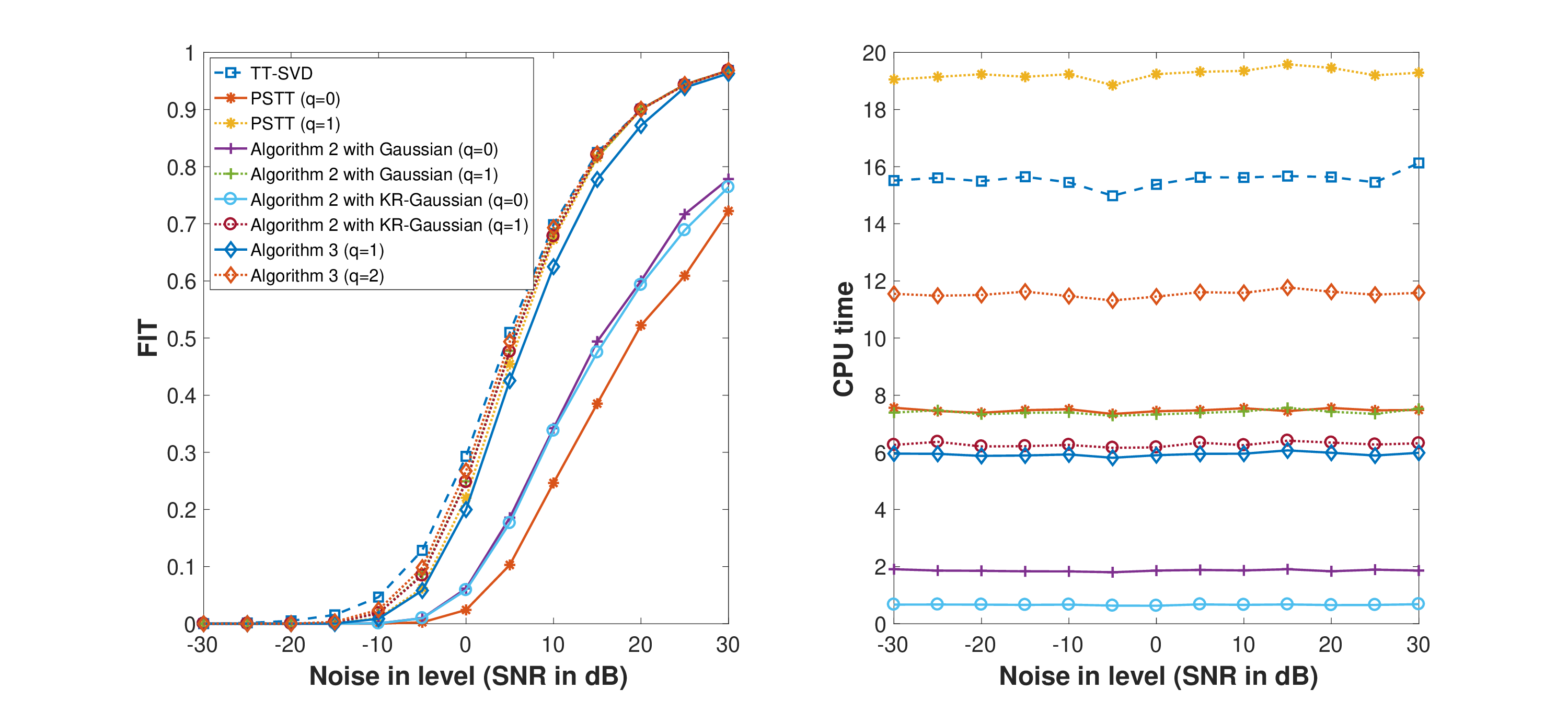}\\
\end{tabular}
\caption{The results by applying TT-SVD, PSTT, Algorithm \ref{RTT:alg3} with Gaussian and KR-Gaussian, and Algorithm \ref{RTT:alg2} to $\mathcal{A}$ (top row) with $\mu=2,4,6,\dots,20$, and $\mathcal{B}$ (bottom row) with ${\rm SNR}= -30,-25,\dots,25,30$.}\label{RTT-fig4add1}
\end{figure}

It is worth noting that T-HOSVD \cite{kolda2009tensor} and ST-HOSVD \cite{vannieuwenhoven2012new} are two common methods to compress data from real-world. For clarity, we compare the accuracy of T-HOSVD, ST-HOSVD, Algorithm \ref{RTT:alg3} with Gaussian, KR-Gaussian and Kron-Gaussian, and Algorithm \ref{RTT:alg2} for compressing the third-order tensors $\mathcal{A}_{\rm YaleB-3D}$, $\mathcal{A}_{{\rm DCmall-3D}}$, and $\mathcal{A}_{{\rm COIL-3D}}$ (see Example \ref{RTT-exm4}). For each tensor, the multilinear rank and TT-rank are denoted by $\{\mu,\mu,\mu\}$ and $\{\mu,\mu\}$, respectively, where $\mu$ is a given positive integer. For $\mathcal{A}_{\rm YaleB-3D}$, we set $\mu$ from 20 to 400 of Step 20. For $\mathcal{A}_{{\rm DCmall-3D}}$ and $\mathcal{A}_{{\rm COIL-3D}}$, we set $\mu$ from 10 to 150 in Step 10.

Figure \ref{RTT-fig5} illustrates that: in terms of CPU time, T-HOSVD is the slowest one, all the cases of Algorithm \ref{RTT:alg3} with $q=0$ are faster than other algorithms, and as $\mu$ increases, the difference between ST-HOSVD and all the cases of Algorithm \ref{RTT:alg3} with $q=0$ is becoming smaller and smaller; and in terms of RE, all the cases of Algorithm \ref{RTT:alg3} with $q=0$ is worse than other algorithms and all the cases of Algorithm \ref{RTT:alg3} with $q=1$ and Algorithm \ref{RTT:alg2} with $q=1$ are comparable to T-HOSVD and ST-HOSVD.

\begin{figure}[hth]
\centering
\begin{tabular}{c}
\includegraphics[width=5in, height=1.8in]{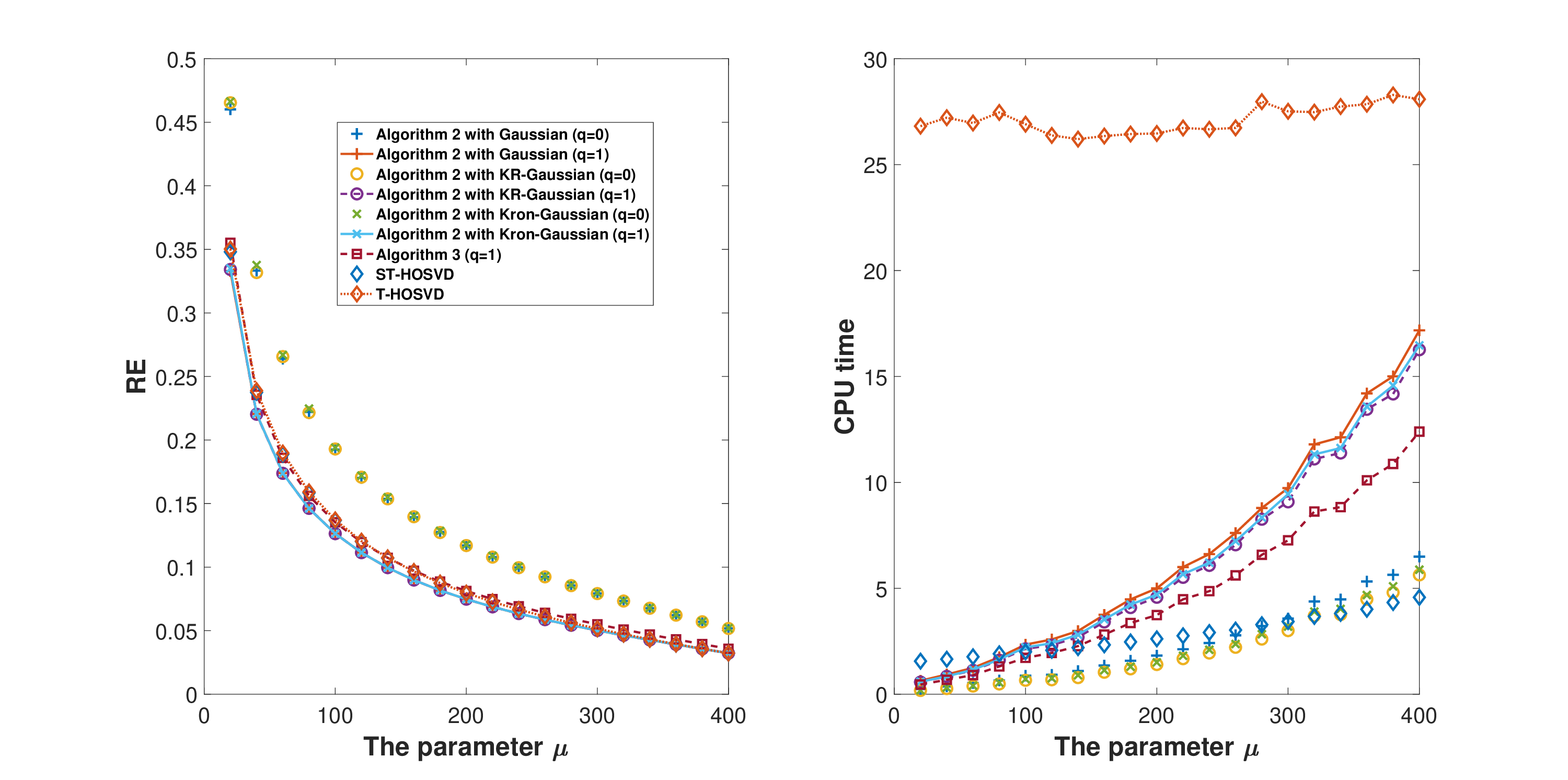}\\
\includegraphics[width=5in, height=1.8in]{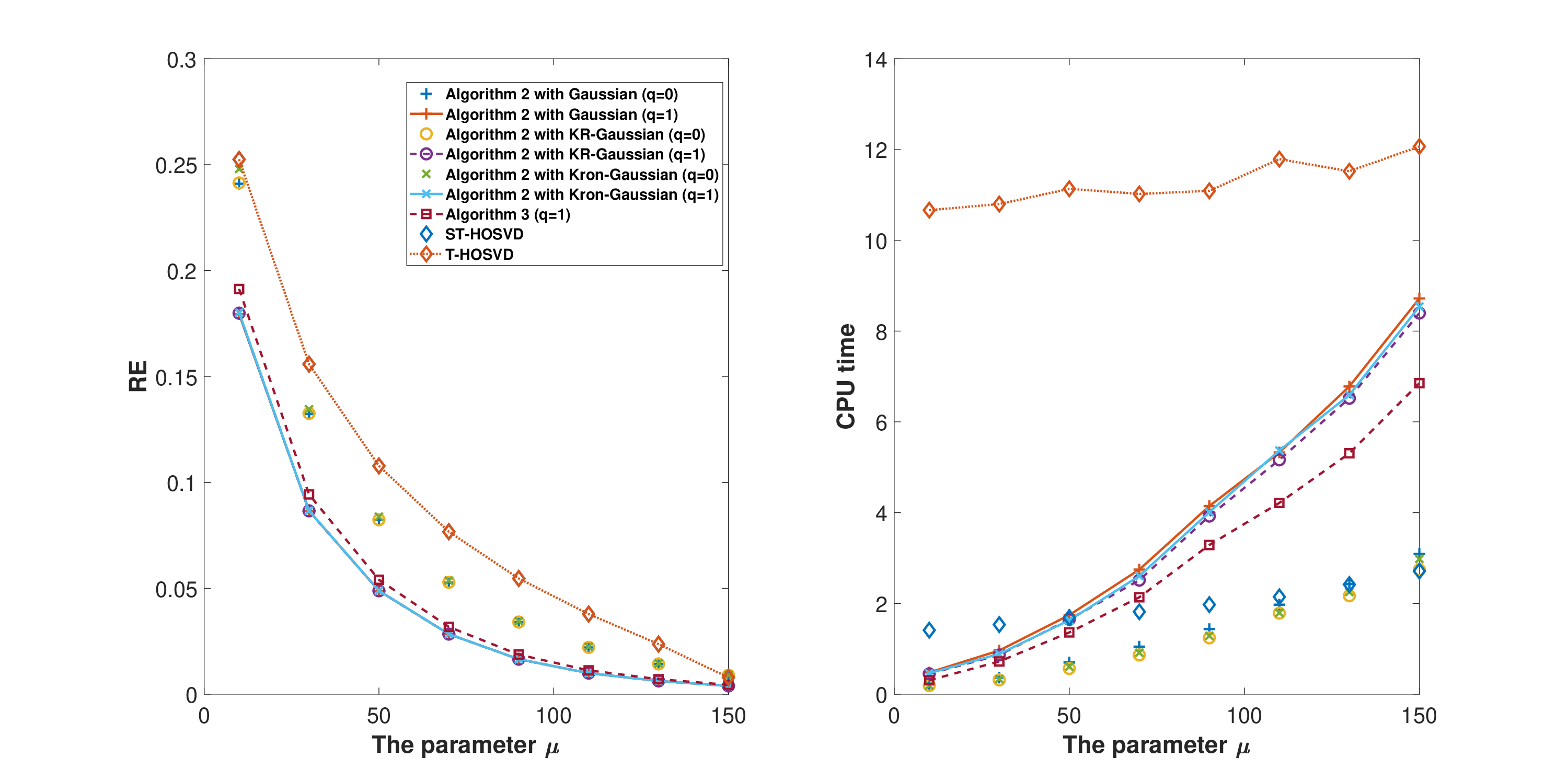}\\
\includegraphics[width=5in, height=1.8in]{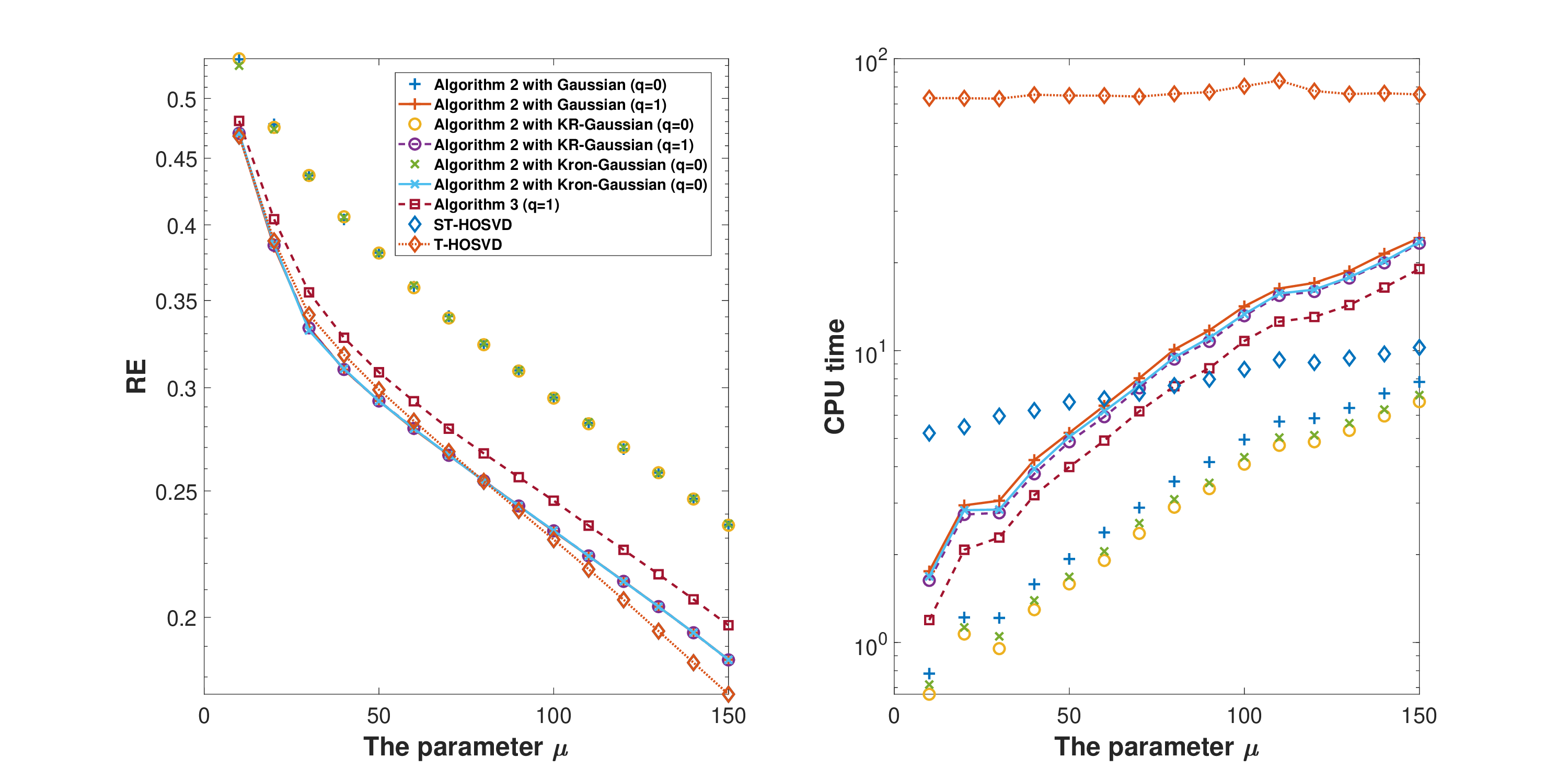}\\
\end{tabular}
\caption{For different $\mu$, the results by applying T-HOSVD, ST-HOSVD, Algorithm \ref{RTT:alg3} with Gaussian, KR-Gaussian and Kron-Gaussian, and Algorithm \ref{RTT:alg2} to $\mathcal{A}_{\rm YaleB-3D}$ (top row), $\mathcal{A}_{{\rm DCmall-3D}}$ (middle row), and $\mathcal{A}_{{\rm COIL-3D}}$ (bottom row).}\label{RTT-fig5}
\end{figure}
\subsection{The case of the fixed-precision problem}
\label{randomizedTT:sec6:6}
For a given tolerance $0<\epsilon<1$ and $\mathcal{A}\in\mathbb{R}^{I_1\times I_2\times \dots\times I_N}$, we first compare Algorithm \ref{RTT:alggreedy-tt} with TT-cross and TT-SVD for estimating the $\epsilon$-TT-rank of $\mathcal{A}$, and we then consider the accuracy of TT-cross, TT-SVD, Adap-rand-TT with Gaussian and KR-Gaussian, and Algorithm \ref{RTT:algadapt} to compute the approximation of TT decomposition and the corresponding $\epsilon$-TT-rank of $\mathcal{A}$.

For different $\epsilon$, by applying TT-cross, TT-SVD, and Algorithm \ref{RTT:alggreedy-tt} to $\mathcal{C}$, $\mathcal{D}$, and $\mathcal{E}$ in Example \ref{RTT-exm3}, the related results are listed in Table \ref{RTT:tab2}, which illustrates that for the same $\epsilon$, each term of the $\epsilon$-TT-rank obtained by Algorithm \ref{RTT:alggreedy-tt} is between that obtained by TT-cross and TT-SVD.
\begin{table}
\scriptsize
\centering
\begin{tabular}{|c|c|ccc|}
\hline
\multirow{2}{*}{Tensors}  &  \multirow{2}{*}{$\epsilon$} & \multicolumn{3}{c|}{$\epsilon$-TT-rank}  \\
\cline{3-5}
& &  TT-cross & TT-SVD & Algorithm \ref{RTT:alggreedy-tt} \\
\hline
\multirow{4}{*}{$\mathcal{C}$} & 1e-2 & (5,5,4,4) & (2,2,2,2) & (3,3,3,3) \\
\cline{2-5}
\multirow{4}{*}{} & 1e-3 & (6,6,5,5) & (3,3,3,3) & (4,5,5,4) \\
\cline{2-5}
\multirow{4}{*}{} & 1e-4 & (8,7,7,7) & (4,5,5,4) & (6,6,6,6) \\
\cline{2-5}
\multirow{4}{*}{} & 1e-5 & (10,10,9,9) & (6,7,7,6) & (7,8,8,8) \\
\hline
\multirow{4}{*}{$\mathcal{D}$} & 1e-2 & (4,4,4,4) & (2,2,2,2) & (3,3,3,3) \\
\cline{2-5}
\multirow{4}{*}{} & 1e-3 & (5,5,5,5) & (2,3,3,2) & (4,4,4,3) \\
\cline{2-5}
\multirow{4}{*}{} & 1e-4 & (5,6,6,5) & (3,3,3,3) & (4,5,4,4) \\
\cline{2-5}
\multirow{4}{*}{} & 1e-5 & (6,6,6,6) & (4,4,4,4) & (5,5,5,5) \\
\hline
\end{tabular}
\caption{For different $\epsilon>0$, numerical results by applying TT-cross, TT-SVD and Algorithm \ref{RTT:alggreedy-tt} to estimate the $\epsilon$-TT-rank of the tensors from smooth functions.}	
 \label{RTT:tab2}
\end{table}

There exist two parameters $b$ and $q$ in Algorithm \ref{RTT:algadapt}. We now use three test tensors in Example \ref{RTT-exm3} to illustrate how the choice of $(b,q)$ affects the accuracy of Algorithm \ref{RTT:algadapt}. By setting $\epsilon=10^{-\sigma}$ and $\sigma={\rm linspace}(1,2,10)$, when we apply Algorithm \ref{RTT:algadapt} with different $(b,q)$ to these three tensors $\mathcal{C}$, $\mathcal{D}$ and $\mathcal{E}$, the values of RE and CPU time are shown in Figure \ref{RTT-fig6}. From this figure, we set $(b, q)=(10,0)$ in the rest.

\begin{figure}[htb]
\centering
\begin{tabular}{c}
\includegraphics[width=5in, height=1.8in]{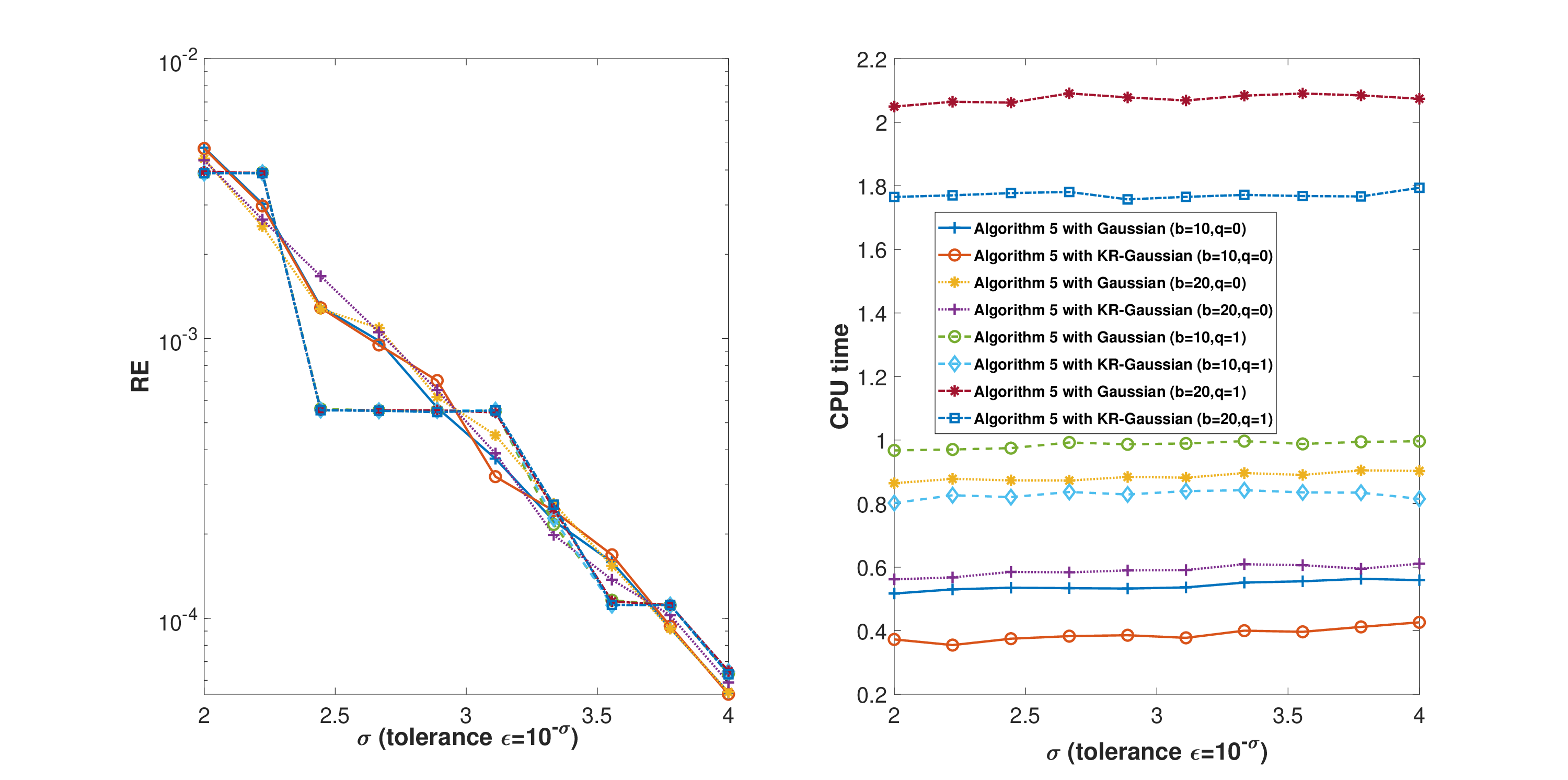}\\
\includegraphics[width=5in, height=1.8in]{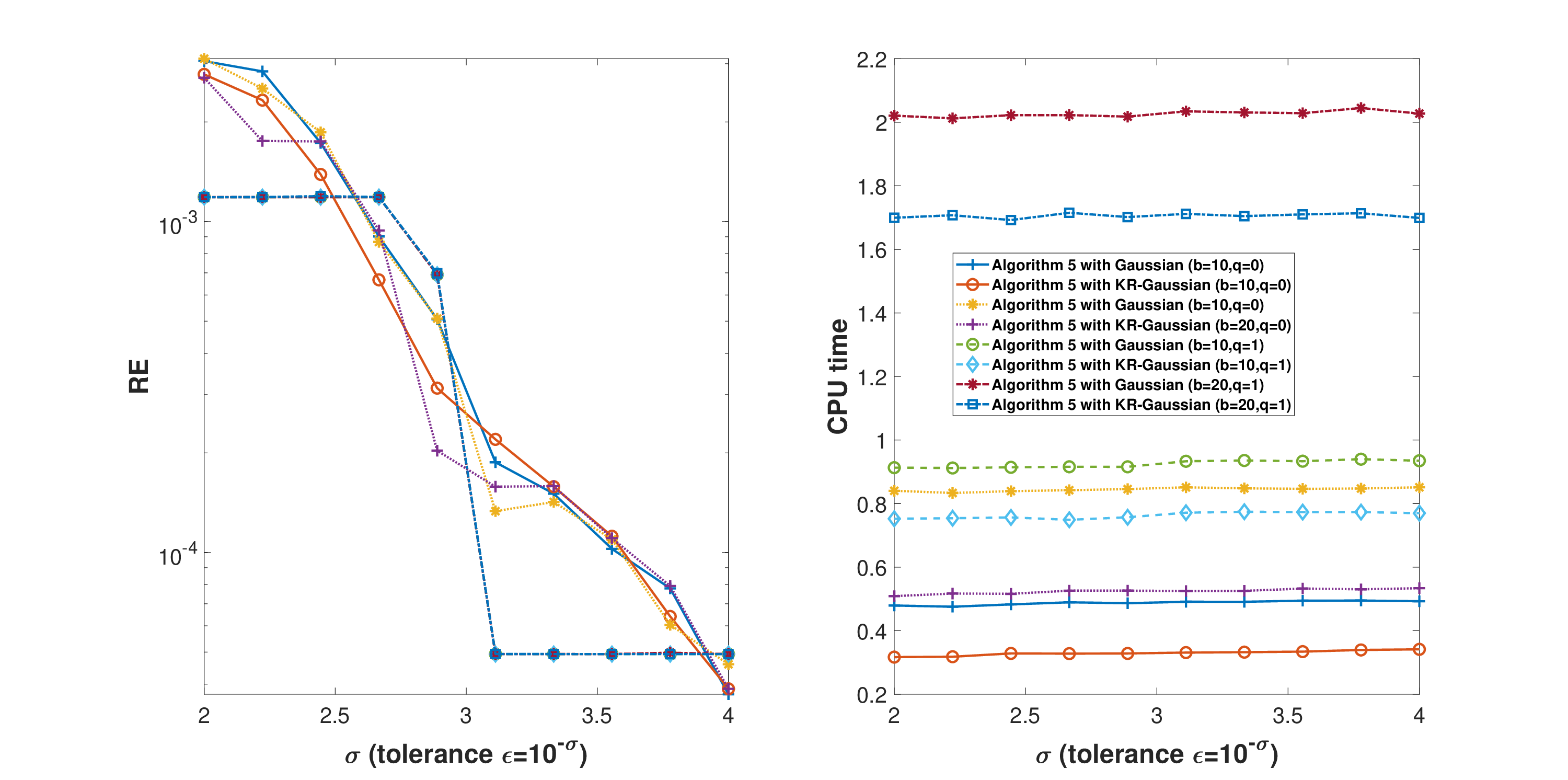}
\end{tabular}
\caption{For different $\epsilon>$, the results by applying Algorithm \ref{RTT:algadapt} with different $(b,q)$ to $\mathcal{C}$ (top row) and $\mathcal{D}$ (bottom row).}\label{RTT-fig6}
\end{figure}

For each $\epsilon$, we apply TT-SVD, Adap-rand-TT with Gaussian, Adap-rand-TT with KR-Gaussian, Algorithm \ref{RTT:algadapt} with Gaussian and Algorithm \ref{RTT:algadapt} with KR-Gaussian to the three tensors $\mathcal{A}_{\rm YaleB-5D}$, $\mathcal{A}_{{\rm DCmall-5D}}$, and $\mathcal{A}_{{\rm COIL-5D}}$, and the related results are illustrated in Table \ref{RTT:tab4}. It follows from this table that for each $\epsilon$, each term of the TT-rank obtained by Adap-rand-TT and Algorithm \ref{RTT:algadapt} is larger than that obtained by TT-SVD, and each term of the TT-rank obtained by Algorithm \ref{RTT:algadapt} is less than that obtained by Adap-rand-TT.

\begin{sidewaystable}
	\scriptsize
	\centering
	\resizebox{\textwidth}{!}{
	\begin{tabular}{|c|c|cc|cc|cc|cc|cc|cc|}
 \hline
		\multirow{2}{*}{Database}  &  \multirow{2}{*}{$\epsilon$} & \multicolumn{2}{c|}{TT-SVD} &  \multicolumn{2}{c|}{Adap-rand-TT with Gaussian} & \multicolumn{2}{c|}{Adap-rand-TT with KR-Gaussian} & \multicolumn{2}{c|}{Algorithm \ref{RTT:algadapt} with Gaussian} & \multicolumn{2}{c|}{Algorithm \ref{RTT:algadapt} with KR-Gaussian} \\
  \cline{3-12}
  & &  TT-rank & RE & TT-rank & RE & TT-rank & RE & TT-rank & RE & TT-rank & RE\\
  \hline
  \multirow{4}{*}{Yale B} & 0.5  & (2,10,30,14) & 4.6799e-01 & (5,29,91,32) & 4.1514e-01 & (6,34,112,34) & 3.9439e-01 & (3,22,74,26) & 4.8716e-01 & (3,21,73,24) & 4.8042e-01\\
  \cline{2-12}
  \multirow{4}{*}{} & 0.1 & (15,226,566,44) & 9.2185e-02 & (30,588,983,45) & 7.6231e-02 & (30,624,1067,45) & 7.0673e-02 & (26,504,950,45) & 8.4941e-02 & (26,499,949,45) & 8.6267e-02\\
  \cline{2-12}
  \multirow{4}{*}{} & 0.05 & (26,600,1003,45) & 4.3020e-02 & (40,1252,1466,45) & 3.5344e-02 & (40,1354,1557,45) & 3.0164e-02 & (34,1048,1430,45) & 4.3099e-02 & (36,1107,1434,45) & 4.2823e-02\\
  \cline{2-12}
  \multirow{4}{*}{} & 0.01 & (41,1623,1769,45) & 7.0699e-03 & (41,1712,1949,45) & 5.3246e-03 & (41,1706,1954,45) & 5.5759e-03 & (41,1697,1939,45) & 6.9719e-03 & (41,1696,1939,45) & 7.0583e-03\\
  \hline
  \multirow{4}{*}{Washington} & 0.5 & (3,18,6,1) & 4.6151e-01 & (14,130,109,11) & 4.0547e-01 & (12,143,204,12) & 3.9409e-01 & (8,52,21,2) & 4.7482e-01 & (7,53,28,2) & 4.7815e-01\\
  \cline{2-12}
  \multirow{4}{*}{} & 0.1 & (26,732,743,22) & 9.5051e-02 & (34,1129,756,29) & 7.5061e-02 & (35,1174,989,32) & 6.2834e-02 & (32,1040,938,26) & 9.1395e-02 & (31,1010,939,26) & 9.9541e-02\\
  \cline{2-12}
  \multirow{4}{*}{} & 0.05 & (33,1102,884,25) & 4.3190e-02 & (35,1261,1111,31) & 3.7452e-02 & (35,1282,1129,34) & 2.9173e-02 & (34,1217,1092,32) & 4.5669e-02 & (35,1251,1097,32) & 4.1550e-02\\
  \cline{2-12}
  \multirow{4}{*}{} & 0.01 & (35,1288,1211,33) & 8.4993e-03 & (35,1295,1426,39) & 4.4861e-03 & (35,1295,1430,38) & 5.4370e-03 & (35,1294,1413,38) & 7.7422e-03 & (35,1294,1414,39) & 6.2001e-03\\
  \hline
  \multirow{4}{*}{COIL-100} & 0.5 & (4,6,19,1) & 4.1068e-01 & (16,34,73,5) & 3.7285e-01 & (13,38,107,4) & 3.9145e-01 & (7,16,46,2) & 4.4508e-01 & (11,20,56,4) & 4.7798e-01\\
  \cline{2-12}
  \multirow{4}{*}{} & 0.1 & (27,349,903,21) & 9.8246e-02 & (45,852,1388,24) & 8.1669e-02 & (48,1000,1404,24) & 7.7937e-02 & (43,822,1346,23) & 9.6959e-02 & (47,876,1351,23) & 9.5534e-02\\
  \cline{2-12}
  \multirow{4}{*}{} & 0.05 & (46,1034,1442,24) & 4.3036e-02 & (60,2164,1653,24) & 3.8408e-02 & (62,2369,1662,24) & 3.4717e-02 & (57,2014,1638,24) & 4.2554e-02 & (58,2053,1638,24) & 4.2137e-02\\
  \cline{2-12}
  \multirow{4}{*}{} & 0.01 & (64,3303,1712,24) & 8.5419e-03 & (64,3905,1728,24) & 4.6805e-03 & (64,3914,1728,24) & 4.6090e-03 & (64,3883,1725,24) & 6.8851e-03 & (64,3886,1725,24) & 6.8672e-03\\
  \hline
\end{tabular}
 }
 \caption{For different $\epsilon>0$, numerical results by applying TT-SVD, Adap-rand-TT with Gaussian, Adap-rand-TT with KR-Gaussian, Algorithm \ref{RTT:algadapt} with Gaussian and Algorithm \ref{RTT:algadapt} with KR-Gaussian to the tensors from three real databases.}	
 \label{RTT:tab4}
\end{sidewaystable}

\section{Conclusions}
\label{randomizedTT:sec7}

For the fixed TT-rank and precision cases, we proposed several randomized algorithms for computing the approximation of the TT decomposition of a tensor. For the case of fixed TT-rank, we derived the computational complexities and error bounds for our Algorithms \ref{RTT:alg3} and \ref{RTT:alg2}. All figures in Section \ref{randomizedTT:sec6:5} illustrate that for a given TT-rank, Algorithms \ref{RTT:alg3} and \ref{RTT:alg2} with $q=1$ are productive for computing the approximation of the TT decomposition of the tensor. For a given tolerance, we derived a new algorithm to estimate the $\epsilon$-TT-rank of a tensor, and proposed an adaptive randomized algorithm for the approximation of the TT decomposition of the tensor. Furthermore, we also derived the accuracy of $\|\mathcal{A}-\mathcal{Q}_1\times_{2}^1\mathcal{Q}_2\times_{3}^1\dots\times_{3}^1\mathcal{Q}_N\|_F^2$ in floating point arithmetic. As shown in all the Tables in Section \ref{randomizedTT:sec6:6}, Algorithm \ref{RTT:algadapt} with Gaussian and KR-Gaussian is suitable for estimating the $\epsilon$-TT-rank and corresponding TT decomposition of any tensor with a given tolerance $0<\epsilon<1$.

For a given TT-rank, let $\{\mathcal{Q}_1,\mathcal{Q}_2,\dots,\mathcal{Q}_N\}$ be obtained from Algorithms \ref{RTT:alg3} and \ref{RTT:alg2}, Theorems \ref{RTT:thm13}, \ref{RTT-thm7} and \ref{RTT-thm4} state the upper bound for $\|\mathcal{A}-\mathcal{Q}_1\times_{2}^1\mathcal{Q}_2\times_{3}^1\dots\times_{3}^1\mathcal{Q}_N\|_F$. Note that we restrict $q=0$ for Theorems \ref{RTT:thm13} and \ref{RTT-thm7} and $q=1$ for Theorem \ref{RTT-thm4}. Hence, one open issue is to consider the upper bound for $\|\mathcal{A}-\mathcal{Q}_1\times_{2}^1\mathcal{Q}_2\times_{3}^1\dots\times_{3}^1\mathcal{Q}_N\|_F$ with any integer $q\geq 1$, which will be considered in the future.

\section*{Acknowledgments}
This work is supported by the Hong Kong Innovation and Technology Commission (InnoHK Project CIMDA), the Hong Kong Research Grants Council (Project 11204821),  City University of Hong Kong (Projects 9610034 and 9610460) and the National Natural Science Foundation of China under the grant 12271108.
{\small
\bibliographystyle{siam}
\bibliography{paper}

\begin{thebibliography}{10}

\bibitem{ahmadiasl2021randomized}
{\sc S.~Ahmadiasl, A.~Cichocki, A.~Phan, I.~Oseledets, S.~Abukhovich, and
  T.~Tanaka}, {\em Randomized algorithms for computation of {T}ucker
  decomposition and higher order {SVD (HOSVD)}}, IEEE Access, 9 (2021),
  pp.~28684--28706.

\bibitem{ailon2009the}
{\sc N.~Ailon and B.~Chazelle}, {\em The fast {J}ohnson-{L}indenstrauss
  transform and approximate nearest neighbors}, SIAM Journal on Computing, 39
  (2009), pp.~302--322.

\bibitem{aizenbud2016randomized}
{\sc Y.~Aizenbud, G.~Shabat, and A.~Averbuch}, {\em Randomized {LU}
  decomposition using sparse projections}, Computers and Mathematics with
  Applications, 72 (2016), pp.~2525--2534.

\bibitem{alger2020tensor}
{\sc N.~Alger, P.~Chen, and O.~Ghattas}, {\em Tensor train construction from
  tensor actions, with application to compression of large high order
  derivative tensors}, SIAM Journal on Scientific Computing, 42 (2020),
  pp.~A3516--A3539.

\bibitem{avron2010blendenpik}
{\sc H.~Avron, P.~Maymounkov, and S.~Toledo}, {\em Blendenpik: Supercharging
  lapack's least-squares solver}, SIAM Journal on Scientific Computing, 32
  (2010), pp.~1217--1236.

\bibitem{bader2017matlab}
{\sc B.~W. Bader, T.~G. Kolda, et~al.}, {\em Matlab tensor toolbox version
  3.0-dev}.
\newblock Available online, Oct. 2017.
\newblock \url{https://www.tensortoolbox.org}.

\bibitem{ballani2013black}
{\sc J.~Ballani, L.~Grasedyck, and M.~Kluge}, {\em Black box approximation of
  tensors in {H}ierarchical {T}ucker format}, Linear Algebra and its
  Applications, 438 (2013), pp.~639--657.

\bibitem{Battaglino2017a}
{\sc C.~Battaglino, G.~Ballard, and T.~G. Kolda}, {\em A practical randomized
  {CP} tensor decomposition}, SIAM Journal on Matrix Analysis and Applications,
  39 (2018), pp.~876--901.

\bibitem{bengua2017efficient}
{\sc J.~A. Bengua, H.~N. Phien, H.~D. Tuan, and M.~N. Do}, {\em Efficient
  tensor completion for color image and video recovery: Low-rank tensor train},
  IEEE Transactions on Image Processing, 26 (2017), pp.~2466--2479.

\bibitem{biagioni2015randomized}
{\sc D.~Biagioni, D.~J. Beylkin, and G.~Beylkin}, {\em Randomized interpolative
  decomposition of separated representations}, Journal of Computational
  Physics, 281 (2015), pp.~116--134.

\bibitem{bjarkason2019pass}
{\sc E.~K. Bjarkason}, {\em Pass-efficient randomized algorithms for low-rank
  matrix approximation using any number of views}, SIAM Journal on Scientific
  Computing, 41 (2019), pp.~A2355--A2383.

\bibitem{boutsidis2013improved}
{\sc C.~Boutsidis and A.~Gittens}, {\em Improved matrix algorithms via the
  subsampled randomized {H}adamard transform}, SIAM Journal on Matrix Analysis
  and Applications, 34 (2013), pp.~1301--1340.

\bibitem{caiafa2010generating}
{\sc C.~F. Caiafa and A.~Cichocki}, {\em Generalizing the column-row matrix
  decomposition to multi-way arrays}, Linear Algebra and its Applications, 433
  (2010), pp.~557--573.

\bibitem{carroll1970analysis}
{\sc J.~D. Carroll and J.~Chang}, {\em Analysis of individual differences in
  multidimensional scaling via an $n$-way generalization of
  ``{E}ckart-{Y}oung'' decomposition}, Psychometrika, 35 (1970), pp.~283--319.

\bibitem{che2019randomized}
{\sc M.~Che and Y.~Wei}, {\em Randomized algorithms for the approximations of
  tucker and the tensor train decompositions}, Advances in Computational
  Mathematics, 45 (2019), pp.~395--428.

\bibitem{che2020the}
{\sc M.~Che, Y.~Wei, and H.~Yan}, {\em The computation for low multilinear rank
  approximations of tensors via power scheme and random projection}, SIAM
  Journal on Matrix Analysis and Applications, 41 (2020), pp.~605--636.

\bibitem{che2021randomized}
\leavevmode\vrule height 2pt depth -1.6pt width 23pt, {\em Randomized
  algorithms for the low multilinear rank approximations of tensors}, Journal
  of Computational and Applied Mathematics, 390 (2021).
\newblock Article no. 113380.

\bibitem{che2023efficient}
\leavevmode\vrule height 2pt depth -1.6pt width 23pt, {\em Efficient algorithms
  for {T}ucker decomposition via approximate matrix multiplication}, arXiv
  preprint arXiv:2303.11612,  (2023).

\bibitem{chen2019support}
{\sc C.~Chen, K.~Batselier, C.-Y. Ko, and N.~Wong}, {\em A support tensor train
  machine}, in 2019 International Joint Conference on Neural Networks (IJCNN),
  IEEE, 2019, pp.~1--8.

\bibitem{chen2022kernelized}
{\sc C.~Chen, K.~Batselier, W.~Yu, and N.~Wong}, {\em Kernelized support tensor
  train machines}, Pattern Recognition, 122, Article no. 108337 (2022).

\bibitem{cichocki2009nonnegative}
{\sc A.~Cichocki, R.~Zdunek, A.~H. Phan, and S.-i. Amari}, {\em Nonnegative
  {M}atrix and {T}ensor {F}actorizations: {A}pplications to {E}xploratory
  {M}ulti-way {D}ata {A}nalysis and {B}lind {S}ource {S}eparation}, John Wiley
  \& Sons, 2009.

\bibitem{clarkson2013low}
{\sc K.~L. Clarkson and D.~P. Woodruff}, {\em Low rank approximation and
  regression in input sparsity time}, in: Proceedings of the 45th Annual ACM
  Symposium on Theory of Computing,  (2013), pp.~81--90.

\bibitem{daas2022parallel}
{\sc H.~A. Daas, G.~Ballard, and P.~Benner}, {\em Parallel algorithms for
  tensor train arithmetic}, SIAM Journal on Scientific Computing, 44 (2022),
  pp.~C25--C53.

\bibitem{dian2019learning}
{\sc R.~Dian, S.~Li, and L.~Fang}, {\em Learning a low tensor-train rank
  representation for hyperspectral image super-resolution}, IEEE Transactions
  on Neural Networks and Learning Systems, 30 (2019), pp.~2672--2683.

\bibitem{dolgov2013two-level}
{\sc S.~Dolgov and B.~Khoromskij}, {\em Two-level {QTT}-{T}ucker format for
  optimized tensor calculus}, SIAM Journal on Matrix Analysis and Applications,
   (2013), pp.~593--623.

\bibitem{dolgov2014computation}
{\sc S.~V. Dolgov, B.~N. Khoromskij, I.~V. Oseledets, and D.~V. Savostyanov},
  {\em Computation of extreme eigenvalues in higher dimensions using block
  tensor train format}, Computer Physics Communications, 185 (2014),
  pp.~1207--1216.

\bibitem{drineas2016randnla}
{\sc P.~Drineas and M.~W. Mahoney}, {\em Rand{NLA}: randomized numerical linear
  algebra}, Comm. ACM, 59 (2016), pp.~80--90.

\bibitem{ehrlacher2021adaptive}
{\sc V.~Ehrlacher, L.~Grigori, D.~Lombardi, and H.~Song}, {\em Adaptive
  hierarchical subtensor partitioning for tensor compression}, SIAM Journal on
  Scientific Computing, 43 (2021), pp.~A139--A163.

\bibitem{epperly2022efficient}
{\sc E.~N. Epperly and J.~A. Tropp}, {\em Efficient error and variance
  estimation for randomized matrix computations}, arXiv preprint
  arXiv:2207.06342,  (2022).

\bibitem{erichson2017randomized}
{\sc N.~B. Erichson, K.~Manohar, S.~L. Brunton, and J.~N. Kutz}, {\em
  Randomized {CP} tensor decomposition}, Machine Learning: Science and
  Technology,  (2020).
\newblock Article 025012.

\bibitem{grasedyck2010hierarchical}
{\sc L.~Grasedyck}, {\em Hierarchical singular value decomposition of tensors},
  SIAM Journal on Matrix Analysis and Applications, 31 (2010), pp.~2029--2054.

\bibitem{grasedyck2013a}
{\sc L.~Grasedyck, D.~Kressner, and C.~Tobler}, {\em A literature survey of
  low-rank tensor approximation techniques}, GAMM-Mitteilungen, 36 (2013),
  pp.~53--78.

\bibitem{hackbusch2009a}
{\sc W.~Hackbusch and S.~K{\"u}hn}, {\em A new scheme for the tensor
  representation}, The Journal of Fourier Analysis and Applications, 15 (2009),
  pp.~706--722.

\bibitem{halko2011finding}
{\sc N.~Halko, P.~G. Martinsson, and J.~A. Tropp}, {\em Finding structure with
  randomness: probabilistic algorithms for constructing approximate matrix
  decompositions}, SIAM Review, 53 (2011), pp.~217--288.

\bibitem{hochstenbach2010subspace}
{\sc M.~E. Hochstenbach and L.~Reichel}, {\em Subspace-restricted singular
  value decompositions for linear discrete ill-posed problems}, Journal of
  Computational and Applied Mathematics, 235 (2010), pp.~1053--1064.

\bibitem{holtz2012the}
{\sc S.~Holtz, T.~Rohwedder, and R.~Schneider}, {\em The alternating linear
  scheme for tensor optimization in the tensor train format}, SIAM Journal on
  Scientific Computing, 34 (2012), pp.~A683--A713.

\bibitem{huber2017a}
{\sc B.~Huber, R.~Schneider, and S.~Wolf}, {\em A randomized tensor train
  singular value decomposition}, in Compressed Sensing and its Applications:
  Second International MATHEON Conference 2015, Springer, 2017, pp.~261--290.

\bibitem{kazeer2014low}
{\sc V.~Kazeer, O.~Reichmann, and C.~Schwab}, {\em Low-rank tensor structure of
  linear diffusion operators in the {TT} and {QTT} formats}, Linear Algebra and
  its Applications, 438 (2013), pp.~4204--4221.

\bibitem{khoromskij2011o(dlogN)}
{\sc B.~N. Khoromskij}, {\em ${O}(d\log n)$-quantics approximation of n-d
  tensors in high-dimensional numerical modeling}, Constructive Approximation,
  34 (2011), pp.~257--280.

\bibitem{kolda2009tensor}
{\sc T.~G. Kolda and B.~W. Bader}, {\em Tensor decompositions and
  applications}, SIAM Review, 51 (2009), pp.~455--500.

\bibitem{kour2023efficient}
{\sc K.~Kour, S.~Dolgov, M.~Stoll, and P.~Benner}, {\em Efficient
  structure-preserving support tensor train machine}, Journal of Machine
  Learning Research, 24 (2023), pp.~1--22.

\bibitem{kressner2017recompression}
{\sc D.~Kressner and L.~Perisa}, {\em Recompression of {H}adamard products of
  tensors in {T}ucker format}, SIAM Journal on Scientific Computing, 39 (2017),
  pp.~A1879--A1902.

\bibitem{kressner2014low-rank}
{\sc D.~Kressner, M.~Steinlechner, and A.~Uschmajew}, {\em Low-rank tensor
  methods with subspace correction for symmetric eigenvalue problems}, SIAM
  Journal on Scientific Computing, 36 (2014), pp.~A2346--A2368.

\bibitem{kressner2014low}
{\sc D.~Kressner, M.~Steinlechner, and B.~Vandereycken}, {\em Low-rank tensor
  completion by {R}iemannian optimization}, BIT Numerical Mathematics, 54
  (2014), pp.~447--468.

\bibitem{lee2015estimating}
{\sc N.~Lee and A.~Cichocki}, {\em Estimating a few extreme singular values and
  vectors for large-scale matrices in tensor train format}, SIAM Journal on
  Matrix Analysis and Applications, 36 (2015), pp.~994--1014.

\bibitem{lestandi2021numerical}
{\sc L.~Lestandi}, {\em Numerical study of low rank approximation methods for
  mechanics data and analysis}, Journal of Scientific Computing, 87 (2021).
\newblock Article no. 14.

\bibitem{li2022faster}
{\sc L.~Li, W.~Yu, and K.~Batselier}, {\em Faster tensor train decomposition
  for sparse data}, Journal of Computational and Applied Mathematics, 405,
  Article no. 113972 (2022).

\bibitem{litvak2005smallest}
{\sc A.~E. Litvak, A.~Pajor, M.~Rudelson, and N.~Tomczakjaegermann}, {\em
  Smallest singular value of random matrices and geometry of random polytopes},
  Advances in Mathematics, 195 (2005), pp.~491--523.

\bibitem{lopes2020error}
{\sc M.~Lopes, N.~B. Erichson, and M.~Mahoney}, {\em Error estimation for
  sketched svd via the bootstrap}, in International Conference on Machine
  Learning, PMLR, 2020, pp.~6382--6392.

\bibitem{lopes2018error}
{\sc M.~Lopes, S.~Wang, and M.~Mahoney}, {\em Error estimation for randomized
  least-squares algorithms via the bootstrap}, in International Conference on
  Machine Learning, PMLR, 2018, pp.~3217--3226.

\bibitem{lopes2019bootstrap}
{\sc M.~E. Lopes, S.~Wang, and M.~W. Mahoney}, {\em A bootstrap method for
  error estimation in randomized matrix multiplication}, The Journal of Machine
  Learning Research, 20 (2019), pp.~1434--1473.

\bibitem{mahoney2011randomized}
{\sc M.~W. Mahoney}, {\em Randomized algorithms for matrices and data},
  Foundations and Trends in Machine Learning, 3 (2011), pp.~123--224.

\bibitem{malik2019fast}
{\sc O.~A. Malik and S.~Becker}, {\em Fast randomized matrix and tensor
  interpolative decomposition using {C}ount{S}ketch}, Advances in Computational
  Mathematics, 46 (2020).
\newblock Article 76.

\bibitem{martinsson2016randomized}
{\sc P.-G. Martinsson and S.~Voronin}, {\em A randomized blocked algorithm for
  efficiently computing rank-revealing factorizations of matrices}, SIAM
  Journal on Scientific Computing, 38 (2016), pp.~S485--S507.

\bibitem{minster2020randomized}
{\sc R.~Minster, A.~K. Saibaba, and M.~E. Kilmer}, {\em Randomized algorithms
  for low-rank tensor decompositions in the {T}ucker format}, SIAM Journal on
  Mathematics of Data Science, 2 (2020), pp.~189--215.

\bibitem{nene1996columbia}
{\sc S.~A. Nene, S.~K. Nayar, and H.~Murase}, {\em Columbia object image
  library ({COIL}-100)}, Columbia University, New York, NY, USA, Tech. Rep.,
  (1996).

\bibitem{oseledets2011tensor}
{\sc I.~V. Oseledets}, {\em Tensor-train decomposition}, SIAM Journal on
  Scientific Computing, 33 (2011), pp.~2295--2317.

\bibitem{oseledets2012tt}
\leavevmode\vrule height 2pt depth -1.6pt width 23pt, {\em {TT}-toolbox 2.2}.
\newblock Available online, 2012.
\newblock \url{http://github.com/oseledets/TT-toolbox}.

\bibitem{oseledets2012solution}
{\sc I.~V. Oseledets and S.~V. Dolgov}, {\em Solution of linear systems and
  matrix inversion in the {TT}-format}, SIAM Journal on Scientific Computing,
  34 (2012), pp.~A2718--A2739.

\bibitem{oseledets2010tt}
{\sc I.~V. Oseledets and E.~E. Tyrtyshnikov}, {\em T{T}-cross approximation for
  multidimensional arrays}, Linear Algebra and its Applications, 432 (2010),
  pp.~70--88.

\bibitem{reynolds2016randomized}
{\sc M.~J. Reynolds, A.~Doostan, and G.~Beylkin}, {\em Randomized alternating
  least squares for canonical tensor decompositions: Application to a {PDE}
  with random data}, SIAM Journal on Scientific Computing, 38 (2016),
  pp.~A2634--A2664.

\bibitem{rohrig2022performance}
{\sc M.~R\"{o}hrig-Z\"{o}llner, J.~Thies, and A.~Basermann}, {\em Performance
  of the low-rank {TT-SVD} for large dense tensors on modern multicore {CPU}s},
  SIAM Journal on Scientific Computing, 44 (2022), pp.~C287--C309.

\bibitem{rokhlin2010randomized}
{\sc V.~Rokhlin, A.~Szlam, and M.~Tygert}, {\em A randomized algorithm for
  principal component analysis}, SIAM Journal on Matrix Analysis and
  Applications, 31 (2010), pp.~1100--1124.

\bibitem{rudelson2009smallest}
{\sc M.~Rudelson and R.~Vershynin}, {\em Smallest singular value of a random
  rectangular matrix}, Communications on Pure and Applied Mathematics, 62
  (2009), pp.~1707--1739.

\bibitem{savostyanov2011fast}
{\sc D.~Savostyanov and I.~V. Oseledets}, {\em Fast adaptive interpolation of
  multi-dimensional arrays in tensor train format}, in the 2011 International
  Workshop on Multidimensional (nD) Systems, 2011, pp.~1--8.

\bibitem{shabat2016randomized}
{\sc G.~Shabat, Y.~Shmueli, Y.~Aizenbud, and A.~Averbuch}, {\em Randomized {LU}
  decomposition}, Appl. Comput. Harm. Anal., 44 (2016), pp.~246--272.

\bibitem{shi2023parallel}
{\sc T.~Shi, M.~Ruth, and A.~Townsend}, {\em Parallel algorithms for computing
  the tensor-train decomposition}, SIAM Journal on Scientific Computing, 45
  (2023), pp.~C101--C130.

\bibitem{sun2020low-rank}
{\sc Y.~Sun, Y.~Guo, C.~Luo, J.~A. Tropp, and M.~Udell}, {\em Low-rank {T}ucker
  approximation of a tensor from streaming data}, SIAM Journal on Mathematics
  of Data Science, 2 (2020), pp.~1123--1150.

\bibitem{tucker1966some}
{\sc L.~R. Tucker}, {\em Some mathematical notes on three-mode factor
  analysis}, Psychometrika, 31 (1966), pp.~279--311.

\bibitem{vannieuwenhoven2012new}
{\sc N.~Vannieuwenhoven, R.~Vandebril, and K.~Meerbergen}, {\em A new
  truncation strategy for the higher-order singular value decomposition}, SIAM
  Journal on Scientific Computing, 34 (2012), pp.~A1027--A1052.

\bibitem{vershynin2012introduction}
{\sc R.~Vershynin}, {\em Introduction to the non-asymptotic analysis of random
  matrices}, in Compressed Sensing: Theory and Practice, Y.~C. Eldar and
  G.~Kutyniok, eds., Cambridge University Press, 2012, pp.~210--268.

\bibitem{vervliet2016a}
{\sc N.~Vervliet and L.~De~Lathauwer}, {\em A randomized block sampling
  approach to canonical polyadic decomposition of large-scale tensors}, IEEE
  Journal of Selected Topics in Signal Processing, 10 (2016), pp.~284--295.

\bibitem{wang2019distributed}
{\sc X.~Wang, L.~T. Yang, Y.~Wang, X.~Liu, Q.~Zhang, and M.~J. Deen}, {\em A
  distributed tensor-train decomposition method for cyber-physical-social
  services}, ACM Transactions on Cyber-Physical Systems, 3 (2019), pp.~1--15.

\bibitem{wang2020adtt}
{\sc X.~Wang, L.~T. Yang, Y.~Wang, L.~Ren, and M.~J. Deen}, {\em {ADTT}: A
  highly efficient distributed tensor-train decomposition method for {II}o{T}
  big data}, IEEE Transactions on Industrial Informatics, 17 (2020),
  pp.~1573--1582.

\bibitem{woodruff2014sketching}
{\sc D.~P. Woodruff}, {\em Sketching as a tool for numerical linear algebra},
  Foundations and Trends in Theoretical Computer Science, 10 (2014),
  pp.~1--157.

\bibitem{woolfe2008a}
{\sc F.~Woolfe, E.~Liberty, V.~Rokhlin, and M.~Tygert}, {\em A fast randomized
  algorithm for the approximation of matrices}, Applied and Computational
  Harmonic Analysis, 25 (2008), pp.~335--366.

\bibitem{yin2021tt}
{\sc C.~Yin, B.~Acun, C.-J. Wu, and X.~Liu}, {\em {TT}-rec: Tensor train
  compression for deep learning recommendation models}, Proceedings of Machine
  Learning and Systems, 3 (2021), pp.~448--462.

\bibitem{yu2018efficient}
{\sc W.~Yu, Y.~Gu, and Y.~Li}, {\em Efficient randomized algorithms for the
  fixed-precision low-rank matrix approximation}, SIAM Journal on Matrix
  Analysis and Applications, 39 (2018), pp.~1339--1359.

\bibitem{zhou2014decomposition}
{\sc G.~Zhou, A.~Cichocki, and S.~Xie}, {\em Decomposition of big tensors with
  low multilinear rank}, arXiv preprint arXiv:1412.1885v1,  (2014).

\end{thebibliography}
}
\end{document}